\documentclass[11pt]{amsart}
\usepackage{euscript}
\usepackage{amssymb}
\usepackage{amsmath}
\usepackage{epic}
\usepackage{graphics}
\usepackage{epsfig}
\usepackage{color}
\usepackage{enumitem}

\usepackage{amscd,euscript}
\usepackage[frame,cmtip,curve,arrow,matrix,line,graph]{xy}

\usepackage{tikz}
\usetikzlibrary{scopes,intersections}

\usepackage{upgreek}

\usepackage[colorlinks]{hyperref}

\usepackage{colonequals}

\usepackage{tikz,mathrsfs}
\usetikzlibrary{arrows,decorations.pathmorphing,decorations.pathreplacing,positioning,shapes.geometric,shapes.misc,decorations.markings,decorations.fractals,calc,patterns}

\tikzset{>=stealth',
        cvertex/.style={circle,draw=black,inner sep=1pt,outer sep=3pt},
        vertex/.style={circle,fill=black,inner sep=1pt,outer sep=3pt},
        star/.style={circle,fill=yellow,inner sep=0.75pt,outer sep=0.75pt},
        pvertex/.style={circle,inner sep=1pt,outer sep=2pt,font=\scriptsize},
        gap/.style={inner sep=0.5pt,fill=white}}
\tikzset{
W/.style={circle,draw=black,circle,fill=white,inner sep=0pt, minimum size=4pt},
B/.style={circle,draw=black!80!white,circle,fill=black!80!white,inner sep=0pt, outer sep=4pt, minimum size=3pt},
Bs/.style={circle,draw=black!80!white,circle,fill=black!80!white,inner sep=0pt, outer sep=2pt, minimum size=3pt},
BL/.style={circle,draw=blue!60!white,circle,fill=blue!60!white,inner sep=0pt, minimum size=4pt},
R/.style={circle,draw=red!60!white,circle,fill=red!60!white,inner sep=0pt, minimum size=4pt},  
G/.style={circle,draw=green!65!black,circle,fill=green!65!black,inner sep=0pt, minimum size=4pt},     
Rs/.style={circle,draw=red!60!white,circle,fill=red!60!white,inner sep=0pt, minimum size=2pt}, 
BLs/.style={circle,draw=blue!60!white,circle,fill=blue!60!white,inner sep=0pt, minimum size=2pt},
Gs/.style={circle,draw=green!65!black,circle,fill=green!65!black,inner sep=0pt, minimum size=2pt},  }
\usepackage{setspace}
\newcommand{\marginparstretch}{0.6}
\let\oldmarginpar\marginpar
\renewcommand\marginpar[1]{\-\oldmarginpar[\framebox{\setstretch{\marginparstretch}\begin{minipage}{\marginparwidth}{\raggedleft\tiny #1}\end{minipage}}]{\framebox{\setstretch{\marginparstretch}\begin{minipage}{\marginparwidth}{\raggedright\tiny #1}\end{minipage}}}}

\numberwithin{equation}{section}
\newtheoremstyle{my}{1.5em}{0.5em}{\em}{}{\sc}{.}{0.5em}{}

\newtheorem{thm}{Theorem}[section]
\newtheorem{Theorem}[thm]{Theorem}
\newtheorem*{Theorem*}{Theorem}
\newtheorem{Corollary}[thm]{Corollary}

\newtheorem*{corollary*}{Corollary}
\newtheorem{Lemma}[thm]{Lemma}

\newtheorem{prop}[thm]{Proposition}
\newtheorem{Proposition}[thm]{Proposition}

\newtheorem*{conjecture*}{Conjecture}
\newtheorem{Question}[thm]{Question}
\newtheorem*{question*}{Question}

\newtheorem*{definitions*}{Definitions}

\newtheorem*{convention*}{Convention}
\newtheorem*{conventions*}{Conventions}
\newtheorem{cor}[thm]{Corollary}
\newtheorem{lemma}[thm]{Lemma}

\newtheorem{remark}[thm]{Remark}

\theoremstyle{definition}

\newtheorem*{rem*}{Remark}
\newtheorem{Remark}[thm]{Remark}
\newtheorem*{remark*}{Remark}

\newtheorem*{remarks*}{Remarks}
\newtheorem*{example*}{Example}
\newtheorem{Example}[thm]{Example}
\newtheorem*{examples*}{Examples}

\newtheorem{Addendum}[thm]{Addendum}

\newtheorem*{exercise*}{Exercise}
\newtheorem*{bibliographical-note*}{Bibliographical note}

\DeclareMathSymbol{\shortminus}{\mathbin}{AMSa}{"39}
\setlist[enumerate]{format=\normalfont}

\newcommand{\scrB}{\EuScript{B}}
\newcommand{\scrH}{\EuScript{H}}

\newcommand{\scrQ}{\EuScript{Q}}
\newcommand{\scrA}{\EuScript{A}}

\newcommand{\scrY}{\EuScript{Y}}
\newcommand{\scrC}{\EuScript{C}}
\newcommand{\scrS}{\EuScript{S}}

\newcommand{\scrW}{\EuScript{W}}
\newcommand{\scrV}{\EuScript{V}}
\newcommand{\scrL}{\EuScript{L}}
\newcommand{\scrO}{\EuScript{O}}
\newcommand{\scrP}{\EuScript{P}}

\newcommand{\RHom}{\operatorname{RHom}}

\newcommand{\cStab}[1]{\mathrm{Stab}_{#1}^{\kern -0.5pt \circ}\kern -0.2pt}
\newcommand{\cAut}[1]{\mathrm{Aut}_{#1}^{\kern -0.5pt \circ}\kern -0.2pt}

\newcommand{\Curve}{\mathrm{C}}

\newcommand{\bW}{\mathbb{W}}

\newcommand{\bfk}{\mathbf{k}}
\newcommand{\bF}{\mathbb{F}}
\newcommand{\bR}{\mathbb{R}}
\newcommand{\bZ}{\mathbb{Z}}
\newcommand{\bQ}{\mathbb{Q}}
\newcommand{\bC}{\mathbb{C}}

\newcommand{\bP}{\mathbb{P}}

\newcommand{\Sym}{\mathrm{Sym}}

\newcommand{\id}{\mathrm{id}}

\renewcommand{\ker}{\mathrm{ker}}

\newcommand{\Hom}{\mathrm{Hom}}
\newcommand{\charac}{\mathrm{char}}
\newcommand{\End}{\mathrm{End}}

\newcommand{\Br}{\mathrm{Br}}
\newcommand{\PBr}{\mathrm{PBr}}
\newcommand{\MBr}{\mathrm{MBr}}
\newcommand{\Conf}{\mathrm{Conf}}

\newcommand{\Symp}{\mathrm{Symp}}
\newcommand{\Diff}{\mathrm{Diff}}

\newcommand{\fat}{\mathrm{fat}}

\newcommand{\scrE}{\EuScript{E}}
\newcommand{\scrF}{\EuScript{F}}

\newcommand{\scrD}{\EuScript{D}}

\newcommand{\Cech}[1]{\check{#1}}

\newcommand{\bK}{\mathbb{K}}

\numberwithin{equation}{section}
\renewcommand{\leq}{\leqslant}
\renewcommand{\geq}{\geqslant}

\newcommand{\tensor}{\otimes}

\newcommand{\lra}{\longrightarrow}
\newcommand{\blank}{-}

\newcommand{\CC}{\mathcal C}

\newcommand{\Z}{\mathbb{Z}}

\newcommand{\tr}{\operatorname{tr}}

\newcommand{\Cyc}{\operatorname{Cyc}}
\newcommand{\Tw}{\operatorname{Tw}}

\newcommand{\coh}{\operatorname{coh}}
\newcommand{\Ext}{\operatorname{Ext}}

\newcommand{\Spec}{\operatorname{Spec}}

\newcommand{\st}{\operatorname{st}}

\renewcommand{\sc}{\operatorname{sc}}

\newcommand{\Auteq}{\operatorname{Auteq}}
\newcommand{\rk}{\operatorname{rk}}
\newcommand{\con}{\operatorname{con}}
\newcommand{\per}{\operatorname{per}}

\newcommand{\cH}{\mathcal{H}}

\renewcommand{\mod}{\operatorname{mod}}

\def\Db{\mathop{\rm{D}^b}\nolimits}
\newcommand\fundgp{\uppi_{\hspace{0.5pt}1}\hspace{-0.5pt}}

\def\scrEnd{\mathop{\EuScript{E}\rm nd}\nolimits}

\title{Double bubble plumbings and two-curve flops}
\author{Ivan Smith}
\thanks{I.S.\ was partially supported by EP/N01815X/1, and M.W.\ was supported by EP/R009325/1 and EP/R034826/1.}
\address{Ivan Smith, Centre for Mathematical Sciences, University of Cambridge, Wilberforce Road, CB3 0WB, U.K.}
\email{is200@cam.ac.uk}
\author{Michael Wemyss}
\address{Michael Wemyss, School of Mathematics and Statistics, University of Glasgow, University Place, Glasgow G12 8QQ, U.K.}
\email{michael.wemyss@glasgow.ac.uk}

\begin{document}
\maketitle \thispagestyle{empty}

\begin{abstract} 
We discuss the symplectic topology of the Stein manifolds obtained by plumbing two 3-dimensional spheres along a circle.  These spaces are related, at a derived level and working in a characteristic determined by the specific geometry, to local threefolds which contain two floppable $(-1,-1)$-curves meeting at a point.   Using  contraction algebras we classify spherical objects on the B-side, and derive topological consequences including a complete description of the homology classes realised by graded exact Lagrangians.
 \end{abstract}

\parindent0em
\parskip0em

\section{Introduction}

\subsection{Context}

 Let $(Q,\bW)$ be a quiver with potential, i.e.\ a directed graph $Q$ with a cyclic word $\bW$ on the oriented edges of $Q$. Ginzburg \cite{Ginzburg} associates to the pair $(Q,\bW)$ a 3\nobreakdash-dimensional Calabi-Yau category $\CC(Q,\bW)$ over a field $\bK$.   A basic question is to find geometric models for these categories, for instance as derived categories of sheaves on local Calabi-Yau 3-folds or Fukaya categories of convex symplectic six-manifolds.  
 
 \medskip
One remarkable aspect of this question is that extremely elementary quivers already lead to, and in many respects control, rich geometry.  The famous conifold example, which corresponds to the very simplest case when the quiver has one vertex and no arrows, compares on the B-side the local geometry near a floppable $(-1,-1)$-curve in a Calabi-Yau 3-fold, with on the A-side the cotangent bundle $T^*S^3$.

\medskip
This paper treats the simplest case of a loopless quiver which admits non-trivial potentials, namely the oriented 2-cycle.  We construct geometric models, on both the A- and B-sides of the mirror, which through necessity leave the world of toric geometry and cotangent bundles.  Our main results relate our two constructions through a range of categorical equivalences and intertwinement of autoequivalence groups.   We then deploy various techniques on the B-side to establish the classification of spherical objects, one of the first in a 3-fold setting and as far as we know the first such when the underlying $A_{\infty}$-algebras are not formal, which has a number of symplectic-topological corollaries.

\subsection{Constructions}
On the A-side, our models are plumbings of two cotangent bundles $T^*S^3$ along a circle, namely `double bubble plumbings'.  Consider the Stein manifold $W_{\upeta}$ obtained by plumbing two 3-spheres $Q_0$ and $Q_1$ along an unknotted  circle $Z\subset Q_i$ with respect to an identification $\upeta\colon \upnu_{Z/Q_0} \stackrel{\sim}{\longrightarrow} \upnu_{Z/Q_1}$ (see Section \ref{Sec:plumbing} for details).  The choice of $\upeta$ forms a torsor for $\bZ$; we normalise so that $\upeta = 0\in \bZ$ defines the plumbing in which the Morse-Bott surgery of the $Q_i$ along $Z$ is diffeomorphic to $S^1\times S^2$, whilst the Lens space $L(k,1)$ is the Morse-Bott surgery when $|\upeta| = k$.  We may assume $k\geq 0$, since the spaces $W_k$ and $W_{-k}$ are symplectomorphic.  Fix a field $\bK$, and write $\scrQ_k$ for the subcategory of the $\bZ$-graded compact $\bK$-linear Fukaya category $\mathrm{D}^{\uppi}\,\scrF(W_k)$ split-generated by the core components $\{Q_0,Q_1\}$.   

\medskip
Let $\Br_3$ denote the braid group on 3 strings, and $\uprho\colon \Br_3 \to \Sym_3$ the natural homomorphism to the symmetric group.  Recall that the mixed braid group $\MBr_3$  is the stabiliser of the element $2 \in \{1,2,3\}$ under $\uprho$ (equivalently the subgroup preserving  the partition $\{1,2,3\} = \{1,3\} \cup \{2\}$). The pure braid group $\PBr_3$ is $\ker(\uprho)$. By considering parallel transport in appropriate families, there are actions of $\MBr_3$ on $W_0$ and of $\PBr_3$ on $W_k$ for $k>0$.  

\medskip
On the B-side, our models are neighbourhoods of pairs of floppable curves meeting at a point.  Set $R_0 =  \bK[u,v,x,y] / ( uv - xy^2)$ and $R_k = \bK[u,v,x,y] / (uv - xy(x^{k}+y))$ for all $k\geq 1$. It turns out in all cases that there is a specific crepant resolution $f\colon Y_k \to \Spec R_k$ with $f^{-1}(0)$ comprising a pair of $(-1,-1)$-rational curves $\Curve_1$ and $\Curve_2$ meeting at a single point.  Consider the subcategory
\[
\scrC_k\colonequals \langle \scrO_{\Curve_1}(-1),\scrO_{\Curve_2}(-1)\rangle\subset \Db(\coh Y_k).
\] 
All objects of $\scrC_k$ are necessarily supported on the exceptional locus.  Results of \cite{Bridgeland} show that $\scrC_k$ is independent of the choice of small resolution in characteristic zero, and we extend this to characteristic $p$ below.   Similarly, in characteristic zero the papers \cite{DW-twists_and_braids, Wemyss-MMP} show that $\scrC_0$ admits a $\MBr_3$ action, whereas $\scrC_k$ admits a $\PBr_3$ action for all $k>0$.  Happily, these too generalise to all characteristics.

\subsection{Categorical Results}
The following is our first main result.
\begin{Theorem}[\ref{comparison A and B}] \label{thm:main} 
With notation as above, the following hold.
\vspace{-\parskip}
\begin{enumerate}
\item\label{thm:main 1}   If $\bK = \bC$, then there are equivalences $\scrQ_0 \simeq \scrC_0$ and $\scrQ_k \simeq \scrC_{1}$ for all $k \geq 1$.
\item\label{thm:main 2}  If $p>2$ is prime and $\charac\,\bK=p$, then there is an equivalence $\scrQ_p \simeq \scrC_p$.
\end{enumerate}
\vspace{-\parskip}
\end{Theorem}

The subcategory $\scrQ_p$ split-generates $\mathrm{D}^{\uppi}\,\scrF(W_p)$ over $\bC$ if $p=1$, or over $\bK$ when $\mathrm{char}\,\bK=p>2$, so symplectically  $\scrQ_p$  is a natural object.   The comparison in Theorem~\ref{thm:main} should be viewed as an example of mirror symmetry in the non-toric case, where we need to `turn on' the characteristic of the ground field in order to access certain fat-spherical / Lens space objects in part (2).  This positive characteristic case is not simply a curiosity; we \emph{need} part (2) later in order to deduce symplectic-topological properties of $W_p$ (see \ref{faithful cor intro} and \ref{cor: homotopy intro}).

\begin{Addendum}[\ref{entwine}]\label{Addendum}
In (1), when $k=1$, the equivalence entwines the actions of the pure braid group. In (1) when $k=0$, and in (2) generally, a similar statement holds at the level of the actions on objects; we refer the reader to the discussion after \eqref{eqn:lens_twist}. 
\end{Addendum}
Writing $e$ and $f$ for the arrows in the oriented two-cycle, implicit in the above is the assertion that in characteristic zero $\scrQ_0 \simeq \scrC_0$ are geometric models for the zero potential, whilst $\scrQ_k \simeq \scrC_{1}$ are geometric models for the potential $(ef)^2$.  In characteristic $p$,   $\scrQ_p \simeq \scrC_p$ should be viewed as geometric models for a fabled potential $(ef)^{p}$. Even although such an object does not exist in a technical sense (cyclically permuting gives $p(ef)^{p}+p(fe)^{p}=0$), our models still exhibit the higher multiplications that we would expect from $(ef)^{p}$.

\medskip 
Thus, although potentials are not well adapted to the case $\charac\,\bK= p$, mirror symmetry is: our A- and B-side models match directly, regardless of the characteristic, and still exhibit the `expected' higher multiplications from the potentials framework. 

\medskip
By appealing to known results concerning $A_{\infty}$-Koszul duality \cite{Ekholm-Lekili, LiYin2, Kalck-Yang}, we then extend Theorem~\ref{thm:main} to a statement relating wrapped Fukaya categories and relative singularity categories. This should be viewed as a natural lifting of the equivalence in Theorem~\ref{thm:main} to include non-compact objects.

\begin{Corollary}[\ref{Koszul main text}] \label{cor:wrap}
When $k=1$ and $\bK=\bC$, or when $k>2$ is prime and $\mathrm{char}\,\bK=k$, then there is an equivalence of categories $\scrW(W_k;\bK) \simeq \Db(\coh Y_k)/\langle \scrO_{Y_k}\rangle$.
\end{Corollary}

Both Theorem \ref{thm:main} and Corollary~\ref{cor:wrap} can be seen as a direct if \emph{ad hoc} instance of mirror symmetry.  There are SYZ-type prescriptions for constructing mirrors of conic fibrations \cite{AAK}: as explained in Remark \ref{Rmk:SYZ}, these yield rather singular spaces to which the plumbings $W_k$  are fair approximations.

\medskip
Our attention then shifts to classifying fat-spherical objects in $\scrC_k$, over any field.   It should be compared with \cite{Ishii-Uehara}, which discusses two-dimensional $A_k$-Milnor fibres, and  \cite{Abouzaid-Smith_plumbing}, which discusses the case of the $A_2$-Milnor fibre in any dimension.   In what follows, we set $\scrS_1=\scrO_{\Curve_1}(-1)$ and $\scrS_2=\scrO_{\Curve_2}(-1)$.  

\begin{thm}[\ref{cor 2 3folds}]\label{spherical intro}
Fix $k\geq 1$, and suppose that $a\in\scrC_k$ satisfies $\Hom_{\scrC_k}( a, a[i])=0$ for all $i<0$, and further $\dim_k\Hom_{\scrC_k}( a, a)=1$.  Then the complex $ a$ is fat-spherical, and up to the action of the pure braid group, is isomorphic to either $\scrS_1,  \scrS_2, \mathrm{Cone}(\scrS_1[-1]\to \scrS_2)$, or their shifts by $[1]$.
\end{thm}

Our techniques are constructive: given a (fat) spherical object, there is an algorithm which expresses it as a pure braid image of a core component or of the surgery. 
It is however reasonable to object to the lack of symmetry in the statement of Theorem~\ref{spherical intro}, since $\mathrm{Cone}(\scrS_1[-1]\to \scrS_2)$ is preferred over $\mathrm{Cone}(\scrS_2[-1]\to \scrS_1)$, and both are fat-spherical.  However, this is just a manifestation of the rampant choice involved in determining generators for the pure braid group;  different choices lead to different presentations.  

\medskip
Our method to prove Theorem~\ref{spherical intro} iteratively determines how to `improve' an arbitrary fat-spherical object by applying mutation functors and their inverses, and our induction is in fact much simpler than the surfaces case \cite{Ishii-Uehara}.  The method of proof, which may be of independent interest, uses the representation theory of the associated contraction algebras in a crucial way; we refer the reader to \S\ref{sect:rep of contraction algebras} for more details.  A key property is that our particular contraction algebras have only finitely many indecomposable modules, something which is not true generally.

\subsection{Corollaries}
We show in Lemma~\ref{Lem:affine_flag} that $W_1$ is an affine flag 3-fold, and there is a natural action of the pure braid group on $\uppi_0\,\Symp_{ct}(W_1)$ coming from parallel transport in the universal family of flag 3-folds. Whilst we do not know if $W_k$ is affine in general,  in Lemma~\ref{Lem:pure} we use Morse-Bott-Lefschetz descriptions of the plumbings to show that for all $k\geq 1$ there is a natural representation $\uprho_k\colon  \PBr_3 \to \uppi_0\,\Symp(W_k)$ whose image is generated by (spherical and Lens space) Dehn twists in the core components and their surgery.

\medskip
Known results of faithfulness of the pure braid action on the B-side \cite{Hirano-Wemyss}, together with Addendum \ref{Addendum}, then establishes the following.

\begin{Corollary}[\ref{Cor:faithful}]\label{faithful cor intro}
If $p=1$ or $p>2$ is prime, then the natural representation $\uprho_p\colon  \PBr_3 \to \uppi_0\,\Symp_{gr}(W_p)$ is faithful.
 \end{Corollary}

\begin{Remark} \label{Rmk:faithful_to_graded}
We consider the representation to graded symplectomorphisms, which act on the $\bZ$-graded Fukaya category.  For $p=1$, one can check that the representation lifts to a faithful representation $\PBr_3 \to \uppi_0\,\Symp_{ct}(W_1)$; for $p>1$ whilst the generators of the pure braid group are represented by compactly supported maps, our methods do not show that the whole representation lifts. The centre of the pure braid group acts by a symplectomorphism which \emph{is} isotopic to the identity through non-compactly-supported maps (and which acts by a non-trivial shift on the wrapped category), cf. Lemma \ref{Lem:Morse_bott_boundary_twist}, so certainly $\PBr_3\to \uppi_0\,\Symp(W_k)$ is not injective.
\end{Remark}

Presumably it is possible to establish Corollary~\ref{faithful cor intro} more directly, using Floer cohomology computations \`a la \cite{Khovanov-Seidel, Keating}.  Of course, we  suspect that faithfulness of the action holds for all $k\geq 1$, not just for primes, but our current methods are hampered by our inability to work over $\mathbb{Z}$ instead of fields of positive characteristic.  

\medskip
On the other hand, our categorical classification of fat-spherical objects -- more generally, simple objects with no negative-degree self-$\Ext$'s --  in Theorem~\ref{spherical intro} leads to our main, purely topological corollary. 

\begin{Corollary}[\ref{cor: homotopy main}]\label{cor: homotopy intro}
Let $p=1$ or $p>2$ be prime. Let $L\subset W_p$ be a closed connected exact Lagrangian submanifold with vanishing Maslov class. Then $L$ is quasi-isomorphic to a pure braid image of one of the two core components or the Morse-Bott surgery of the core components. 
\end{Corollary} 

In particular, identifying $H_3(W_p;\bZ) \cong \bZ\oplus\bZ$ with generators the core components $Q_i$, with either choice of orientation, we obtain the homological analogue of the `nearby Lagrangian conjecture' in this case:

\begin{Corollary}[\ref{Cor:homology_class}]\label{cor:homology intro} 
Let $p=1$ or $p>2$ be prime. Let $L\subset W_p$ be a closed exact Lagrangian submanifold with vanishing Maslov class. Then $\pm [L] \in \{(1,0), (0,1), (1,\pm 1)\} \subset \bZ\oplus\bZ$.
\end{Corollary}

 As explained in Remark \ref{rmk:need_finite_char}, see also Remark \ref{rmk:char_0_needs_local_systems}, both the previous results rely essentially on our mirror symmetry statements in \emph{finite} characteristic. 

\begin{Remark} The Maslov class hypothesis ensures that $L$ indeed defines an object of the $\bZ$-graded Fukaya category $\scrF(W_p)$. We do not know whether the plumbings $W_k$ contain any exact Lagrangian with non-zero Maslov class. \end{Remark}

The Corollaries are illustrative but not exhaustive. For instance, the proof of Corollary \ref{cor 2 3folds} pins down precisely which classes are represented by spherical objects, and which by fat spherical objects.  It follows that when $p>2$ is prime, the classes $\pm (1,\pm 1) \in H_3(W_p;\bZ)$ can be represented by Lagrangian Lens spaces $L(p,1)$ but \emph{not} by Lagrangian spheres.
\medskip

In the same vein, we obtain the following.

\begin{Corollary}[\ref{Cor:no_S1xS2}] 
If $p=1$ or $p>2$ is prime, there is no exact Lagrangian $S^1\times S^2 \subset W_p$ with vanishing Maslov class. \end{Corollary}

Using the existence of a dilation in symplectic cohomology over a field of characteristic three, Ganatra and Pomerleano \cite{Ganatra-Pomerleano} show that an exact graded  $L\subset W_1$ has prime summands certain spherical space forms, or copies of $S^1\times S^2$.
\medskip

In contrast to the faithful action in Corollary~\ref{faithful cor intro}, it is more challenging to see how Corollary~\ref{cor: homotopy intro} or Corollary \ref{cor:homology intro} could be proved directly.  Our proof heavily relies on the existence of a square root of the twist functors (which exist on the B-side, see Remark~\ref{rem:PhiPhiisTwist}), and an analysis of the representation theory of various contraction algebras.  Techniques based on symplectic cohomology tend to yield primitivity results for homology classes of exact Lagrangians, rather than a complete classification.

\subsection{Future Work}
One might hope that all the results above should fit into a much larger conjectural framework of flopping contractions, and there should be a vast sea of mirror symmetry statements, each controlled by the contraction algebra that arises.  We briefly outline, in Section~\ref{realisation section}, this `realisation problem' for symplectic models of threefold flops, and we explain some obstructions and subtleties that occur when we naively add more spheres and Lens spaces (respectively, more flopping curves) to the spaces considered in this paper. 

\subsection*{Conventions}  Throughout, all fields $\bK$ are algebraically closed.  This is mostly required for the birational geometry in Section~\ref{flops section}.

\subsection*{Acknowledgements} Jonny Evans initiated our study of the symplectic topology of double bubbles and made numerous influential suggestions. The authors are also grateful to Denis Auroux, Matt Booth, Ben Davison, Tobias Ekholm, Karin Erdmann, Yank{\i} Lekili, Cheuk Yu Mak, Sibylle Schroll and Paul Seidel for helpful conversations. The authors thank the anonymous referee for their numerous detailed queries and comments.

\section{Plumbings and the two-cycle quiver}

\subsection{Plumbings\label{Sec:plumbing}}
Fix closed smooth $n$-dimensional manifolds $Q_1$ and $Q_2$,  a smooth $k$-dimensional manifold $Z$ with $0 \leq k < n$, and embeddings $\upiota_j\colon Z \hookrightarrow Q_j$ with a fixed identification
\begin{equation} \label{eqn:normal_bundles}
\upeta\colon \upnu_{Z/Q_1} \stackrel{\simeq}{\longrightarrow} \upnu_{Z/Q_2}
\end{equation}
between the normal bundles of the two embeddings.  This data determines a Stein manifold
\[
W_{\upeta} = W_{\upeta}(\upiota_1,\upiota_2) = T^*Q_1 \#_{Z,\upeta}\, T^*Q_2
\]
which is the \emph{plumbing} of the cotangent bundles of the $Q_i$ along the submanifold $Z$ with respect to the identification $\upeta$.  The plumbing is the completion of a Stein domain obtained as follows. Let $D^*Q_i$ denote the disc cotangent bundle (for some chosen Riemannian metric). Weinstein's neighbourhood theorem implies that there is an open neighbourhood $U(Z)_i$ of $D^*Z \subset D^*Q_i$ which is symplectomorphic to a disc sub-bundle of the symplectic vector bundle $\upnu_{Z/Q_i}\otimes\bC$ over $Z$.  $W_{\upeta}$ is obtained by gluing $U(Z)_1 \stackrel{\simeq}{\longrightarrow} U(Z)_2$ by the map $\upeta \otimes (\cdot \sqrt{-1})$ obtained from composing the given isomorphism of normal bundles with multiplication by $\sqrt{-1}$ in the complexification.  $W_{\upeta}$ is an exact symplectic manifold which contains both $Q_i$ as exact Lagrangian submanifolds; it retracts to the union $Q_1 \cup_Z Q_2$ (its `skeleton').  

\medskip
Conversely, if $(X,\upomega)$ is any symplectic manifold which contains a pair of  closed Lagrangian submanifolds $Q_i$ meeting cleanly along a closed submanifold $Z \subset Q_i$, then $\upomega$ induces a perfect pairing between $\upnu_{Z/Q_1}$ and $\upnu_{Z/Q_2}$, hence an isomorphism $\upeta$ of the underlying unoriented bundles. An open neighbourhood of $Q_1 \cup Q_2 \subset X$ is symplectically equivalent to a Stein subdomain of the corresponding $W_{\upeta}$.  Plumbings therefore describe the local geometry of symplectic manifolds near unions of cleanly intersecting Lagrangian submanifolds.

\begin{Example}
The $A_n$-Milnor fibre $\{x^2+y^2+z^{n+1} = 1\} \subset (\bC^3,\upomega_{\st})$ is  the iterated plumbing of $n$ copies of $T^*S^2$ along points.  Because $S^2$ admits an orientation-reversing diffeomorphism, the choice of identification of normal bundles is not important. 
\end{Example}

\begin{Example} The self-plumbing of $T^*S^2$ at two points $\{p \sqcup q\} \subset S^2$ depends on whether $T_p(S^2) \stackrel{\simeq}{\longrightarrow} T_q(S^2)$ preserves or reverses orientation. One obtains a neighbourhood of an immersed Lagrangian sphere of self-intersection $0$ respectively $-4$ in the two cases. \end{Example}

\subsection{Double-bubble plumbings}
Let $Q_1 = S^3 = Q_2$ and let $Z = S^1$.  The embedding $Z \hookrightarrow Q_i$ defines a knot,  which we will denote $\upkappa_i$ but usually assume is trivial.   The normal bundle $\upnu_{Z/Q_i} \cong S^1\times D^2$ is trivial, so the possible isomorphisms $\upeta\colon \upnu_{Z/Q_1} \to \upnu_{Z/Q_2}$ covering an orientation-preserving diffeomorphism of $Z$ form a torsor for the integers.  

\begin{Lemma}
The plumbing $W_\upeta(\upkappa_1,\upkappa_2)$ contains two exact Maslov zero Lagrangian submanifolds $K$, $K'$ obtained from the two possible Morse-Bott surgeries of the $Q_i$ along $Z$. 
\end{Lemma}
\begin{proof}
There are two Lagrange surgeries of the planes $\Gamma := \bR^2 \sqcup i\bR^2 \subset \bC^2$, which are distinct up to compactly supported Hamiltonian isotopy.  The intersection locus $Z\subset W_{\upeta}$ is an isotropic submanifold, and an open neighbourhood $U(Z) \subset W_{\upeta}$ is symplectomorphic to a disc bundle in the total space of a symplectic vector bundle $E \to T^*Z$. The bundle $E$ admits two transverse Lagrangian subbundles coming from the normal bundles $\upnu_{Z/Q_i}$, which means that $E$ is globally symplectically trivial.   The surgeries of two transverse planes in $\bC^2$ may then be performed pointwise. It is a general fact \cite[Lemma 6.2]{Mak-Wu-cobordism} that the Lagrange surgery of two exact Lagrangians along a connected clean intersection remains exact.
\medskip

We recall the construction of the surgery in one $\bC^2$-fibre.  Set
\[
Q_{\pm}^{\ast}=\{x+iy\in\bC\mid x\ge 0,\; \pm y\ge 0,\; x+iy\ne 0\}.
\]
Fix smooth embedded paths $\gamma_\pm\colon\bR\to Q_{\pm}^{\ast}$ which satisfy 
\[
\gamma_\pm(t) =
\begin{cases}
 - t &\text{for }t < -\epsilon,\\
 \pm it &\text{for }t > \epsilon.
\end{cases}
\]
The Lagrange $1$-handles  $H_{\pm}$ are defined as 
\begin{equation}\label{eq:Lhandle}
H_\pm = \bigcup_{t\in\bR} \gamma_\pm(t)\cdot S^{1} \subset \bC^{2},
\end{equation} 
where $\zeta\cdot$ denotes multiplication by the complex scalar $\zeta$ and $S^1 \subset \bC$ is the unit circle. The $H_{\pm}$ are embedded exact Lagrangian submanifolds diffeomorphic to $S^{1} \times \bR$ and co-inciding with $\Gamma$ outside of a ball of radius $2\epsilon$ around the origin. The surgery replaces $\Gamma$ by one of $H_\pm$. The fact that the Lagrange surgeries $K, K'$ have Maslov class zero follows from the fact that they admit gradings, i.e. they admit phase functions in the sense of \cite{Seidel:graded}.  Indeed, there is a unique phase function (up to constant)
\[
\phi\colon H_{\pm}\to\bR
\]
which vanishes on $H_{\pm}\cap\bR^{2}$. We claim this has a well-defined extension to a phase function which vanishes on the 1-handle; globally, this yields a phase function for $K, K'$ which vanishes identically over the Bott one-handle. To justify the claim, consider the tangent planes of the handle along the path $\gamma_{\pm}(t)\cdot v$, for $v\in S^{1}$. On the subspace of the tangent space perpendicular to $v$ these vary by multiplication by $\gamma_{\pm}(t)$ which contributes $\pm\frac{1}{2}$ to the phase function on $H_{\pm}$. The tangent vector $\dot\gamma_{\pm}(t)\cdot v$ makes a quarter rotation of sign $\mp$ on $H_{\pm}$, which contributes $\mp\frac{1}{2}$ to the phase function. Compare to \cite[Corollary 4.11]{Mak-Wu-cobordism}. 
\end{proof}

The Bott surgeries $K$ and $K'$ are not in general Hamiltonian isotopic, but they are diffeomorphic because the Lagrange handles in the local model are smoothly isotopic. 

\medskip
By exactness, the Floer cohomology groups $HF^*(Q_i,Q_j)$ are defined over $\bZ$. Since $Q_1$ and $Q_2$ meet cleanly along $Z$,  there is a spectral sequence $H^*(Z;\bZ) \Rightarrow HF^*(Q_1,Q_2)$ \cite{Pozniak, Schmaschke}.  Since $Z$ is connected, the action functional has Morse-Bott critical points along $Z$, all of which have the same action, and the spectral sequence degenerates to give  $HF(Q_1, Q_2) \cong H^*(Z;\bZ)$ integrally; the gradings match up to an overall  shift which depends on the choice of gradings of the $Q_i$.  Choose gradings on the $Q_i$ so that 
\[
HF(Q_1,Q_2) \cong H^*(S^1)[-1] \cong HF(Q_2,Q_1)
\]
are symmetrically graded in degrees $1$ and $2$.  Let $e$ and $f$ denote the degree 1 generators in $HF^1(Q_1,Q_2)$ respectively $HF^1(Q_2,Q_1)$, and $e^{\vee} \in HF^2(Q_2,Q_1)$ and  $f^{\vee} \in HF^2(Q_1,Q_2)$ denote the Poincar\'e dual degree 2 generators.   
\medskip

Since $W_k$ is a clean plumbing of cotangent bundles, there is a minimal model for the $A_{\infty}$-algebra $\oplus_{i,j} CF^*(Q_i,Q_j)$ defined over the integers: Abouzaid \cite{Abouzaid-topological} constructs a $dg$-model for this Floer algebra over $\bZ$, and since  in our case the cohomology groups of this algebra are projective $\bZ$-modules, the result follows by classical homotopy transfer.   In such a minimal model for the algebra $\oplus_{i,j} HF^*(Q_i,Q_j)$,  the simplest higher-order operations in the $A_{\infty}$-structure are Massey products
\begin{equation} \label{eqn:triple_product}
\upmu^3(e,f,e) = \uplambda.f^{\vee}; \quad \upmu^3(f,e,f) = \uplambda'.e^{\vee}, \qquad \uplambda, \uplambda' \in \bZ.
\end{equation}
The values $\uplambda, \uplambda'$ are well-defined symplectic invariants, independent of the choice of almost complex structure.

\begin{Lemma}\label{Lem:Massey}
The values $\uplambda$ and $\uplambda'$ co-incide. If $\uplambda =0 $ then $H^*(K;\bZ) \cong H^*(S^1\times S^2;\bZ)$.  If $|\uplambda| >0$, then $K$ is a homology Lens space with $H^1(K;\bZ) \cong \bZ/|\uplambda|$.  
\end{Lemma}
\begin{proof}
Fix a field $\bK$ and work in the $\bK$-linear Fukaya category.  There is an exact Lagrangian cobordism between $K$ and $Q_1 \sqcup Q_2$, as constructed and studied in \cite{Biran-Cornea} for the case of a transverse intersection and in \cite{Mak-Wu-cobordism} in the case of a clean intersection; the latter also shows that the cobordism admits a grading. By the main result of \cite{Biran-Cornea}, the existence of such an exact graded cobordism shows that the surgery $K$ is quasi-isomorphic to a mapping cone 
\begin{equation} \label{eqn:surgery_quasi}
K \, \simeq \, Q_1 \stackrel{\upkappa}{\longrightarrow} Q_2
\end{equation}
in $\scrF(W_{\upeta})$, where $\upkappa \in \Hom_{\scrF}(Q_1,Q_2) = HF(Q_1,Q_2)$. (The other surgery $K'$ is quasi-isomorphic to a mapping cone $Q_2 \stackrel{\upkappa'}{\longrightarrow} Q_1$.)  The original references on Lagrangian cobordism work over $\bZ/2$, but the underlying geometric argument which yields the equivalence \eqref{eqn:surgery_quasi} applies over any field provided the relevant Floer groups are themselves well-defined: see the summary in \cite[Section 6.2]{Sheridan-Smith-tropical} for the geometric argument, and \cite{Biran-Membrez} for another instance of the extension to general characteristic.  It thus suffices to show that there is a spin structure on the total space of the cobordism, so the moduli spaces of discs with boundary on the cobordism are themselves orientable. This follows from the description of the total space of the cobordism as a Bott handle attachment, compare to \cite[Proposition 4.1B]{Biran-Membrez} (which shows that the spin structure extends over the one-handle fibrewise in the Bott surgery). 
\medskip

The class $\upkappa$ is  a multiple $\upkappa = \pm e$ of the unique  degree one morphism corresponding to the fundamental class of the clean-intersection locus. We claim this multiple is primitive. Indeed, since the mapping cone is quasi-isomorphic to the Lagrange surgery, it is indecomposable.  The endomorphisms of the mapping cone depend on  $\upmu^1_{\Tw(\scrF)}(f)$ which has non-trivial term $\upmu^3_{\scrF}(e,f,e)$; since $K$ is exact, $\Hom_{\scrF}(K,K) = H^*(K;\bK)$.  Since this discussion applies over any field $\bK$, the desired conclusions on $\uplambda$ and $\uplambda'$ in $\bZ$ follow. Equality of $\uplambda$ and $\uplambda'$ follows from $K$ and $K'$ being diffeomorphic.  
 \end{proof}

\begin{Example}\label{Ex:Dehn}
If $\upkappa_2$ is an unknot, then $K$ and $K'$ are obtained by Dehn surgery on $\upkappa_1$ with integer slopes $\pm \uplambda$.  
\end{Example}
Via Lemma \ref{Lem:Massey}, we can fix a base-point in the $\bZ$-torsor of normal bundle identifications. (The same base-point arises if one uses a metric on $S^3$ to identify the normal and conormal bundles along $Z$.)  We will write
\[
W_n(\upkappa_1,\upkappa_2)
\]
for the plumbing in which the surgeries have $H^1(K;\bZ) = \bZ/n$, with $n=0$ corresponding to the case of a homology $S^1\times S^2$. When the knots $\upkappa_i$ are both the unknot, for instance plumbing along a linearly embedded circle, we simply write $W_n$.

\subsection{Encoding by a quiver with potential\label{Sec:encode_by_quiver}}
Let $\scrF(W_n)$ denote the Fukaya category of $W_n$ as constructed in \cite{Seidel:FCPLT}, whose objects are closed exact spin Lagrangian branes with vanishing Maslov class; this is a $\bZ$-graded $A_\infty$-category over $\bK$.  Let $\scrQ_n \subset \scrF(W_n)$ denote the subcategory generated by the two core components $Q_i$.  The categories $\scrQ_n$ and $\scrF(W_n)$ are 3-Calabi-Yau categories.  

\begin{Remark} \label{rmk:cyclic} Recall that a ($n$-dimensional weak) proper Calabi-Yau structure on an $A_{\infty}$-category $\scrC$ over a field $\bK$ is a quasi-isomorphism $\scrC_{\Delta} \to \scrC_{\Delta}^{\vee}[n]$ between the diagonal bimodule and the $n$-fold shift of the linear dual diagonal bimodule. The existence of such is equivalent to the existence of a chain map $\tr: CC_{*+n}(\scrC) \to \bK$ from the Hochschild chain complex, which is non-degenerate meaning that for all $X, Y \in \scrC$
\[
H^*(hom_{\scrC}(X,Y) \otimes H^{n-*}(hom_{\scrC}(Y,X)) \stackrel{\upmu^2}{\longrightarrow} H^n(hom_{\scrC}(X,Y)) \rightarrow HH_n(\scrC) \stackrel{\tr}{\longrightarrow} \bK
\]
is a perfect pairing. The compact Fukaya category always admits such a structure over any field, cf. \cite[Theorem 2]{Ganatra:cyclic}, see also \cite{Sheridan:formulae}. In characteristic zero, this implies by \cite[Section 10]{KontSoi-Ainfty} that there is a minimal model for $\scrC$ in which the morphism groups admit chain-level perfect pairings
\[
\langle -, - \rangle: hom_{\scrC}(X,Y) \otimes hom_{\scrC}(Y,X) \to \bK[-n]
\]
such that the $A_{\infty}$-operations satisfy that
\[
\langle \upmu_{\scrC}^d(-,\ldots,-), -\rangle
\]
is graded cyclically symmetric for every $d$. 
\end{Remark}

Using Remark \ref{rmk:cyclic}, we will see that when $\charac\,\bK=0$ the categories $\scrQ_n$ can be encoded by a quiver with potential; besides being an efficient packaging of the information, it makes it easy to compare to categories arising elsewhere.  We return to our usual convention in which we write $\Hom_{\scrC}^*(S,S')$ for $H^*(hom_{\scrC}(S,S'))$.

\medskip
Let $\CC$ be a minimal $\bK$-linear $A_\infty$  category which is 3-Calabi-Yau and has a finite set of objects $\{S_i\}$. Suppose further that
\vspace{-\parskip}
\begin{enumerate}
\item Each $S_i$ is spherical, and 
\item $\Hom_{\CC}^*(S_i, S_j)$ is supported in degrees $1$ and $2$ for $i\neq j$. 
\end{enumerate}
\vspace{-\parskip}
Then there are $\bK$-vector spaces $V_{ij}$ with
\[
\Hom_{\CC}^*(S_i,S_j)={\bK}^{\updelta_{ij}}\oplus V^*_{ij}[-1]\oplus V_{ji}[-2]\oplus \bK^{\updelta_{ij}}[-3].
\]
If the $A_{\infty}$-structure on $\CC$ is \emph{both} cyclic in the sense of Remark \ref{rmk:cyclic}, so that
\begin{equation} \label{eqn:correlator_cyclic}
\langle \upmu^n(f_n,\ldots,f_1), f_0\rangle = \langle \upmu^n(f_{n-1},\ldots, f_1, f_0),f_n\rangle,
\end{equation}
\emph{and} strictly unital, then the only possible non-trivial $A_{\infty}$-products 
\[
\upmu^{n}\colon \Hom_\CC^*(S_{j_{n-1}},S_{j_{n}})\tensor \hdots \tensor \Hom_\CC^*(S_{j_{1}},S_{j_{2}})\lra \Hom_\CC^*(S_{j_1},S_{j_{n}})[2-n]
\]
are those where all inputs $f_i$ have degree $1$.  There is a manifest non-degenerate pairing
\[\langle\blank,\blank\rangle\colon \Hom_\CC^*(S_i,S_j)\times \Hom_\CC^*(S_j,S_i)\to \bK[-3],\]
and the non-trivial $A_{\infty}$-products are encoded by the values 
 \[c_{n+1}(f_{n+1},\hdots, f_{1})= \langle f_{n+1}, \upmu^{n}(f_{n},\hdots,f_{1})\rangle,\]
hence by linear maps 
\begin{equation} \label{Eqn:CyclicVersion}
c_{n+1}\colon \Hom^1_\CC(S_{j_{n+1}},S_{j_{1}})\tensor \Hom^1_\CC(S_{j_{n}},S_{j_{n+1}}) \tensor\hdots \tensor \Hom^1_\CC(S_{j_{1}},S_{j_{2}})\lra \bK[-n-1].
\end{equation}
 We may therefore encode $\CC$ entirely by a quiver $Q$, which has a vertex for each $S_i$ and $\rk_{\bK}(V_{ij})$ arrows from vertex $i$ to vertex $j$.  The coefficients $c_{n+1}$ determine a potential on $Q$, i.e. an element $\bW\in \widehat{\bK Q}$ of the closure of the  subspace  of the path algebra of $Q$ (completed with respect to path length) spanned by cyclic paths  of length at least two.  We will write $\CC(Q,\bW)$ for the $A_{\infty}$-category associated to a quiver $Q$ with potential $\bW$;  the zero-potential $\bW=0$ defines the formal $A_{\infty}$-structure. 
 
  \begin{remark}\label{danger in char p}
  Any $A_{\infty}$-structure over a field is quasi-isomorphic to a strictly unital structure. 
 As remarked  in Remark \ref{rmk:cyclic}, cyclicity of the pairings \eqref{eqn:correlator_cyclic} can always be achieved in characteristic zero, but not in general in finite characteristic:  $QH^*(S^2;\bZ/2)$ has a non-trivial $A_{\infty}$-product $\upmu^4$ with output in the fundamental class \cite{EL}, and this violates cyclicity in any strictly unital model. 
 \end{remark}

Two potentials $\bW$ and $\bW'$ are \emph{weakly right-equivalent} if there is an automorphism $\upphi\colon \widehat{\bK Q} \rightarrow \widehat{\bK Q}$ of the completed path algebra which fixes the zero-length paths and a scalar $t\in \bK^*$ such that $\upphi(\bW)-t\bW'$ lies in the closure of the subspace generated by differences $a_1\ldots a_s - a_2\ldots a_sa_1$, where $a_1\ldots a_s$ is a cycle in $\bK Q$.   \cite[Lemma 2.9]{KY}  shows weakly right-equivalent potentials  yield quasi-equivalent $A_{\infty}$-categories $\CC(Q,\bW) \simeq \CC(Q,\bW')$. 

\medskip
Let $\mathrm{Cyc}_2$ denote the quiver which is a single oriented 2-cycle, with arrows $e$ and $f$.
\[
\begin{tikzpicture}[scale=0.8]
\node (A) at (0,0) {$\scriptstyle 1$};
\node (B) at (2,0) {$\scriptstyle 2$};
\draw[->, bend left] (A) to node[above]{$\scriptstyle e$} (B);
\draw[->, bend left] (B) to node[below]{$\scriptstyle f$} (A);
\end{tikzpicture}
\]

\begin{Lemma}\label{Lem:formal}
If $\charac\,\bK=0$, then any potential on $\mathrm{Cyc}_2$ is weakly equivalent to some $(ef)^n$.  
\end{Lemma}
\begin{proof}
The only cyclic words in $\widehat{\bC Q}$ have the form $\bW = \sum_j a_j (ef)^j$ for scalars $a_j \in \bK$. Since we are working formally, a power series $1+\sum_j b_j (ef)^j$ is invertible, so $\bW$ is weakly equivalent to $(ef)^n$ where $n = \min \{j \, | \, a_j \neq 0\}$.  
\end{proof}

\begin{Proposition}\label{Prop:they_are_quivers}
Over $\bC$, there are equivalences
\[
\scrQ_0 \simeq (\mathrm{Cyc}_2, \bW=0) \qquad \mathrm{and} \qquad  \scrQ_n \simeq (\mathrm{Cyc}_2, \bW = (ef)^2) \ \textrm{for} \ n\geq 1.
\]
\end{Proposition}
\begin{proof}
This follows from Lemmas \ref{Lem:Massey} and \ref{Lem:formal}: the former says that once $n>0$ the potential must have a non-trivial quartic term to yield the non-trivial $\upmu^3$ in the $A_{\infty}$-structure. 
\end{proof}

In characteristic zero, the potentials $(ef)^n$ with $n\geq 3$ cannot arise from any $W_{\upeta}(\upkappa_1, \upkappa_2)$, even if one allows the embedding $Z \hookrightarrow Q_i$ to be knotted. Indeed, for $n\geq 3$, since the lowest order term in the potential has degree $>4$ the 3-fold Massey product $\upmu^3(e,f,e)$ vanishes, and the surgery $K$ is an exact Lagrangian homology $S^1\times S^2$ by Lemma \ref{Lem:Massey}. The triviality of the classical $A_{\infty}$-structure on $H^*(S^1\times S^2)$ rules out potentials of higher degree.

\subsection{Finite characteristic} 
Proposition~\ref{Prop:they_are_quivers} shows that, over $\mathbb{C}$, the categories $\scrQ_n$ are all equivalent, for any $n\geq 1$.  To see the more general phenomena, for example the Lens space twist, requires us to work over a different field, and to  turn on the characteristic.  Because of this, due to Remark~\ref{danger in char p}, from this section onwards we will work directly with the higher $A_{\infty}$ products.

\medskip
Let $p>2$ be prime, and write $W_p$ for the plumbing in which the surgery of the two cores is a Lens space $L(p,1)$.  Then, since $\charac\,\bK=p$, 
\[
\upmu^3(e,f,e) = p.f^{\vee} \equiv 0.
\]
It is possible to pin down another non-trivial higher product. For $\mathbb{F}_p=\Z/p$, and $L(p,1)$ the 3-dimensional Lens space, the cohomology $H^*(L(p,1);\mathbb{F}_p) = \Lambda(y) \otimes \mathbb{F}_p[z]/\langle z^2\rangle$ with $|y|=1$ and $|z|=2$. 

\begin{lemma} \label{lem:classical}
There is a non-trivial $A_{\infty}$-product $\upmu^{p}(y,\ldots,y) = z$ in $H^*(L(p,1);\mathbb{F}_p)$.
\end{lemma} 
\begin{proof}
The infinite-dimensional Lens space $S^{\infty}/(\bZ/p)$ is also the  Eilenberg-Maclane space $K(\Z/p,1)$, so its based loop space $\Omega(K(\Z/p,1)) = \Z/p\Z$ is just a finite set up to homotopy (with an addition structure since it is an $h$-space). 
 The Eilenberg-Moore equivalence 
\[
H^*(X;\bZ) \stackrel{\sim}{\longrightarrow} \Ext_{H_{-*}(\Omega X)}(\bZ,\bZ)
\]
is induced by a $dga$ morphism $C^*(X) \to hom_{C_{-*}(\Omega X)}(\bZ,\bZ)$, and hence respects $A_{\infty}$-structures.  In the special case where $X=BG$  is a classifying space, which is all that we require,  the standard bar construction for the group $G$ gives a simplicial decomposition of $BG$ in which the chain-level cup product in the bar resolution is the Eilenberg-Maclane cup product in the simplicial structure. This gives a $dga$ equivalence
\[
C^*(B\bZ/p; \bF_p) \to hom_{\mathbb{F}_p[\bZ/p]}(\bF_p,\bF_p).
\]
Write $\bF_p[\bZ/p] = \bF_p[t]/(t-1)^p$ as an iterated extension of simple modules $\bF_p = \bF_p[t]/(t-1)$, by taking the filtration by powers of $t-1$, to compute that 
\begin{equation} \label{eqn:infinite_lens_1}
H^*(S^{\infty}/\bZ/p\,; \mathbb{F}_p) \ \cong \ \Ext^*_{\mathbb{F}_p[\Z/p]} (\mathbb{F}_p,\mathbb{F}_p) = \Lambda(y) \otimes \mathbb{F}_p[z] \qquad \textrm{with} \ |y| = 1, \, |z| = 2.
\end{equation}
Moreover, this carries a non-trivial $A_{\infty}$-product
\begin{equation} \label{eqn:infinite_lens}
\upmu^p(y,y,\ldots,y) = z
\end{equation}
which appears as the first non-trivial differential in the spectral sequence computing $\Ext^*$ from the given resolution,  so $\upmu^i = 0$ for $2<i<p$.  Compare to \cite[Example 8.2]{LPWZ2}, where in their notation $H^*(S^{\infty}/\bZ/p\,; \mathbb{F}_p) = B(p)$ with $A_{\infty}$-Koszul dual the group ring $\bF_p[x]/(x^p)$; see also \cite[Example 6.3]{LPWZ} (which  however has different degree conventions).

\medskip
Consider an $A_{\infty}$-minimal model $H^*(L(p,1);\mathbb{F}_p)$ for the cohomology of the 3-dimensional Lens space $ L(p,1)$, and a minimal $A_{\infty}$-functor $F\colon H^*(L(p,1);\mathbb{F}_p) \dashrightarrow H^*(K(\Z/p,1);\mathbb{F}_p)$ for the map which classifies the degree one generator $y$.  This map is represented by the inclusion $L(p,1) \subset K(\Z/p,1)$ for suitable cellular cochain models, which shows that it is an isomorphism on $H^{<4}$. The $A_{\infty}$-functor equations yield
\[
\upmu^k(F^1(y),\ldots,F^1(y)) + \sum_{j<k} \upmu^j(\cdots) = F^1(\upmu^k(y,\ldots,y)) + \sum_{1<j} F^j(\cdots)
\]
If $k<p$, then the $A_{\infty}$-operations $\upmu^j$ for $j\leq k$ all vanish on the target, so the left hand side vanishes identically, which inductively shows that the $A_{\infty}$-product $\upmu^k(y,\ldots,y) = 0$ on $L(p,1)$. However, for $k=p$, one obtains the identity
\[
z = \upmu^p(y,\ldots,y) = \upmu^p(F^1(y),\ldots,F^1(y)) = F^1(\upmu^p(y,\ldots,y))
\]
showing that $\upmu^p$ is non-trivial for the finite-dimensional $L(p,1)$, as well as for $K(\Z/p,1)$.
\end{proof}

\begin{cor} \label{cor:from_classical}
If $\charac\,\bK = p$, then in the category $\scrQ_p$  there is a non-trivial product $\upmu^{2p-1}(e,f,\ldots,e,f,e) = f^{\vee}$.
\end{cor}
\begin{proof} 
As in Lemma \ref{Lem:Massey}, we know from the existence of the surgery cobordism that the Lens space $K = L(p,1)$ quasi-represents the mapping cone $Q_0 \stackrel{e}{\longrightarrow} Q_1$; then its degree one cohomological generator is represented by the cocycle $f$.  The product in Lemma \ref{lem:classical} is then the twisted complex differential $\upmu^p_{\Tw}(f,\ldots,f)$, which has non-trivial contribution exactly from the term appearing in the statement.
\end{proof}

Recall from Lemma \ref{Lem:formal} that over $\bC$, every $A_{\infty}$-structure on the cohomological algebra underlying the two-cycle quiver is quasi-isomorphic to that given by the potential $(ef)^n$ for some $n$.  Over a  field $\bK$ of finite characteristic $p$, we will in mild abuse of notation write $(\Cyc_2, (ef)^p)$ for the unique quasi-isomorphism class of $A_{\infty}$-structures on the algebra underlying $\Cyc_2$ in which the only non-vanishing higher products are $\upmu^{2p-1}(e,f,\ldots,e) = \pm f^{\vee}$ and $\upmu^{2p-1}(f,e,\ldots, f) = \pm e^{\vee}$.  

\begin{cor}\label{A side char p main}
If $\charac\,\bK = p$, there is an equivalence $\scrQ_p \simeq (\Cyc_2,  (ef)^p)$. \end{cor}
\begin{proof} 
\emph{A priori} we cannot make the $A_{\infty}$-structure strictly cyclic over $\bK$, but it can be made strictly unital, since the reduced Hochschild complex is quasi-isomorphic to the full Hochschild complex over any field. For degree reasons there are then only two kinds of possible non-zero $A_{\infty}$-product: those with only degree one inputs, and products with a single degree two input, say
\[
\upmu^k(e, e^{\vee}, e, f, e, f,\ldots, e, f) = ?[Q_1] \quad \textrm{or} \quad \upmu^k(e,f,\ldots, e, e^{\vee}) = ?[Q_1]
\]
and the corresponding products with output $[Q_2]$.  The latter products have the feature that, for at least one of $e,f$, if one removes all occurences of that symbol, one is left only with copies of the other symbol and its dual (e.g. removing $f$'s from the first expression leaves all inputs being $e, e^{\vee}$). Such  $A_{\infty}$-products arise in the $A_{\infty}$-structure on the twisted complex of $Q_2 \stackrel{f}{\longrightarrow} Q_1$ or $Q_1 \stackrel{e}{\longrightarrow} Q_2$. These represent the exact Morse-Bott surgery $L(p,1)$, for which the $A_{\infty}$-structure on its endomorphism algebra recovers the classical $A_{\infty}$-structure on cohomology over any field.  Arguing as in Lemma \ref{lem:classical} and Corollary \ref{cor:from_classical}, if these $A_{\infty}$-products were non-trivial they would be visible in the group cohomology $\Ext_{\mathbb{F}_p[\Z/p]}(\mathbb{F}_p,\mathbb{F}_p)$. The vanishing of the products with degree two output follows  by Koszul duality as in \cite[Example 8.2]{LPWZ2} (or from an explicit resolution of the diagonal). Indeed, \cite{LPWZ2} shows that the only non-trivial $A_{\infty}$-products in $H^*(S^{\infty}/(\bZ/p);\mathbb{F}_p)$ are of the form
\[
\upmu^p(yz^{j_1},\ldots, yz^{j_p}) = z^{1+\sum j_i}.
\]
None of the products with $j_i > 0$ can survive to $L(p,1)$ for degree reasons. It follows that the displayed product \eqref{eqn:infinite_lens} is the only non-trivial product up to the symmetry between the degree one generators in the algebra, or geometrically the two Morse-Bott Lagrange surgeries; compare to the argument of Proposition \ref{Prop:B_side}.
\end{proof}

\begin{Remark} Lazaroiu \cite{Lazaroiu} has given a sufficient criterion for a DG-algebra with a pairing to be strictly cyclic over fields of characteristic $\neq 2$, in terms of a homotopy operator with suitable properties.  It would be interesting to establish cyclicity directly for double bubbles from Abouzaid's topological model \cite{Abouzaid-topological}.
\end{Remark}

\section{Mirrors from flops\label{Sec:flops}}

\subsection{Neighbourhoods of two-curve flops}\label{flops section}
The quivers with potential appearing in Section \ref{Sec:encode_by_quiver}  also arise when studying neighbourhoods of a pair of intersecting floppable curves in a threefold.   We focus on the affine 3-folds $X_n = \Spec R_n$ where 
\begin{equation}
R_n = \frac{\bK[u,v,x,y]}{(uv - xy(x^{n}+y))} \qquad \textrm{for} \ n\geq 1; \qquad R_0 = \frac{\bK[u,v,x,y]}{(uv-xy^2)}.\label{flop equation}
\end{equation}
For $n\neq 0$ it is clear that $R_n$ localises to a regular local ring away from the origin, so this affine variety has a unique singular point $0\in X_n$.  Being an isolated hypersurface singularity, it follows immediately from Serre's criterion (in any characteristic) that $R_n$ is normal.  On the other hand, $X_0$ is singular precisely along the $x$-axis.  However, it too is normal.

\medskip
From these varieties, over any $\bK$, we will geometrically construct 3-CY categories $\scrC$.  This construction is standard when $\charac\,\bK=0$; the main issue in this subsection is determining that the characteristic does not effect the birational and derived geometry.

\begin{Lemma}\label{specific Yn}
For each $n\geq 0$ the variety $X_n$ admits a specific resolution $f\colon Y_n \to X_n$ for which $f^{-1}(0)$ comprises two  $(-1,-1)$-curves meeting at one point.  \end{Lemma}
\begin{proof} 
See cf.\ \cite[\S5.1]{Iyama-WemyssCrelle}.  When $\bK$ has characteristic zero, this is standard.    Below, the main fact that we will use is that 
\[
\upvarphi\colon \scrO_{\bP^1}(-1)^{\oplus 2}\to \Spec\bK[u,v,x,y]/( uv-xy),
\]
which is the blowup of either $(u,x)$ or $(u,y)$, results in a projective birational morphism with a single $(-1,-1)$-curve as the exceptional locus, regardless of the characteristic of the field.  Given the base is normal, Zariski's Main Theorem implies that $\upvarphi_*\scrO=\scrO$, and the same \v{C}ech calculation as in characteristic zero shows that $\mathbf{R}^t\upvarphi_*\scrO=0$ for all $t>0$.  Indeed, the blowup at $(u,x)$ is covered by two open charts $\Spec\bK[v,x,\tfrac{u}{x}]$ and $\Spec\bK[u,y,\tfrac{x}{u}]$, with the gluing map identifying $\bK[v,x,(\tfrac{u}{x})^{\pm1}]=\bK[u,y,(\tfrac{x}{u})^{\pm1}]$ via $x\mapsto \tfrac{x}{u}u$ and $v\mapsto \tfrac{x}{u}y$.  The resulting \v{C}ech complex for $\scrO$ is thus
\[
0\to \bK[v,x,\tfrac{u}{x}]\oplus\bK[u,y,\tfrac{x}{u}]\xrightarrow{\upalpha} \bK[v,x,(\tfrac{u}{x})^{\pm1}]=\bK[u,y,(\tfrac{x}{u})^{\pm1}]\to 0,
\]
where the map $\upalpha$ sends $(a,b)\mapsto a-b$.  The only possible nonzero cohomology lies in degree zero and one, so $\mathbf{R}^t\upvarphi_*\scrO=0$ for all $t>1$.  We claim that $\mathbf{R}^1\upvarphi_*\scrO=0$, which is equivalent to $\upalpha$ being surjective.  Choose an arbitrary monomial $\mathsf{m}$, which necessarily has the form $\mathsf{m}=v^{m}x^{n}(\tfrac{u}{x})^{z}$ with $m,n\geq 0$ and $z\in\mathbb{Z}$.  If $z\geq 0$ then $\mathsf{m}=\upalpha(v^{m}x^{n}(\tfrac{u}{x})^{z},0)$. Alternatively, if $z<0$ then under the identification above $\mathsf{m}=y^{m}u^{n}(\tfrac{x}{u})^{m+n-z}$, thus since $m+n-z\geq 0$,  $\mathsf{m}=\upalpha(0,-y^{m}u^{n}(\tfrac{x}{u})^{m+n-z})$.  Either way,  $\mathsf{m}$ is in the image, so $\upalpha$ is surjective. 

\medskip
 Now, write $R_n=\bK[u,v,x,y]/( uv-g)$, with $g=g_1g_2g_3$ a factorization into primes in $\bK[x,y]$, then small resolutions can be obtained by blowing up ideals locally of the form $(u,g_i)$. Different sequences can lead to different resolutions.

\medskip
To construct our specific resolution, consider the ordering
\[
g_1g_2g_3=\left\{
\begin{array}{cl}
yx(x^n+y) & \mbox{if }n>0,\\
yxy & \mbox{if }n=0.
\end{array}\right.
\]
Then for all $X_n$ with $n\geq 0$, we first blowup the ideal $(u,g_1)=(u,y)$ to obtain
\[
g\colon Z\to\Spec R_n,
\]
where the variety $Z$ is covered by the two open charts
\[
\bK[u,x,S]  \qquad \textrm{and} \qquad \bK[v,x,y,T] / ( vT-g_2g_3).
\]
The co-ordinates $S=g_1/u$ and $T=u/g_1$ patch to define a copy of $\bP^1$ in $Z$.  Since the base is normal, again Zariski's Main Theorem shows that $g_*\scrO=\scrO$, and the same \v{C}ech cohomology calculation as in characteristic zero (which is very similar to the above) gives $\mathbf{R}^tg_*\scrO=0$ for all $t>0$.  We conclude that $\mathbf{R}\upvarphi_*\scrO=\scrO$.  

\medskip
When $n=0$ the second chart is given by $vT=xy$, whereas when $n\geq 1$ the second chart is $vT=x(x^n+y)$.  However, after the polynomial change in co-ordinates $Y\mapsto x^n+y$, the second chart is isomorphic to $\bK[v,x,Y,T] / (vT-xY)$.  Hence in all cases, it is clear that the second chart localises to a regular local ring away from the origin, and we can resolve the second chart by blowing up $(T,g_2)=(T,x)$, whose zero set does not intersect the first chart.  This results in a morphism
\[
h\colon Y_n\to Z
\] 
Using the fact at the beginning of the proof, the second blowup replaces the unique singular point in the second chart of $Z$ with a $(-1,-1)$-curve.  Since $Z$ is normal necessarily $h_*\scrO=\scrO$, and also we have that $\mathbf{R}^1h_*\scrO=0$ since this can be checked locally.  Again as the dimension of the fibres are at most one, it follows that $\mathbf{R}h_*\scrO=\scrO$.

\medskip
Composing, we obtain a resolution $f=g\circ h\colon Y_n\to \Spec R_n$, such that $\mathbf{R}f_*\scrO=\scrO$.  Since the co-ordinates $S=g_1/u$ and $T=u/g_1$ patch to define a copy of $\bP^1$ in the partial resolution $Z$, and this $\bP^1$ passes through the node $\{vT=g_2g_3\}$ in the second chart, in the full resolution $Y$ we thus have two intersecting $\bP^1$s,  say $\mathrm{C}_1$ and $\mathrm{C}_2$, where by construction the second of which is a $(-1,-1)$ curve.  It can be seen explicitly that contracting $\mathrm{C}_1$ results in $uv=g_1g_2$, which in all cases is $uv=xy$.  Hence $\mathrm{C}_1$ is also a $(-1,-1)$-curve.
\end{proof}

\begin{Remark}\label{long remark}
Over any field $\bK$, the following hold.
\vspace{-\parskip}
\begin{enumerate}
\item The variety $X_n$ admits three distinct crepant resolutions if $n=0$ and six if $n>0$.  This is since we can also blowup using different orderings of $g_1g_2g_3$;  there are $3!=6$ options if the factors are distinct (which they are when $n>0$), but when $n=0$ the repetition of one of the factors means we only obtain three. All are related by flops, as can explicitly be seen. 

\item When $n=0$ the two curves belong to a surface; one can flop either curve to obtain a resolution (now containing a $(-2,0)$-curve) whose exceptional set is not pure-dimensional, see \cite[Figure on p.24]{DW-noncommutative_enhancements}, but the curves cannot flop together, as then the whole surface is contracted.

\item When $n>0$ the picture is similar to the above, but the surface containing the two curves is now infinitesimal.  In the case $n=1$, after all possible flops, all curves remain $(-1,-1)$-curves.  However, in the case $n>1$, floppable $(-2,0)$-curves arise.
\end{enumerate}
\vspace{-\parskip}
\end{Remark}

As in the introduction, for $n\geq 0$ consider the category
\[
\scrC_n\colonequals \langle \scrO_{\Curve_1}(-1),\scrO_{\Curve_2}(-1)\rangle\subset\Db(\coh Y_n).
\] 
where $f\colon Y_n\to X_n$ is the specific resolution constructed in Lemma~\ref{specific Yn}.  To relate this to the categories $\scrQ_k$, we pass to noncommutative resolutions, which give a third realisation of the same category.

\subsection{Noncommutative Resolutions}\label{sect:NCres}

For each $n \geq 0$, consider the maximal Cohen--Macaulay $R_n$-module $M\colonequals R_n\oplus (u,y)\oplus (u,xy)$.
\begin{Proposition}\label{Rk admits NCCR}
For each $n \geq 0$, the following statements hold.
\vspace{-\parskip}
\begin{enumerate}
\item $Y_n$ admits a tilting bundle $\scrV$, which is a direct sum of line bundles.
\item There is an isomorphism $f_*\scrV\cong M$.
\end{enumerate}  
\vspace{-\parskip}
Consequently, $Y_n$ is derived equivalent to $\Lambda_n\colonequals \End_{R_n}(M)$, the resolution $f$ is crepant, and $\Lambda_n$ is a NCCR. 
\end{Proposition}
\begin{proof}
By the explicit construction of $Y_n$, we can find a divisor $D_1$ such that $D_1\cdot \mathrm{C}_1$ is a point, and $D_1\cdot \mathrm{C}_2$ is empty (e.g.\ consider $S=0$ in the first chart).  Similarly, we can find a divisor $D_2$ such that $D_2\cdot \mathrm{C}_2$ is a point, and $D_2\cdot \mathrm{C}_1$ is empty.  Set $\scrL_i=\scrO(D_i)$, and $\scrV=\scrO\oplus \scrL_1\oplus \scrL_2$. 

\medskip
The fact that $\scrV$ is tilting in characteristic zero is \cite{VdB}.  In characteristic $p$ here, we can instead follow the more explicit method of \cite{WemyssTypeA}. Indeed, exactly as in \cite[5.7]{WemyssTypeA} the fact that $\mathbf{R}f_*\scrO=\scrO$, together with all tensor shifts of the short exact sequences
\[
0\to\scrL_i^*\to\scrO^{\oplus 2}\to\scrL_i\to 0
\]
\[
0\to\scrO\to\scrL_1\oplus\scrL_2\to \scrL_1\otimes\scrL_2\to 0
\]
is enough to establish that $\Ext^t(\scrV,\scrV)=0$ for all $t>0$.  Generation is then word-for-word identical to \cite[5.8]{WemyssTypeA}, since here $\scrL_1\otimes\scrL_2$ is ample.  Hence $\scrV$ is a tilting bundle.

\medskip
As in characteristic zero,  $h_*\scrV=\scrO\oplus\scrL_1'\oplus (T,g_2)$, where $\scrL_1'$ is also a line bundle on $Z$.  Pushing this down once more, since $T=u/g_1$ we obtain
\[
f_*\scrV=\scrO\oplus (u,g_1)\oplus (u,g_1g_2),
\] 
proving the second statement.  The fact that the natural map
\[
\End_{Y_n}(\scrV)\to \End_{R_n}(f_*\scrV)=\End_{R_n}(M)\colonequals\Lambda_n
\]
is an isomorphism follows immediately from \cite[4.3]{DW2ContractionsandDeformations}, since $f$ is a proper birational morphism between varieties, $\scrV$ is generated by global sections, and $\Ext^1(\scrV,\scrV)=0$.

\medskip
It is already known, regardless of the characteristic, that the property of being an NCCR can be checked complete locally \cite[2.17]{Iyama-Wemyss-MMAspaper}, and complete locally in the situation here the result is \cite[\S5.1]{Iyama-WemyssCrelle}.  Hence $\Lambda_n$ is a NCCR, so $Y_n$, being derived equivalent to $\Lambda_n$, is automatically a crepant resolution \cite{Iyama-Wemyss}.
\end{proof}

The algebra $\Lambda_n$ has a simple module corresponding to each summand of $M$, say $\scrS_0,\scrS_1,\scrS_2$ corresponding to $R, (u,y), (u,xy)$ respectively.  In characteristic zero the following is standard; here we just need to directly verify it.

\begin{Corollary}\label{simple corresponds}
For any $n\geq 0$, under the derived equivalence induced by the tilting bundle $\scrV$, the simple $\scrS_i$ corresponds to $\scrO_{\mathrm{C}_i}(-1)[1]$.   
\end{Corollary}
\begin{proof}
We prove $i=1$.  It is clear that $\mathbf{R}\Hom_{Y_n}(\scrO,\scrO_{\mathrm{C}_1}(-1)[1])=0$.  Further, since $D_2\cdot \mathrm{C}_1$ is empty, it is also clear that $\mathbf{R}\Hom_{Y_n}(\scrL_2,\scrO_{\mathrm{C}_1}(-1)[1])=0$.  Lastly, observe that
$
\mathbf{R}\Hom_{Y_n}(\scrL_1,\scrO_{\mathrm{C}_1}(-1)[1])=
\mathbf{R}\Hom_{Y_n}(\scrO,\scrO_{\mathrm{C}_1}(-2)[1])
=
\mathbb{K}
$
in homological degree zero.  The result follows.
\end{proof}

It follows that the category $\scrC_n$ is equivalent to all those complexes in $\Db(\mod\Lambda_n)$ whose cohomology modules are filtered by $\scrS_1,\scrS_2$. Being generated by $\scrS_1\oplus\scrS_2$, the category $\scrC_n$ is Morita equivalent to perfect modules over the $A_\infty$-algebra  $\Ext^*_{\Lambda_n}(\scrS_1\oplus\scrS_2,\scrS_1\oplus\scrS_2)$.  We will compute this $A_\infty$-algebra using the following presentation of $\Lambda_n$.

\begin{Lemma}\label{present NCCR}
For $n \geq 0$, the ring $\Lambda_n$ can be presented as the following quiver with relations:
\[
\begin{array}{ccc}
\begin{array}{c}
\begin{tikzpicture} [bend angle=45, looseness=1]
\node[name=s,regular polygon, regular polygon sides=3, minimum size=2.5cm] at (0,0) {}; 
\node (C1) at (s.corner 1)  {$\scriptstyle 1$};
\node (C2) at (s.corner 2) {$\scriptstyle 0$};
\node (C3) at (s.corner 3)  {$\scriptstyle 2$};
\draw[->] (C1) -- node[gap] {$\scriptstyle e$} (C3); 
\draw[->] (C2) --  node[gap] {$\scriptstyle c_1$} (C1); 
\draw [->,bend right] (C2) to  node[gap] {$\scriptstyle a_2$} (C3);
\draw [->,bend right] (C3) to  node[gap] {$\scriptstyle f$} (C1);
\draw [->,black ,bend right] (C1) to  node[gap] {$\scriptstyle a_1$} (C2);
\draw[->,black] (C3) --  node[gap] {$\scriptstyle c_2$} (C2);
\draw[->]  (C2) edge [in=-160,out=-95,loop,looseness=7]  node[left] {$\scriptstyle \ell$} (C2);
\end{tikzpicture}
\end{array}
&\qquad
\begin{aligned}
efa_1&=a_1\ell\\
\ell c_1&=c_1ef\\
fec_2&=c_2\ell\\
\ell a_2&=a_2fe
\end{aligned}
&\quad
\begin{aligned}
a_1c_1e&=e(c_2a_2-(fe)^n)\\
fa_1c_1&=(c_2a_2-(fe)^n)f\\
\ell^n&=a_2c_2-c_1a_1
\end{aligned}
\end{array}
\]
where if $n=0$, any monomial involving a power of $n$ should be interpreted as being zero.
\end{Lemma}
\begin{proof}
Consider first the  key variety $K=\bK[u,v,x,y,z]/(uv-xyz)$.  It is easier to first present the NCCR $\Updelta=\End_K(K\oplus (u,y)\oplus (u,xy))$ of $K$, by considering the following morphisms:
\[
\begin{tikzpicture} [bend angle=45, looseness=1]
\node[name=s,regular polygon, regular polygon sides=3, minimum size=2.5cm] at (0,0) {}; 
\node (C1) at (s.corner 1)  {$\scriptstyle (u,y)$};
\node (C2) at (s.corner 2) {$\scriptstyle K$};
\node (C3) at (s.corner 3)  {$\scriptstyle (u,xy)$};
\node (C1a) at ($(s.corner 1)+(90:0.2cm)$) {};
\node (C3a) at ($(s.corner 3)+(-45:0.25cm)$) {};
\node (C2a) at ($(s.corner 2)+(-135:0.2cm)$) {};

\draw[->] (C1) -- node[gap] {$\scriptstyle x$} (C3); 
\draw[->] (C2) --  node[gap] {$\scriptstyle y$} (C1); 
\draw [->,bend right] (C2) to  node[gap] {$\scriptstyle u$} (C3);
\draw [->,bend right] (C3) to  node[gap] {$\scriptstyle inc$} (C1);
\draw [->,black ,bend right] (C1) to  node[gap] {$\scriptstyle inc$} (C2);
\draw[->,black] (C3) --  node[gap] {$\scriptstyle z/u$} (C2);
\draw[->]  (C2a) edge [in=-160,out=-85,loop,looseness=10]  node[left] {$\scriptstyle x$} (C2a);
\draw[->]  (C1a) edge [in=55,out=125,loop,looseness=10]  node[above] {$\scriptstyle z$} (C1a);
\draw[->]  (C3a) edge [in=-95,out=-20,loop,looseness=10]  node[right] {$\scriptstyle y$} (C3a);

\end{tikzpicture}
\]
where $\frac{z}{u}$ is the morphism that sends $u\mapsto z$ and $xy\mapsto v$.
Using exactly the same technique as in \cite{Iyama-WemyssCrelle}, it is easy to see that the stated maps generate the algebra $\Updelta$, and furthermore the following relations hold:  
\[
\begin{array}{cccc}
\begin{array}{c}
\begin{tikzpicture} [bend angle=45, looseness=1]
\node[name=s,regular polygon, regular polygon sides=3, minimum size=2.5cm] at (0,0) {}; 
\node (C1) at (s.corner 1)  {$\scriptstyle 1$};
\node (C2) at (s.corner 2) {$\scriptstyle 0$};
\node (C3) at (s.corner 3)  {$\scriptstyle 2$};
\node (C1a) at ($(s.corner 1)+(90:0.2cm)$) {};
\node (C3a) at ($(s.corner 3)+(-45:0.25cm)$) {};
\node (C2a) at ($(s.corner 2)+(-135:0.2cm)$) {};

\draw[->] (C1) -- node[gap] {$\scriptstyle e$} (C3); 
\draw[->] (C2) --  node[gap] {$\scriptstyle c_1$} (C1); 
\draw [->,bend right] (C2) to  node[gap] {$\scriptstyle a_2$} (C3);
\draw [->,bend right] (C3) to  node[gap] {$\scriptstyle f$} (C1);
\draw [->,black ,bend right] (C1) to  node[gap] {$\scriptstyle a_1$} (C2);
\draw[->,black] (C3) --  node[gap] {$\scriptstyle c_2$} (C2);
\draw[->]  (C2a) edge [in=-160,out=-85,loop,looseness=10]  node[left] {$\scriptstyle \ell_1$} (C2a);
\draw[->]  (C1a) edge [in=55,out=125,loop,looseness=10]  node[above] {$\scriptstyle \ell_2$} (C1a);
\draw[->]  (C3a) edge [in=-95,out=-20,loop,looseness=10]  node[right] {$\scriptstyle \ell_3$} (C3a);
\end{tikzpicture}
\end{array}
\quad
&
\begin{aligned}
\ell_3 f&=fa_1c_1\\
\ell_3c_2&=c_2c_1a_1\\
e\ell_3&=a_1c_1e\\
a_2\ell_3&=c_1a_1a_2
\end{aligned}
\quad
&
\begin{aligned}
\ell_1 a_2&=a_2fe\\
\ell_1c_1&=c_1ef\\
c_2\ell_1&=fec_2\\
a_1\ell_1&=efa_1
\end{aligned}
\quad
&
\begin{aligned}
\ell_2 a_1&=a_1a_2c_2\\
\ell_2e&=ec_2a_2\\
c_1\ell_2&=a_2c_2c_1\\
f\ell_2&=c_2a_2f
\end{aligned}
\end{array}
\]
Thus there is an induced surjection $\bK Q/\mathcal{R} \twoheadrightarrow \Updelta$.   A diamond lemma calculation (see \cite[Appendix A]{Hao}, after substituting $n=3$) shows that this is in fact an isomorphism.

\medskip
Then, exactly as in \cite[Theorem 6.2]{KIWY}, the algebra $\Lambda_n$ can be obtained from $\Updelta$ by slicing by the central element corresponding to $z-(x^n+y)\in K$.  Hence,  a presentation of $\Lambda_n$ can be obtained from a presentation of $\Updelta$ by adding in the following three additional relations:
\[
\begin{aligned}
a_2c_2&=\ell_1^n+c_1a_1\\
\ell_2&=(ef)^n+a_1c_1\\
c_2a_2&=(fe)^n+\ell_3
\end{aligned}
\]
The second allows us to eliminate $\ell_2$, and the third allows us to eliminate $\ell_3$.  This leaves $15-2=13$ relations, but it is easy to check that six come for free from the rest, and so we are left with the seven relations in the statement of the lemma.
\end{proof}

\begin{Remark}
Note that Lemma~\ref{present NCCR} gives a non-minimal presentation when $n=1$, since in that case the loop at $0$ is not strictly required.  This does not effect the computation of the $A_\infty$ structure below, which takes place at the other vertices.
\end{Remark}

\subsection{Computing the \texorpdfstring{$A_\infty$}{A infinity} structure}
Set $\scrS=\scrS_1\oplus\scrS_2$. We now compute $A_n=\Ext^*_{\Lambda_n}(\scrS,\scrS)$, for all $n\geq 0$, over an arbitrary field $\bK$.  

\begin{Lemma}\label{proj resolutions}
The sequences
\[
0\to \scrP_1 
\xrightarrow{\left(\begin{smallmatrix} 
c_1\\f\end{smallmatrix}\right)} 
\scrP_0\oplus \scrP_2
\xrightarrow{\left(\begin{smallmatrix} 
\ell&-c_1e\\
-fa_1&c_2a_2-(fe)^n
\end{smallmatrix}\right)} 
\scrP_0\oplus \scrP_2
\xrightarrow{\left(\begin{smallmatrix} 
a_1&e\end{smallmatrix}\right)} \scrP_1\to \scrS_1\to 0
\]
\[
0\to \scrP_2
\xrightarrow{\left(\begin{smallmatrix} 
a_2\\e\end{smallmatrix}\right)} 
\scrP_0\oplus \scrP_1
\xrightarrow{\left(\begin{smallmatrix} 
\ell&-a_2f\\
-ec_2&a_1c_1+(ef)^n
\end{smallmatrix}\right)} 
\scrP_0\oplus \scrP_1
\xrightarrow{\left(\begin{smallmatrix} 
c_2&f\end{smallmatrix}\right)} \scrP_2\to \scrS_2\to 0
\]
are projective resolutions of the simples $\scrS_1$ and $\scrS_2$. 
\end{Lemma}
\begin{proof}
This is just projectivisation applied to the mutation sequences \cite[5.12, 5.13]{Iyama-WemyssCrelle}.
\end{proof}

\begin{Lemma}
For all $n\geq 0$ the algebra $\Lambda_n$ is $\mathbb{N}$-graded, with arrows graded as follows:
\[
\begin{array}{c}
\begin{tikzpicture} [bend angle=45, looseness=1]
\node[name=s,regular polygon, regular polygon sides=3, minimum size=1.8cm] at (0,0) {}; 
\node (C1) at (s.corner 1)  [B]{};
\node (C2) at (s.corner 2) [B] {};
\node (C3) at (s.corner 3)  [B] {};
\draw[->] (C1) -- node[gap] {$\scriptstyle 1$} (C3); 
\draw[->] (C2) --  node[gap] {$\scriptstyle n$} (C1); 
\draw [->,bend right] (C2) to  node[gap] {$\scriptstyle n$} (C3);
\draw [->,bend right] (C3) to  node[gap] {$\scriptstyle 1$} (C1);
\draw [->,black ,bend right] (C1) to  node[gap] {$\scriptstyle n$} (C2);
\draw[->,black] (C3) --  node[gap] {$\scriptstyle n$} (C2);
\draw[->]  (C2) edge [in=-160,out=-95,loop,looseness=7]  node[left] {$\scriptstyle 2$} (C2);
\end{tikzpicture}
\end{array}
\]
\end{Lemma}
\begin{proof}
It is a direct verification that the seven relations in Lemma~\ref{present NCCR} are all homogeneous with respect to the stated grading.
\end{proof}

To avoid confusion, we will refer to the above grading as the Adams grade, to distinguish it from the homological degree.  We can shift the projective modules via the Adams grade to make the differentials in Lemma~\ref{proj resolutions} have grade zero, namely 
\[
\begin{tikzpicture}
\node (A1) at (-0.5,0) {$\scrP_1({-2n-2})$}; 
\node (A2) at (3.75,0) {$\scrP_0({ -n-2})\oplus \scrP_2({ -2n-1})$}; 
\node (A3) at (8.5,0) {$\scrP_0({ -n})\oplus \scrP_2({ -1})$}; 
\node (A4) at (11,0) {$\scrP_1$}; 
\node (B1) at (-0.5,-1) {$\scrP_2({ -2n-2})$}; 
\node (B2) at (3.75,-1) {$\scrP_0({ -n-2})\oplus \scrP_1({ -2n-1})$}; 
\node (B3) at (8.5,-1) {$\scrP_0({ -n})\oplus \scrP_1({ -1})$}; 
\node (B4) at (11,-1) {$\scrP_2$}; 
\draw[->] (A1) --(A2);
\draw[->] (A2) --(A3);
\draw[->] (A3) -- (A4);
\draw[->] (B1) --(B2);
\draw[->] (B2) --(B3);
\draw[->] (B3) -- (B4);
\end{tikzpicture}
\]
Write $\scrP$ for the direct sum of the above two projective resolutions, and consider the DGA $\scrEnd_{\Lambda_n}(\scrP)$.  This inherits a graded structure, as all morphisms between the projectives are Adams graded.

\medskip
From Lemma~\ref{proj resolutions} it follows that the cohomology of $\scrEnd_{\Lambda_n}(\scrP)$, namely the Ext algebra $\Ext^*_{\Lambda_n}(\scrS,\scrS)$, has Hilbert series $2+2t+2t^2+2t^3$, and so we next seek generators.  For degree one, there are obvious maps representing $\Ext^1(\scrS_1,\scrS_2)=\bK$ and $\Ext^1(\scrS_2,\scrS_1)=\bK$, namely
\[
\begin{array}{c}
\begin{tikzpicture}
\node at (-3,-1) {$\upxi_{12}=$};
\node (A1) at (-0.5,0) {$\scrP_1({\scriptstyle -2n-2})$}; 
\node (A2) at (3.75,0) {$\scrP_0({\scriptstyle -n-2})\oplus \scrP_2({\scriptstyle -2n-1})$}; 
\node (A3) at (7.75,0) {$\scrP_0({\scriptstyle -n})\oplus \scrP_2({\scriptstyle -1})$}; 
\node (A4) at (10,0) {$\scrP_1$}; 
\draw[->] (A1) --(A2);
\draw[->] (A2) --(A3);
\draw[->] (A3) -- (A4);
\node (B1) at (-3.75,-2) {$\scrP_2({\scriptstyle -2n-2})$}; 
\node (B2) at (-0.5,-2) {$\scrP_0({\scriptstyle -n-2})\oplus \scrP_1({\scriptstyle -2n-1})$}; 
\node (B3) at (3.75,-2) {$\scrP_0({\scriptstyle -n})\oplus \scrP_1({\scriptstyle -1})$}; 
\node (B4) at (7.75,-2) {$\scrP_2$}; 
\draw[->] (B1) -- (B2);
\draw[->] (B2) -- (B3);
\draw[->] (B3) -- (B4);
\draw[->] (A1) -- node[left] {$\left(\begin{smallmatrix} 
0\\1\end{smallmatrix}\right)$}(B2);
\draw[->] (A2) --node[right]{$\left(\begin{smallmatrix} 
0&a_2\\
-a_1&-(ef)^{n-1}e
\end{smallmatrix}\right)$}(B3);
\draw[->] (A3) -- node[right]{$\left(\begin{smallmatrix} 
0&-1\end{smallmatrix}\right)$} (B4);
\end{tikzpicture}\\
\begin{tikzpicture}
\node at (-3,-1) {$\upxi_{21}=$};
\node (A1) at (-0.5,0) {$\scrP_2({\scriptstyle -2n-2})$}; 
\node (A2) at (3.75,0) {$\scrP_0({\scriptstyle -n-2})\oplus \scrP_1({\scriptstyle -2n-1})$}; 
\node (A3) at (7.75,0) {$\scrP_0({\scriptstyle -n})\oplus \scrP_1({\scriptstyle -1})$}; 
\node (A4) at (10,0) {$\scrP_2$}; 
\draw[->] (A1) --(A2);
\draw[->] (A2) --(A3);
\draw[->] (A3)--(A4);
\node (B1) at (-3.75,-2) {$\scrP_1({\scriptstyle -2n-2})$}; 
\node (B2) at (-0.5,-2) {$\scrP_0({\scriptstyle -n-2})\oplus \scrP_2({\scriptstyle -2n-1})$}; 
\node (B3) at (3.75,-2) {$\scrP_0({\scriptstyle -n})\oplus \scrP_2({\scriptstyle -1})$}; 
\node (B4) at (7.75,-2) {$\scrP_1$}; 
\draw[->] (B1)--(B2);
\draw[->] (B2) -- (B3);
\draw[->] (B3) -- (B4);
\draw[->] (A1) -- node[left] {$\left(\begin{smallmatrix} 
0\\1\end{smallmatrix}\right)$}(B2);
\draw[->] (A2) --node[right]{$\left(\begin{smallmatrix} 
0&c_1\\
-c_2&(fe)^{n-1}f
\end{smallmatrix}\right)$}(B3);
\draw[->] (A3) -- node[right]{$\left(\begin{smallmatrix} 
0&-1\end{smallmatrix}\right)$} (B4);
\end{tikzpicture}
\end{array}
\]
All diagrams above commute, and both chain maps give non-zero elements of $\Ext$ because the $1$'s on the vertical maps cannot be obtained by any homotopy; otherwise the trivial path is in the arrow ideal, which would contradict the grading.  For dimension reasons, these generate $\Ext^1$.  Both $\upxi_{12},\upxi_{21}$ have Adams grade $+1$.

\medskip
It is easy to check that both $\upxi_{21}\circ\upxi_{12}$ and $\upxi_{12}\circ\upxi_{21}$ are homotopic to zero.  We next seek generators of $\Ext^2(\scrS_1,\scrS_2)$ and $\Ext^2(\scrS_2,\scrS_1)$, so consider the following:
\[
\begin{array}{c}
\begin{tikzpicture}
\node at (-5,-1) {$\upzeta_{12}=$};
\node (A1) at (0,0) {$\scrP_1$}; 
\node (A2) at (2.5,0) {$\scrP_0\oplus \scrP_2$}; 
\node (A3) at (5,0) {$\scrP_0\oplus \scrP_2$}; 
\node (A4) at (7,0) {$\scrP_1$}; 
\draw[->] (A1) --node[above] {$\left(\begin{smallmatrix} 
c_1\\f\end{smallmatrix}\right)$} (A2);
\draw[->] (A2) -- (A3);
\draw[->] (A3) -- (A4);
\node (B1) at (-4.5,-2) {$\scrP_2$}; 
\node (B2) at (-2.5,-2) {$\scrP_0\oplus \scrP_1$}; 
\node (B3) at (0,-2) {$\scrP_0\oplus \scrP_1$}; 
\node (B4) at (2.5,-2) {$\scrP_2$}; 
\draw[->] (B1) -- (B2);
\draw[->] (B2)--(B3);
\draw[->] (B3) -- node[above]{$\left(\begin{smallmatrix} 
c_2&f\end{smallmatrix}\right)$} (B4);
\draw[->] (A1) -- node[left] {$\left(\begin{smallmatrix} 
0\\1\end{smallmatrix}\right)$}(B3);
\draw[->] (A2) --node[right]{$\left(\begin{smallmatrix} 
0&1
\end{smallmatrix}\right)$}(B4);
\end{tikzpicture}
\end{array}
\]
\[
\begin{array}{c}
\begin{tikzpicture}
\node at (-5,-1) {$\upzeta_{21}=$};
\node (A1) at (0,0) {$\scrP_2$}; 
\node (A2) at (2.5,0) {$\scrP_0\oplus \scrP_1$}; 
\node (A3) at (5,0) {$\scrP_0\oplus \scrP_1$}; 
\node (A4) at (7,0) {$\scrP_2$}; 
\draw[->] (A1) --node[above] {$\left(\begin{smallmatrix} 
a_2\\e\end{smallmatrix}\right)$} (A2);
\draw[->] (A2) -- (A3);
\draw[->] (A3) -- (A4);
\node (B1) at (-4.5,-2) {$\scrP_1$}; 
\node (B2) at (-2.5,-2) {$\scrP_0\oplus \scrP_2$}; 
\node (B3) at (0,-2) {$\scrP_0\oplus \scrP_2$}; 
\node (B4) at (2.5,-2) {$\scrP_1$}; 
\draw[->] (B1) -- (B2);
\draw[->] (B2)--(B3);
\draw[->] (B3) -- node[above]{$\left(\begin{smallmatrix} 
a_1&e\end{smallmatrix}\right)$} (B4);
\draw[->] (A1) -- node[left] {$\left(\begin{smallmatrix} 
0\\1\end{smallmatrix}\right)$}(B3);
\draw[->] (A2) --node[right]{$\left(\begin{smallmatrix} 
0&1
\end{smallmatrix}\right)$}(B4);
\end{tikzpicture}
\end{array}
\]
These give non-zero elements in $\Ext^2$ for the same reason as before: the trivial paths don't belong to the arrow ideal.  Both $\upzeta_{12},\upzeta_{21}$ have Adams grade $2n+1$.

\medskip
It is very easy to check that $\upzeta_{21}\circ\upxi_{12}=-\upxi_{21}\circ\upzeta_{12}$, as the left hand side is the identity on $\scrP_1$ (and is zero elsewhere), whilst the right hand side is the negative of the identity on $\scrP_1$.  Consequently both are non-zero elements of $\Ext^3(\scrS_1,\scrS_1)$, again since the trivial path is not in the arrow ideal.  A similar thing happens for the indices the other way round.  Thus we can realise all non-zero elements in $\Ext^*(\scrS,\scrS)$ via combinations of the elements $\upxi_{12}, \upxi_{21}, \upzeta_{12}, \upzeta_{21}$, and so these generate $\Ext^*(\scrS,\scrS)$.  

\medskip
To compute the higher Massey products, it will be convenient to consider the following:
\[
\begin{array}{c}
\begin{tikzpicture}
\node at (-3,-1) {$h_{11}=$};
\node (A1) at (-0.5,0) {$\scrP_1({\scriptstyle -2n-2})$}; 
\node (A2) at (3.75,0) {$\scrP_0({\scriptstyle -n-2})\oplus \scrP_2({\scriptstyle -2n-1})$}; 
\node (A3) at (7.75,0) {$\scrP_0({\scriptstyle -n})\oplus \scrP_2({\scriptstyle -1})$}; 
\node (A4) at (10,0) {$\scrP_1$}; 
\draw[->] (A1) --(A2);
\draw[->] (A2) --(A3);
\draw[->] (A3) -- (A4);
\node (B1) at (-3.75,-2) {$\scrP_1({\scriptstyle -2n-2})$}; 
\node (B2) at (-0.5,-2) {$\scrP_0({\scriptstyle -n-2})\oplus \scrP_2({\scriptstyle -2n-1})$}; 
\node (B3) at (3.75,-2) {$\scrP_0({\scriptstyle -n})\oplus \scrP_2({\scriptstyle -1})$}; 
\node (B4) at (7.75,-2) {$\scrP_1$}; 
\draw[->] (B1) -- (B2);
\draw[->] (B2) -- (B3);
\draw[->] (B3) -- (B4);
\draw[->] (A1) -- node[left] {$\scriptstyle 0$}(B2);
\draw[->] (A2) --node[right]{$\left(\begin{smallmatrix} 
1&0\\
0&(fe)^{n-1}
\end{smallmatrix}\right)$}(B3);
\draw[->] (A3) -- node[right]{$\scriptstyle 0$} (B4);
\end{tikzpicture}\\
\begin{tikzpicture}
\node at (-3,-1) {$h_{22}=$};
\node (A1) at (-0.5,0) {$\scrP_2({\scriptstyle -2n-2})$}; 
\node (A2) at (3.75,0) {$\scrP_0({\scriptstyle -n-2})\oplus \scrP_1({\scriptstyle -2n-1})$}; 
\node (A3) at (7.75,0) {$\scrP_0({\scriptstyle -n})\oplus \scrP_1({\scriptstyle -1})$}; 
\node (A4) at (10,0) {$\scrP_2$}; 
\draw[->] (A1) --(A2);
\draw[->] (A2) --(A3);
\draw[->] (A3)--(A4);
\node (B1) at (-3.75,-2) {$\scrP_2({\scriptstyle -2n-2})$}; 
\node (B2) at (-0.5,-2) {$\scrP_0({\scriptstyle -n-2})\oplus \scrP_1({\scriptstyle -2n-1})$}; 
\node (B3) at (3.75,-2) {$\scrP_0({\scriptstyle -n})\oplus \scrP_1({\scriptstyle -1})$}; 
\node (B4) at (7.75,-2) {$\scrP_2$}; 
\draw[->] (B1)--(B2);
\draw[->] (B2) --(B3);
\draw[->] (B3)--(B4);
\draw[->] (A1) -- node[left] {$\scriptstyle 0$}(B2);
\draw[->] (A2) --node[right]{$\left(\begin{smallmatrix} 
1&0\\
0&-(ef)^{n-1}
\end{smallmatrix}\right)$}(B3);
\draw[->] (A3) -- node[right]{$\scriptstyle 0$} (B4);
\end{tikzpicture}
\end{array}
\]
and further, provided $n\geq 2$, the maps $g_{11}^j, g_{22}^j$ (for $0\leq j\leq n-2$) defined
\[
\begin{array}{c}
\begin{tikzpicture}
\node at (-3,-1) {$g_{11}^j=$};
\node (A1) at (-0.5,0) {$\scrP_1({\scriptstyle -2n-2})$}; 
\node (A2) at (3.75,0) {$\scrP_0({\scriptstyle -n-2})\oplus \scrP_2({\scriptstyle -2n-1})$}; 
\node (A3) at (7.75,0) {$\scrP_0({\scriptstyle -n})\oplus \scrP_2({\scriptstyle -1})$}; 
\node (A4) at (10,0) {$\scrP_1$}; 
\draw[->] (A1) --(A2);
\draw[->] (A2) --(A3);
\draw[->] (A3) -- (A4);
\node (B1) at (-3.75,-2) {$\scrP_1({\scriptstyle -2n-2})$}; 
\node (B2) at (-0.5,-2) {$\scrP_0({\scriptstyle -n-2})\oplus \scrP_2({\scriptstyle -2n-1})$}; 
\node (B3) at (3.75,-2) {$\scrP_0({\scriptstyle -n})\oplus \scrP_2({\scriptstyle -1})$}; 
\node (B4) at (7.75,-2) {$\scrP_1$}; 
\draw[->] (B1) -- (B2);
\draw[->] (B2) -- (B3);
\draw[->] (B3) -- (B4);
\draw[->] (A1) -- node[left] {$\scriptstyle 0$}(B2);
\draw[->] (A2) --node[right]{$\left(\begin{smallmatrix} 
0&0\\
0&(fe)^{n-2-j}
\end{smallmatrix}\right)$}(B3);
\draw[->] (A3) -- node[right]{$\scriptstyle 0$} (B4);
\end{tikzpicture}\\
\begin{tikzpicture}
\node at (-3,-1) {$g^j_{22}=$};
\node (A1) at (-0.5,0) {$\scrP_2({\scriptstyle -2n-2})$}; 
\node (A2) at (3.75,0) {$\scrP_0({\scriptstyle -n-2})\oplus \scrP_1({\scriptstyle -2n-1})$}; 
\node (A3) at (7.75,0) {$\scrP_0({\scriptstyle -n})\oplus \scrP_1({\scriptstyle -1})$}; 
\node (A4) at (10,0) {$\scrP_2$}; 
\draw[->] (A1) --(A2);
\draw[->] (A2) --(A3);
\draw[->] (A3)--(A4);
\node (B1) at (-3.75,-2) {$\scrP_2({\scriptstyle -2n-2})$}; 
\node (B2) at (-0.5,-2) {$\scrP_0({\scriptstyle -n-2})\oplus \scrP_1({\scriptstyle -2n-1})$}; 
\node (B3) at (3.75,-2) {$\scrP_0({\scriptstyle -n})\oplus \scrP_1({\scriptstyle -1})$}; 
\node (B4) at (7.75,-2) {$\scrP_2$}; 
\draw[->] (B1)--(B2);
\draw[->] (B2) --(B3);
\draw[->] (B3)--(B4);
\draw[->] (A1) -- node[left] {$\scriptstyle 0$}(B2);
\draw[->] (A2) --node[right]{$\left(\begin{smallmatrix} 
0&0\\
0&-(ef)^{n-2-j}
\end{smallmatrix}\right)$}(B3);
\draw[->] (A3) -- node[right]{$\scriptstyle 0$} (B4);
\end{tikzpicture}
\end{array}
\]
and also the maps $k_{12}^j, k_{21}^j$ (for $0\leq j\leq n-2$) defined
\[
\begin{array}{c}
\begin{tikzpicture}
\node at (-3,-1) {$k_{12}^j=$};
\node (A1) at (-0.5,0) {$\scrP_1({\scriptstyle -2n-2})$}; 
\node (A2) at (3.75,0) {$\scrP_0({\scriptstyle -n-2})\oplus \scrP_2({\scriptstyle -2n-1})$}; 
\node (A3) at (7.75,0) {$\scrP_0({\scriptstyle -n})\oplus \scrP_2({\scriptstyle -1})$}; 
\node (A4) at (10,0) {$\scrP_1$}; 
\draw[->] (A1) --(A2);
\draw[->] (A2) --(A3);
\draw[->] (A3) -- (A4);
\node (B1) at (-3.75,-2) {$\scrP_2({\scriptstyle -2n-2})$}; 
\node (B2) at (-0.5,-2) {$\scrP_0({\scriptstyle -n-2})\oplus \scrP_1({\scriptstyle -2n-1})$}; 
\node (B3) at (3.75,-2) {$\scrP_0({\scriptstyle -n})\oplus \scrP_1({\scriptstyle -1})$}; 
\node (B4) at (7.75,-2) {$\scrP_2$}; 
\draw[->] (B1) -- (B2);
\draw[->] (B2) -- (B3);
\draw[->] (B3) -- (B4);
\draw[->] (A1) -- node[left] {$\scriptstyle 0$}(B2);
\draw[->] (A2) --node[right]{$\left(\begin{smallmatrix} 
0&0\\
0&-(ef)^{n-2-j}e
\end{smallmatrix}\right)$}(B3);
\draw[->] (A3) -- node[right]{$\scriptstyle 0$} (B4);
\end{tikzpicture}\\
\begin{tikzpicture}
\node at (-3,-1) {$k^j_{21}=$};
\node (A1) at (-0.5,0) {$\scrP_2({\scriptstyle -2n-2})$}; 
\node (A2) at (3.75,0) {$\scrP_0({\scriptstyle -n-2})\oplus \scrP_1({\scriptstyle -2n-1})$}; 
\node (A3) at (7.75,0) {$\scrP_0({\scriptstyle -n})\oplus \scrP_1({\scriptstyle -1})$}; 
\node (A4) at (10,0) {$\scrP_2$}; 
\draw[->] (A1) --(A2);
\draw[->] (A2) --(A3);
\draw[->] (A3)--(A4);
\node (B1) at (-3.75,-2) {$\scrP_1({\scriptstyle -2n-2})$}; 
\node (B2) at (-0.5,-2) {$\scrP_0({\scriptstyle -n-2})\oplus \scrP_2({\scriptstyle -2n-1})$}; 
\node (B3) at (3.75,-2) {$\scrP_0({\scriptstyle -n})\oplus \scrP_2({\scriptstyle -1})$}; 
\node (B4) at (7.75,-2) {$\scrP_1$}; 
\draw[->] (B1)--(B2);
\draw[->] (B2) --(B3);
\draw[->] (B3)--(B4);
\draw[->] (A1) -- node[left] {$\scriptstyle 0$}(B2);
\draw[->] (A2) --node[right]{$\left(\begin{smallmatrix} 
0&0\\
0&(fe)^{n-2-j}f
\end{smallmatrix}\right)$}(B3);
\draw[->] (A3) -- node[right]{$\scriptstyle 0$} (B4);
\end{tikzpicture}
\end{array}
\]
Write $\updelta$ for the differential on $\scrEnd_{\Lambda_n}(\scrP)$.
\begin{Lemma}\label{Massey prep}
If $n\geq 1$, then with notation as above, the following hold.
\begin{enumerate}
\item\label{Massey prep 1} $\updelta h_{11}=\upxi_{21}\circ\upxi_{12}$ and $\updelta h_{22}=\upxi_{12}\circ\upxi_{21}$. 
\item\label{Massey prep 2} There are equalities
\[
\upxi_{12} h_{11}+h_{22}\upxi_{12}
=
\left\{
\begin{array}{ll}
-\upzeta_{12}&\mbox{if }n=1\\
\updelta k_{12}^{0}&\mbox{if }n>1
\end{array}
\right.
\quad
\upxi_{21} h_{22}+h_{11}\upxi_{21}
=
\left\{
\begin{array}{ll}
\upzeta_{21}&\mbox{if }n=1\\
\updelta k_{21}^0&\mbox{if }n>1.
\end{array}
\right.
\]
\item\label{Massey prep 3} If $n\geq 2$, then for all $0\leq t\leq n-2$ there are equalities
\[
\upxi_{12}g_{11}^t+g_{22}^t\upxi_{12}
=
\left\{
\begin{array}{cl}
-\upzeta_{12}&\mbox{if }t=n-2\\
\updelta k_{12}^{t+1}&\mbox{if }t<n-2
\end{array}
\right.
\quad
\upxi_{21}g_{22}^t+g_{11}^t\upxi_{21}
=
\left\{
\begin{array}{cl}
\upzeta_{21}&\mbox{if }t=n-2\\
\updelta k_{21}^{t+1}&\mbox{if }t<n-2.
\end{array}
\right.
\]
\end{enumerate}
\end{Lemma}
\begin{proof}
All are direct verifications. For example, we see explicitly that $\upxi_{21}\circ\upxi_{12}$ equals
\[
\begin{array}{c}
\begin{tikzpicture}
\node (A1) at (0,0) {$\scrP_1$}; 
\node (A2) at (2.5,0) {$\scrP_0\oplus \scrP_2$}; 
\node (A3) at (5,0) {$\scrP_0\oplus \scrP_2$}; 
\node (A4) at (7,0) {$\scrP_1$}; 
\draw[->] (A1) --node[above] {$\left(\begin{smallmatrix} 
c_1\\f\end{smallmatrix}\right)$} (A2);
\draw[->] (A2) -- (A3);
\draw[->] (A3) -- (A4);
\node (B1) at (-4.5,-2) {$\scrP_1$}; 
\node (B2) at (-2.5,-2) {$\scrP_0\oplus \scrP_2$}; 
\node (B3) at (0,-2) {$\scrP_0\oplus \scrP_2$}; 
\node (B4) at (2.5,-2) {$\scrP_1$}; 
\draw[->] (B1) -- (B2);
\draw[->] (B2)--(B3);
\draw[->] (B3) -- node[above]{$\left(\begin{smallmatrix} 
a_1&e\end{smallmatrix}\right)$} (B4);
\draw[->] (A1) -- node[left] {$\left(\begin{smallmatrix} 
c_1\\(fe)^{n-1}f\end{smallmatrix}\right)$}(B3);
\draw[->] (A2) --node[right]{$\left(\begin{smallmatrix} 
a_1&(ef)^{n-1}e
\end{smallmatrix}\right)$}(B4);
\end{tikzpicture}
\end{array}
\]
This is homotopic to zero, via $h_{11}$, so $\updelta h_{11}=\upxi_{21}\upxi_{12}$.  Similarly, for the first claim in (2), we explicitly see that  $\upxi_{12} h_{11}+h_{22}\upxi_{12}$ equals
\[
\begin{array}{c}
\begin{tikzpicture}
\node (A1) at (0,0) {$\scrP_1$}; 
\node (A2) at (2.5,0) {$\scrP_0\oplus \scrP_2$}; 
\node (A3) at (5,0) {$\scrP_0\oplus \scrP_2$}; 
\node (A4) at (7,0) {$\scrP_1$}; 
\draw[->] (A1) --node[above] {$\left(\begin{smallmatrix} 
c_1\\f\end{smallmatrix}\right)$} (A2);
\draw[->] (A2) -- (A3);
\draw[->] (A3) -- (A4);
\node (B1) at (-4.5,-2) {$\scrP_2$}; 
\node (B2) at (-2.5,-2) {$\scrP_0\oplus \scrP_1$}; 
\node (B3) at (0,-2) {$\scrP_0\oplus \scrP_1$}; 
\node (B4) at (2.5,-2) {$\scrP_2$}; 
\draw[->] (B1) -- (B2);
\draw[->] (B2)--(B3);
\draw[->] (B3) -- node[above]{$\left(\begin{smallmatrix} 
c_2&f\end{smallmatrix}\right)$} (B4);
\draw[->] (A1) -- node[left] {$\left(\begin{smallmatrix} 
0\\-(ef)^{n-1}\end{smallmatrix}\right)$}(B3);
\draw[->] (A2) --node[right]{$\left(\begin{smallmatrix} 
0&-(fe)^{n-1}
\end{smallmatrix}\right)$}(B4);
\end{tikzpicture}
\end{array}
\]
When $n=1$ the vertical maps are $0\choose -1$ and $(0\,-\!1)$, which is $-\upzeta_{12}$.  When $n>1$ the above is homotopic to zero, via $k_{12}^{0}$, so $\updelta k_{12}^0=\upxi_{12} h_{11}+h_{22}\upxi_{12}$.  All other claims are similar.
\end{proof}

To set notation, consider the $n$-fold Massey product $\langle a,b,\hdots\rangle_n$.
\begin{Lemma}\label{Massey exists}
If $n\geq 1$, then over any field $\bK$, we have $\langle \upxi_{12},\upxi_{21},\hdots,\upxi_{12}\rangle_{2n+1}=-\upzeta_{12}$ and $\langle \upxi_{21},\upxi_{12},\hdots,\upxi_{21}\rangle_{2n+1}=\upzeta_{21}$. 
\end{Lemma}
\begin{proof}
We prove the first, with the second being similar.  As already remarked $[\upxi_{12}][\upxi_{21}]=0$ and $[\upxi_{21}][\upxi_{12}]=0$, and further $\updelta h_{11}=\upxi_{21}\upxi_{12}$ and $\updelta h_{22}=\upxi_{12}\upxi_{21}$ by Lemma~\ref{Massey prep}\eqref{Massey prep 1}. Hence by definition $[\upxi_{12}h_{11}+h_{22}\upxi_{12}]\in\langle\upxi_{12},\upxi_{21},\upxi_{12}\rangle$.  If $n=1$, then the result follows by Lemma~\ref{Massey prep}\eqref{Massey prep 2}.

\medskip
Hence we can assume that $n>1$.  Using Lemma~\ref{Massey prep}\eqref{Massey prep 2}\eqref{Massey prep 3}, we see directly that
\[
[\upxi_{12}g_{11}^0+h_{22}k_{12}^0+k_{12}^0h_{11}+g_{22}^0\upxi_{12}]\in\langle\upxi_{12},\upxi_{21},\upxi_{12},\upxi_{21},\upxi_{12}\rangle
\]
But the middle $h_{22}k_{12}^0+k_{12}^0h_{11}=0$.  Hence $[\upxi_{12}g_{11}^0+g_{22}^0\upxi_{12}]$ belongs to the fivefold Massey product.  If $n=2$, we are done, again by Lemma~\ref{Massey prep}\eqref{Massey prep 3}.

\medskip
The proof proceeds by induction.  Using the equations in Lemma~\ref{Massey prep}\eqref{Massey prep 2}\eqref{Massey prep 3}, together with the fact that all middle terms involving compositions with $g$, $h$ and $k$ are zero, and we deduce that $[\upxi_{12}g_{11}^{n-2}+g_{22}^{n-2}\upxi_{12}]\in\langle \upxi_{12},\upxi_{21},\hdots,\upxi_{12}\rangle_{2n+1}$.  The result follows by Lemma~\ref{Massey prep}\eqref{Massey prep 3}.
\end{proof}

\begin{Proposition} \label{Prop:B_side}
For $n>0$, over any $\bK$ the algebra $A_n$ is $A_{\infty}$-quasi-equivalent to the two-cycle quiver, with $\upmu^1=0$, $\upmu^2$ being multiplication, and the unique higher multiplications being $\upmu^{2n+1}(e,f,\hdots)=e$ and $\upmu^{2n+1}(f,e,\hdots)=f$.
The algebra $A_0$ is formal. 
\end{Proposition}
\begin{proof}
Again, as in Remark~\ref{danger in char p} we cannot \emph{a priori} make the $A_{\infty}$-structure both strictly cyclic and strictly unital over $\bK$, so we work directly with the higher $A_\infty$ products.  We can however chose the higher multiplications so that $\upmu^1=0$, $\upmu^2$ is multiplication, and all the $\upmu^n$ preserve the Adams degree; see \cite{LPWZ, Booth}, which works over any characteristic.

\medskip
Recall that $\Ext^1$ is generated by terms of Adams grade $1$, $\Ext^2$ by elements of Adams grade $2n+1$, and $\Ext^3$ by elements of Adams grade $2n+2$.  Hence, when $n=0$, all higher multiplications $\upmu^k$ with $k\geq 3$ are zero, just since they preserve the Adams grading and have $\geq 3$ inputs.

\medskip
When $n\geq 1$ we appeal to Lemma~\ref{Massey exists}.  Indeed, now for purely Adams grading reasons the only possible non-zero multiplications are 
\[
\upmu^{2n+1}(\upxi_{12},\upxi_{21},\hdots)\quad\mbox{and}\quad\upmu^{2n+1}(\upxi_{21},\upxi_{12},\hdots).
\]
Higher multiplications give rise to Massey products, and the first non-trivial higher multiplication is a Massey product, up to sign (see e.g.\ \cite[Theorem 3.1]{LPWZ} or \cite[2.1(ii)]{BMM}).  Hence
\[
\upmu^{2n+1}(\upxi_{12},\upxi_{21},\hdots)=\pm \upzeta_{12}\quad\mbox{and}\quad\upmu^{2n+1}(\upxi_{21},\upxi_{12},\hdots)=\pm \upzeta_{21}.
\]
Regardless of the signs, we can change variables $\upxi\mapsto -\upxi$ if necessary to ensure that 
\[
\upmu^{2n+1}(\upxi_{12},\upxi_{21},\hdots)=\upzeta_{12}\quad\mbox{and}\quad\upmu^{2n+1}(\upxi_{21},\upxi_{12},\hdots)=\upzeta_{21}.\qedhere
\]
\end{proof}

\begin{Remark}
In a toric 3-fold, two intersecting rational curves cannot be flopped together, so none of the potentials $(ef)^{n+1}$ with $n>0$ can arise from toric geometries.  The zero potential, corresponding to $uv=xy^2$, is toric, and is known as the suspended pinch point.
\end{Remark}

\subsection{The categories \texorpdfstring{$\scrC_n$}{Cn} and \texorpdfstring{$\scrQ_n$}{Qn}}

As in the introduction, consider the full subcategory
\[
\scrC_n =\langle \scrO_{\Curve_1}(-1),\scrO_{\Curve_2}(-1)\rangle\subset\Db(\coh Y_n).
\]  
Since the distinct crepant resolutions are related by flops, in characteristic zero the quasi-equivalence class of $\scrC_n$ is independent of the choice of resolution by \cite{Bridgeland:flop}.   In the positive characteristic case here, all NCCRs are derived equivalent, regardless of characteristic, and the mutation functors still provide derived equivalences between them.  

\medskip
Indeed, repeating Lemma~\ref{specific Yn} for the other resolutions described in Remark~\ref{long remark}, we explicitly verify that each crepant resolution admits a tilting bundle with direct summand $\scrO$.  The mutation functors on the resulting categories $\scrC$ show that all are derived equivalent.  In particular, even in characteristic $p$, in our particular setting the categories $\scrC$ are independent of the chosen crepant resolution.

\medskip
By definition, a strong generator for $\scrC_n$ is given by the direct sum of the sheaves $\scrO_{\mathrm{C}_i}(-1)$, where $\mathrm{C}_i$ are the irreducible components of $f^{-1}(0)$.  Across the equivalence with $\Lambda_n$, this corresponds to the direct sum $\scrS=\scrS_1\oplus\scrS_2$.  Thus we can encode the category $\scrC_n$ by recording the $A_{\infty}$-structure on the algebra $A_n = \Ext_{\Lambda_n}^*(\scrS,\scrS)$.

\medskip
The following establishes Theorem~\ref{thm:main}.
\begin{Corollary}\label{comparison A and B}
With notation as above, the following hold.
\begin{enumerate}
\item If $\bK = \bC$, then there are equivalences $\scrQ_0 \simeq \scrC_0$ and $\scrQ_k \simeq \scrC_{1}$ for all $k \geq 1$.
\item If $p>2$ is prime and $\charac\,\bK=p$, then there is an equivalence $\scrQ_p \simeq \scrC_p$.
\end{enumerate}
\end{Corollary}
\begin{proof}
Part (1) is now an immediate consequence of Propositions~\ref{Prop:they_are_quivers} and \ref{Prop:B_side}, and Part (2) is an immediate consequence of  Corollary~\ref{A side char p main} and Proposition~\ref{Prop:B_side}.
\end{proof}

\begin{Remark}\label{Rmk:SYZ} 
The work of Abouzaid, Auroux and Katzarkov \cite{AAK} provides SYZ-type mirrors to affine conic fibrations of the form $Y_H = \{uv=H(x,y)\}$. (They primarily take $Y_H$ on the $A$- rather than $B$-side, but cf. their Section 8. When $H$ is not part of a maximally degenerating family of hypersurfaces, their mirror involves  passing to a crepant resolution. In any case, one can follow their general prescription, building Lagrangian tori in $Y_H$ invariant under the $S^1$-action rotating $u,v$ in opposite directions by lifting Lagrangian tori from the reduced spaces, and see what one gets in our particular cases.) The torically non-generic, meaning some factors are not seen in $(\bC^*)^2$, discriminant hypersurface $H = \{xy(x+y^n) = 0\} \subset \bC^2$ has tropicalization with moment polytope 
\[
\{(p,q,s)\  \, | \,  s \geq p+q+\max(p,nq) \}.
\]
The corresponding toric variety is $\bC^2 \times \bC^*$, with co-ordinates $(a,b,c)$, and the mirror to $Y_H$ following \cite{AAK} is given by $\Cech{Y}_H = (\bC^2 \backslash \{ab=1\}) \times \bC^*$ with superpotential $W = a+b^n\cdot c$. Whilst the corresponding Landau-Ginzburg category may contain spherical objects, the absence of any  Lagrangian sphere in $\Cech{Y}_H$ in makes this arguably a less compelling mirror than the double bubble. One can perturb the defining equation to be torically generic, for instance considering $H' = \{(x-1)(y-1)(x-1 + (y-1)^n) = 0 \} \subset \bC^2$. This yields a rather singular mirror $\Cech{Y}_{H'}$, given by deleting a hypersurface from the toric variety with moment polytope
\[
\{(p,q,s) \, | \, s \geq \max(p,0) + \max(q,0) + \max(p,nq,0)\}
\]
and equipping the result with a superpotential associated to toric monomials defined by the points $(-1,0,0)$ and $(0,-1,0)$.  This singular mirror admits a torus fibration whose discriminant has components contained in three parallel planes, comprising lines parallel to the $x$-axis and $y$-axis in two, and a singular discriminant component in a third which contains a ray parallel to the line $y=nx$  (see  \cite[Figure 10(b)]{Lau} for the case $n=1$, when the discriminant is actually smooth). Such a mirror contains a pair of Lagrangian 3-spheres meeting cleanly in a circle with local model $W_n$ (however, this `mirror' has much more complicated global topology). 
\end{Remark}

\section{Autoequivalences}

\subsection{Morse-Bott-Lefschetz presentations\label{Subsec:MBL_presentations}}

A Morse-Bott-Lefschetz fibration $p: X \to \bC$ is a holomorphic map with transversely non-degenerate critical manifolds, meaning that the complex Hessian 
\[
(D^2p)_x \colon \nu_x \otimes \nu_x \to T_{p(x)}\bC
\]
 is non-degenerate as a complex bilinear form on each fibre of the complex normal bundle over the critical locus $x\in \mathrm{Crit}(p)$. One can then construct oriented local almost complex co-ordinates $(z_1,\ldots,z_n)$ near a point of the critical locus with respect to which the map is given by 
\[
p\colon (z_1,\ldots, z_n) \mapsto \sum_{i=1}^k z_i^2
\]
for some $1\leq k\leq n$ (where $k=n$ would correspond to the Lefschetz case).  A detailed reference for the symplectic geometry of such fibrations is \cite{Perutz}; note that an `elementary' fibration (with a unique critical fibre) deformation retracts to that fibre, and has a model determined up to symplectic deformation by the critical locus and its normal bundle (as a complex vector bundle; this determines the value $k$). 
\medskip

Consider a Morse-Bott Lefschetz fibration $X \to \bC$ with general fibre $(\bC^*)^2$, with critical fibres isomorphic to $(\bC^*)\times (\bC\vee\bC)$, and with critical manifolds $\bC^*$. 

\begin{Lemma}
$X$ admits  a symplectic form for which parallel transport is globally defined for paths away from the critical values. 
\end{Lemma}

\begin{proof}
We begin with a general statement about Hamiltonian group actions preserving the fibres of a symplectic fibration.  Suppose $\uppi\colon X \to B$ is a symplectic fibration and $X$ admits a Hamiltonian action of a group $G$ preserving the fibres of $\uppi$. The symplectic structure on $X$ gives a distinguished symplectic connection and hence locally defined symplectic parallel transport maps. We claim that these locally defined maps  preserve the level sets of the moment map; in particular, if the moment map is proper, parallel transport is globally defined. To justify the claim, let $\upgamma\colon [0,1] \to B$ be a path and let $V$ be the horizontal lift at a point of $X$ of $\upgamma'(t)$. Let $H_{\upxi}$ be a component of the moment map associated to a Lie algebra element $\upxi \in \mathfrak{g}$, with associated Hamiltonian vector field $v_{\upxi}$, so that $\upiota_{v_{\upxi}}\upomega=dH_{\upxi}$, and let $\upphi_t$ be the Hamiltonian flow of $v_{\upxi}$. Differentiating $\uppi(\upphi_t(x))=\uppi(x)$ with respect to $t$, we get $\uppi_*v_{\upxi}=0$, so $v_{\upxi}$ is tangent to the fibres of $\uppi$. Then
  \[
  \mathcal{L}_{V}H_{\upxi}=\upiota_VdH_{\upxi} = \upomega(v_{\upxi},V)=0,
  \]
  since $v_{\upxi}$ is tangent to the fibres of $\uppi$ whilst $V$ is symplectically orthogonal to them. This proves that each component of the moment map is preserved by the flow of $V$, i.e. by the symplectic parallel transport along $\upgamma$.
  
 \medskip
 It will follow from Lemma \ref{Lem:affine_flag} below that the local model of a Morse-Bott-Lefschetz fibration with one critical fibre admits a fibrewise Hamiltonian $T^2$-action with proper moment map, hence has globally defined parallel transport.  We may view $X$ as a fibre sum of such local pieces by mapping the base $B = \bC$ to a tree of discs each containing one critical point, and then patch the locally defined parallel transport maps.
\end{proof}

The monodromy maps are fibred Dehn twists; these act trivially on the homology of the fibre, so the one-dimensional homology classes of $(\bC^*)^2 \simeq T^2$ can be labelled consistently in the fibration. We will usually write $a,b$ for a choice of meridian and longitude of the torus.

\begin{Example}\label{Ex:MBLefschetz}
A path $\upgamma \subset \bC$ between two critical values (and meeting no critical values in its interior) defines a Lagrangian submanifold $L_{\upgamma} \subset X$, formed of two Morse-Bott Lefschetz thimbles; by construction $L_{\upgamma}$ is obtained from gluing together two Lagrangian solid tori, hence is diffeomorphic to a Lens space or $S^1\times S^2$.   Denoting the standard meridian and longitude curves in the torus $a$ and $b$, then the double bubble plumbings $W_n$ are associated  to fibrations with vanishing cycles $a,b, nb \pm a$ respectively.  See Figure \ref{Fig:MBLdegeneration}.
\end{Example}

\begin{figure}[ht]
\begin{center}
\begin{tikzpicture}[scale=1]

\draw[semithick,dashed] (1,0.5) -- (1,2);
\draw[semithick,blue] (1,3) ellipse (0.5cm and 1cm);
\draw[semithick, blue] (1,3.5) arc (150:210:1);
\draw[semithick, blue] (0.9,3.3) arc (35:-30:0.6);

\draw[semithick, red] (1,3) ellipse (0.3cm and 0.8 cm);
\draw[semithick, red] (1.02,3) arc (200:345:0.25);
\draw[semithick, red, dotted] (1.02,3) arc (110:70:0.75);

\draw[semithick, red, dashed, rounded corners] [->] (0,2) -- (0,3) -- (0.7,3);
\draw (0,1.5) node {$b$};

\draw[semithick, red, dashed, rounded corners] [->] (2,2) -- (2,2.5) -- (1.5,3);
\draw (2,1.5) node {$a$};

\draw[fill] (0,0) circle (0.1);
\draw[fill] (2,0) circle (0.1);
\draw[fill] (-2,0) circle (0.1);

\draw[semithick] plot [smooth, tension=1] coordinates { (0,0)  (0.5,0.5) (1.25,-0.25) (2,0) };
\draw[semithick] plot [smooth, tension=1] coordinates { (0,0)  (-0.5,-0.25) (-1.25,0.25) (-2,0) };

\draw[fill] (7,0) circle (0.1);
\draw[fill] (9,0) circle (0.1);
\draw[fill] (5,0) circle (0.1);

\draw[semithick] plot [smooth, tension=1] coordinates { (7,0)  (7.5,0.5) (8.25,-0.25) (9,0) };
\draw[semithick] plot [smooth, tension=1] coordinates { (7,0)  (6.5,-0.25) (5.755,0.25) (5,0) };

\draw (6.8,0.4) node {$b$};
\draw (5,0.45) node {$a$};
\draw (9,0.45) node {$a+kb$};

\draw[semithick, blue] (6.3,2) -- (7.7,2) -- (7.7,2.2) -- (6.3,2.2) -- (6.3,2);
\draw[semithick, blue] (4.9,1.5) -- (5.1,1.5) -- (5.1,2.5) -- (4.9,2.5) -- (4.9,1.5);
\draw[semithick,blue] (8.5,1.5) -- (8.9,1.5) -- (9.9,2.5) -- (9.5,2.5) -- (8.5,1.5);

\draw[semithick, dashed, gray] (5,1.3) -- (5,0.7);
\draw[semithick, dashed, gray] (7,1.8) -- (7,0.6);
\draw[semithick, dashed, gray] (9,1.3) -- (9,0.7);


\end{tikzpicture}
\caption{Vanishing cycles for the Morse-Bott-Lefschetz presentation of $W_k$\label{Fig:MBLdegeneration}}
\end{center}
\end{figure}

\subsection{Affine realisations\label{Sec:affine}}
Let $\{f(x_1,\ldots,x_n)=0\}\subset A \subset \bC^n$ define an affine hypersurface $\{f=0\}$ of an affine variety $A$. The \emph{spinning} of $f$ is the affine variety
\[
\scrS(f) = \{(u,v,\mathbf{x}) \in \bC^2 \times A \, | \, uv = f(x_1,\ldots,x_n)\}
\]
There is a natural projection $\scrS(f) \to A$ which is a conic bundle with the zero-locus of $f$ as a Morse-Bott-Lefschetz discriminant.   We equip both $\{f=0\}$ and $\scrS(f)$ with the obvious K\"ahler structures induced from Euclidean space.

\begin{Lemma}
A Lagrangian disc $D^n \subset \bC^n$ with $\partial D^n \subset \{f=0\}$ and whose interior is disjoint from that hypersurface naturally lifts to a Lagrangian $S^{n+1} \subset \scrS(f)$.
\end{Lemma}
\begin{proof}
This `suspension' (or `spinning') construction is standard, see \cite{Seidel:suspend_LF, AAK, Ganatra-Pomerleano}. \end{proof}

\begin{Lemma}\label{Lem:affine_quartic}
 $W_0$ is an affine quartic surface. 
\end{Lemma}
\begin{proof}
Begin with $T^*S^2 = \{xy+z^2=1\} = A \subset \bC^3$, which has a standard Lefschetz fibration $T^*S^2 \stackrel{\uppi_z}{\longrightarrow} \bC$ by projection to the $z$ co-ordinate, with two critical fibres.  Let $f=z|_A$.  The spinning $\scrS(f)$ is then a $\bC^*$-bundle over $T^*S^2$ with critical fibres along a smooth conic fibre $\bC^* = \uppi_z^{-1}(0) \subset T^*S^2$.  The zero-section $S^2$ which is a matching sphere for the path $[-1,1] \subset \bC_z$ meets the critical fibre in a circle, and each of its two constituent discs spin to a Lagrangian $3$-sphere. These two spheres meet cleanly along the unknotted circle $\bC_z \cap S^2$.  The Morse-Bott surgery of the two 3-spheres is the $S^1\times S^2$ obtained from spinning up a perturbation of the zero-section in $T^*S^2$ which does not meet the critical fibre.  It follows that the total space is the completion of a subdomain symplectomorphic to $W_0$. The total space $\scrS(f)$ is defined by $\{xy+u^2v^2=1\}\subset \bC^4$, which is an affine quartic. 
\end{proof}

The variety $F$ of complete flags in $\bC^3$ is a smooth hypersurface in $\bP^2\times \bP^2$ of bidegree $(1,1)$. 

\begin{Lemma} \label{Lem:affine_flag}
$W_1$ is an affine flag 3-fold. 
\end{Lemma}
\begin{proof}   Fix $F = \{xx'+yy'+zz'=0\} \subset \bP^2\times \bP^2$ and consider the map $(\bP^2 \times \bP^2)\backslash \Theta \to \bP^2$ given by
\[
 ([x:y:z]), ([x':y':z']) \mapsto [xx':yy':zz']
\]
which is defined away from the hexagon of lines $\Theta = \{xx'=yy'=zz'=0\} \subset \bP^2\times \bP^2$.  We pull back a pencil of lines on $\bP^2$, and consider the pencil of $(1,1)$-divisors 
\[
D_{[\uplambda:\upmu]} = \{\uplambda (xx'-zz') + \upmu\, yy' = 0\} \cap F \, \subset F, \qquad [\uplambda:\upmu] \in \bP^1,
\] 
again with base locus $\Theta$. Remove the smooth fibre $D_{[1:0]} \cap F$ from $F$, which is a del Pezzo surface of degree six representing $-K_F/2$. The projection $p = \uplambda / \upmu\colon F\backslash D_{[1:0]} \to \bC$ is then a Morse-Bott-Lefschetz fibration with general fibre $(\bC^*)^2$ as in Example \ref{Ex:MBLefschetz}, and with singular fibres over $\{0, \pm1\}$. The map $p$ is equivariant for a Hamiltonian $T^2$-action 
\[
(\upvartheta,\upphi)\cdot ([x:y:z],[x':y':z']) \ = \ ([e^{i\upvartheta}x:y:e^{i\upphi}z],[e^{-i\upvartheta}x':y':e^{-i\upphi}z'])
\] on $\bP^2\times \bP^2$, which has proper moment map restricted to $W_1$. There are two Lagrangian 3-spheres which fibre, as Morse-Bott matching cycles, over the arcs $[-1,0]$ and $[0,1] \subset \bC$.  The circles which collapse at the critical fibres are those corresponding to $\upvartheta=0$ respectively $\upphi = 0$ at $\pm1$, and $\upvartheta=\upphi$ at $0$.  This shows that the two 3-spheres meet according to the $W_1$-plumbing.  \end{proof}

\begin{Remark} Jonny Evans (private communication) has shown that the union of the two matching spheres form a Lagrangian skeleton of  $W_1$, i.e.\ they are the critical locus of a global plurisubharmonic function.  The weaker fact that $W_1$ embeds into the flag 3-fold follows from `semi-toric' considerations, cf.\ \cite[Figure 10(b)]{Lau}.  This exhibits a singular Lagrangian torus fibration on $F$ in which two 3-spheres meeting cleanly along a circle are mapped to straight arcs in the base of the fibration (the `moment polytope'). By considering the slopes of those arcs in the affine structure in the base, one sees that the local model is that of $W_1$. 
\end{Remark}

\begin{Remark}
We do not know if $W_n$ is affine if $n>1$, but Remark \ref{Rmk:SYZ} shows that $W_n$ embeds as a Stein subdomain of an affine variety.  $W_1$ is a smoothing of the cone $\mathrm{Cone}(-K_S)$ given by collapsing the zero-section in the anticanonical bundle over a del Pezzo surface $S$ diffeomorphic to the 3-point blow-up of $\mathbb{C}P^2$. A smoothing of $\mathrm{Cone}(-K_S^{\otimes n})$ gives an affine model for a plumbing of two Lens spaces.
\end{Remark}

\begin{Remark}
$\bP^3$ contains no Lagrangian sphere \cite{Seidel:graded}, whilst a quadric 3-fold cannot contain two Lagrangian 3-spheres meeting cleanly along a circle (an easy consequence of results in \cite{Smith:HFquadrics}). Degeneration techniques show a hypersurface $H_d \subset \bP^4$ of degree $d\geq 4$, or  a complete intersection $H_{d,d'} \subset \bP^5$ of bidegrees $d,d' \geq 2$, contains two Lagrangian 3-spheres meeting cleanly in a circle. \end{Remark}

\subsection{A digression on knotting}\label{digression section}

The realisation of $W_0$ as an affine quartic yields an attractive local obstruction to knotting. Recall that a \emph{dilation} on a Stein manifold $Y$ is a class $b \in SH^1(Y)$ in the first symplectic cohomology with $\Delta(b)=1$, where $\Delta$ is the $BV$-operator.  If $Y$ admits a dilation, it admits no exact Lagrangian $K(\uppi,1)$ by \cite{Seidel-Solomon}.

\begin{Lemma} 
$W_0$ admits a dilation. 
\end{Lemma}
\begin{proof} We will construct an embedding of $W_0$ into the total space of a Lefschetz fibration over $\bC$ with fibre the $A_3$-Milnor fibre. The result will then follow from \cite[Proposition 7.3]{Seidel-Solomon} together with Viterbo restriction (which takes dilations to dilations).

\medskip
Inside the  four-real-dimensional $A_3$-Milnor fibre $M$ we can find two Lagrangian matching 2-spheres which meet cleanly along a circle: if the critical points of the standard Lefschetz fibration $M \to \CC$  lie at $\{1,2,3,4\} \subset \CC$ and the obvious real matching spheres are labelled $S_1, S_2, S_3$ then we can take the spheres $S_1$ and $S' = \uptau_{S_2}^2(S_3)$. We now construct $Z$ as the total space of a Lefschetz fibration with four singular fibres at $\{i,2i,3i,4i\} \subset \CC$, for which the associated vanishing cycles along paths parallel to $\mathbb{R}_{\geq 0}$ are respectively $S_1, S_1, S', S''$, where $S''$ is chosen so that when the basis of vanishing paths is mutated so as to replace the fourth one with a path passing through $3i/2$ then the corresponding vanishing cycle agrees with $S'$.  In this total space, the interval $[i,2i] \subset i\mathbb{R}$ is a matching path, as is a path from $3i$ to $4i$ which meets the interval $[i,2i]$ in its interior transversely once. The corresponding spheres in the $A_3$-fibre are identified with $S_1$ and $S'$.  In particular, this constructs a pair of Lagrangian 3-spheres in $Z$ which meet cleanly along a circle, hence an embedding of some double bubble $W_k$ into $Z$. It remains to identify $k$.
\medskip

The $A_3$ Milnor fibre admits a circle action compatible with its Lefschetz fibration structure over $\bC$; the spheres $S_i$ and $S'$ are setwise $S^1$-invariant and meet along an orbit. Since this action is compatible with the $Br_4$-action on the Milnor fibre, there is an induced fibrewise action on  $Z$, and the 3-dimensional matching spheres in $Z$ again meet along an orbit.   One can then construct the surgery of the two matching spheres equivariantly to obtain a 3-manifold $K$ which carries an induced $S^1$-action free in a neighbourhood of the surgery locus. Inspection shows that globally the action has no exceptional orbits (with non-trivial finite stabiliser) and two circles of fixed points (coming from points lying at the endpoints of the matching spheres in the $A_3$-fibres, over the intervals of points comprising the matching paths in the base of $Z \to \bC$); moreover the quotient of the action has genus zero and is orientable.  It follows that the surgery $K$ is diffeomorphic to $S^1\times S^2$ from the classification of circle actions on 3-manifolds with non-empty fixed point loci  \cite[Theorem 1, (ii)a]{Raymond}.  We conclude that a neighbourhood in $Z$ of the embedded double bubble plumbing is symplectomorphic to $W_0$.
\end{proof}

\begin{Proposition} \label{Prop:knot_1}
There is no Hamiltonian isotopy of the spheres $Q_i \subset W_0$ to a pair of spheres meeting cleanly in a circle knotted in either component.
\end{Proposition}
\begin{proof}
Suppose such an isotopy exists, yielding an embedding $W_n(\upkappa_1,\upkappa_2) \subset W_0$ with at least one $\upkappa_i$ non-trivial.  Label the core spheres of the knotted clean intersection $Q_i'$. Since the $Q_i'$ are isomorphic in the Fukaya category to the $Q_i$, necessarily $n=0$ from \eqref{eqn:triple_product}. By compatibility of the unit and BV-operator with Viterbo restriction, we see that $W_{0}(\upkappa_1,\upkappa_2)$ also admits a dilation, hence contains no exact Lagrangian $K(\uppi,1)$.  If $\upkappa_2$ is the trivial knot, then the Bott surgery $K$ of the isotoped spheres $Q_i'$  is given by $0$-surgery on $\upkappa_1$. If $\upkappa_1$ is non-trivial, this is aspherical by Gabai's proof of Property $R$ \cite{Gabai:foliations_III}, a contradiction.  So we can suppose both $\upkappa_i$ are non-trivial.  Then the surgery admits an incompressible torus.

\medskip
By a result of Ganatra-Pomerleano \cite[Corollary 5.9]{Ganatra-Pomerleano}, again building on \cite{Kotschick-Neofytidis},  one sees that the surgery $K$ of the $Q_i'$  -- whose cotangent bundle also admits a dilation, by the same restriction argument, and hence a quasi-dilation --  is finitely covered by $S^1\times \Sigma_g$ ($g\geq 1$) or by a connect sum of $S^1\times S^2$'s. In the former case, $K$ is aspherical, which is a contradiction. Thus $K$ is finitely covered by $\#_k (S^1\times S^2)$ for some $k\geq 0$. These are exactly the manifolds with no aspherical summand in their prime decomposition; since $H_1(K;\bQ) \cong \bQ$, we see that $K=(S^1\times S^2) \# M$ for $M$ a connect sum of spherical 3-manifolds. But this means $\uppi_1(K)$ is a free product of $\bZ$ and a collection of finite groups, and this contains no $\bZ\oplus\bZ$ subgroup by Kurosh's theorem, contradicting existence of the incompressible torus. 
\end{proof}

We remark that there is no smooth obstruction to changing the knot type of the intersection.

\medskip
After the first version of this note was circulated, Ganatra and Pomerleano constructed a quasi-dilation on $W_1$  in characteristic 3, which extends\footnote{The first version of this note used \emph{ad hoc}  methods  to prove a weaker result, namely that the only possible non-trivial knot type that could occur in $W_1$ was the trefoil.} Proposition \ref{Prop:knot_1} from $W_0$ to $W_1$.  When $p>1$ is prime, one can rule out some knot types using the classification result Corollary \ref{cor 2 3folds}, along the vein of the topological Corollaries given in the Introduction.

\subsection{Monodromy\label{Sec:twists_1}}

A Lagrangian sphere $L \subset X$ has an associated Dehn twist symplectomorphism $\uptau_L$. Keating \cite{Keating} proved that the Dehn twists $\uptau_{Q_1}$ and $\uptau_{Q_2}$ generate a free group inside the mapping class group $\uppi_0\,\Symp_{ct}(W_{\upeta}(\upkappa_1,\upkappa_2))$. Since the intersection pairing on $H_3(W_n(\upkappa_1,\upkappa_2);\bZ)$ is trivial, the twists $\uptau_{Q_i}$ act trivially on cohomology.  The cone $K = \{Q_1 \stackrel{e}{\longrightarrow} Q_2\}$ has surgery representative a Lagrangian Lens space, hence there is a geometric Dehn twist $\uptau_K$ (cf. \cite{Mak-Wu} for the Lens space case). The Lens space  $K$ represents the diagonal homology class in 
\[
H_3(W_n;\bZ/2)\,  \cong\, \bZ/2 \langle Q_1\rangle  \oplus \bZ/2 \langle Q_2 \rangle
\]
so (even when $n=1$ and $K$ is a sphere) $K$ is not quasi-isomorphic in $\scrF(W_n)$ to the image of $Q_i$ under any element of $\langle \uptau_{Q_1},\uptau_{Q_2}\rangle$, and $\uptau_K$ does not belong to this group. The question arises as to which larger group naturally acts, i.e.\ what is the group $\langle \uptau_{Q_1}, \uptau_{Q_2}, \uptau_K\rangle$.

\medskip
Let $\Br_3 = \langle a,b \, | \, aba=bab \rangle$ denote the braid group on 3 strings.  The kernel of the natural surjection $\Br_3 \to \Sym_3$, which quotients out the subgroup normally generated by $a^2$ and $b^2$, is called the \emph{pure braid group} $\PBr_3$.  One can also consider the preimage of a subgroup $S_2 \leq S_3$, to define the \emph{mixed braid group} $\MBr_3$ of braids which fix one end-point. The pure braid group $\PBr_3$ is generated by either of the triples  $\{a^2,b^2,ab^2a^{-1}\}$  or $\{a^2,b^2, (aba)^2\}$, the mixed braid group $\MBr_3$ by $\{a^2,b^2,(aba)\}$.

\begin{Lemma}\label{Lem:mixed}
There is a natural representation $\uprho_0\colon \MBr_3 \to \uppi_0\,\Symp(W_0)$.
\end{Lemma}
\begin{proof}
The construction of the affine quartic above has input (i) an element of $\Conf_2(\bC)$, defining a degree two polynomial $p(z)$ with distinct roots and hence an affine surface $S_p= \{xy+p(z)=0\}$ symplectically equivalent to $T^*S^2$; and a point $\uplambda \in \bC\backslash \{p^{-1}(0)\}$ defining a smooth conic fibre of the projection $S_p \to \bC_z$ about which one can spin to produce the 3-fold.  One therefore obtains a family of affine quartics over a bundle over $\Conf_2(\bC)\simeq S^1$ with fibre a twice-punctured plane. This parameter space has fundamental group $\MBr_3$.  The representation is then obtained from symplectic parallel transport, well-defined since one can work relative to a family of compactification divisors with fixed topology.

\medskip
The generator $a^2$ of $\MBr_3$ is represented by the monodromy of a loop of affine varieties in which one of the Lefschetz critical points of $T^*S^2 \to \bC$ rotates around the fixed fibre in which we spin to construct the threefold.  One can therefore view this as the monodromy of a family of varieties over a disc, in which the central fibre is singular, obtained by spinning $\{uv=\uppi\}$ along a \emph{singular} conic fibre, so where $\uppi^{-1}(0)$ is a critical fibre of the Lefschetz fibration $\uppi\colon T^*S^2 \to \bC$.  The family of threefolds then acquires a nodal singularity $\{uv=x^2+y^2\}$, so the monodromy is isotopic to the Dehn twist in the corresponding vanishing cycle.  It follows that $\uprho_0$ takes $\{a^2, b^2\}$ to the Dehn twists $\{\uptau_{Q_1}, \uptau_{Q_2}\}$.

\medskip
The generator $(aba)$ of $\MBr_3$ is a lift, under the map $\MBr_3 \to \Br_2$ which forgets the fixed strand, of the generating half-twist exchanging the other two strands. Since the fixed strand corresponds to the location of the fibre over which we spin, this means that it is a lift to $W_0$ of the Dehn twist in the zero-section of $T^*S^2$, under the natural map $W_0 \to T^*S^2$.  
\end{proof}

\begin{Remark} \label{Rmk:pull_back_from_surface}
If $L = S^1\times S^2 \subset X$ is a Lagrangian $S^1\times S^2$, the non-compactly-supported symplectomorphism  $\id \times \uptau_{S^2} \in \Symp(T^*S^1\times T^*S^2)$  need not extend to a global symplectomorphism of $X$;  our situation is special since $W_0$ admits a global map to $T^*S^2$.  This phenomenon is familiar on the B-side of the mirror, where there are objects $\scrE$ such that
\[
\Hom^*(\scrE,\scrE) = \begin{cases} k[x] & * \in \{0,2\} \\ 0 & else. \end{cases}
\]
The existence of an associated autoequivalence is again delicate, because the object $\scrE$ defining the autoequivalence admits an \emph{infinite} filtration by copies of the underlying $(S^1\times S^2)$-type-object $E$, hence need not belong to the `compact' category in question.  
\end{Remark}

Consider a Morse-Bott-Lefschetz fibration $T^*S^3 \cong E_1$ as in Example \ref{Ex:MBLefschetz}, with singular fibres at $\{0,1\}$ with vanishing cycles the usual meridional and longitude curves on the torus (call these $a$ and $b$ say).  One can construct a family of Morse-Bott-Lefschetz fibrations $E_{\uplambda}$ over $\bC^*$ with singular fibres at $0$ and $\uplambda$, and there is a monodromy symplectomorphism $\upphi_1\colon E_1 \to E_1$ associated to parallel transport around the unit circle.  

\begin{Lemma} \label{Lem:local_MBL_twists}
The symplectomorphism $\upphi_1\colon E_1 \to E_1$ is the Dehn twist.
\end{Lemma}

\begin{proof}
We adapt a construction from \cite{Smith-Thomas}, see also \cite{Greer}. Consider a Morse-Bott-Lefschetz fibration with general fibre $(\bC^*)^2$, with two critical fibres $(\bC^*) \times (\bC\vee \bC)$, and with total space containing a matching 3-sphere fibred over an arc between the critical values.  Such a fibration can be obtained as the fibre product of a pair of two-complex-dimensional Lefschetz fibrations $\uppi_i\colon X_i \to \bC$ with general fibre $\bC^*$, each with a single singular fibre $\bC\vee \bC$, and whose critical values do not agree. The half-twist exchanging the critical values is realised by monodromy around a family in which the singular fibre is the self-fibre product 
\[
\{(u,v,x,y) \in \bC^2\times \bC^2 \, | \, \uppi(u,v) = \uppi(x,y)\}
\]
of the standard Lefschetz singularity $\uppi\colon \bC^2 \to \bC$, $\uppi(u,v) = u^2+v^2$. This is again a threefold ordinary double point. 
\end{proof}

This local model admits fibrewise $\bZ/k$-actions which give a corresponding description for the cotangent bundle of a Lens space, and a similar monodromy viewpoint on the Lens space Dehn twist. 
Viewing $T^*S^3 = \{z_0z_1+z_2z_3=1\} \subset \bC^4$, the Morse-Bott-Lefschetz fibration is given by projection to $z_0z_1$, and there is a fibrewise Hamiltonian action of $T^2$ generated by circle actions with weights $(1,-1,0,0)$ and $(0,0,1,-1)$.  The $\bZ/k$ diagonal subgroup acts freely and has quotient $T^*(L(k,1))$. Viewing the original $T^2$ as generated by the first and diagonal circle actions (so the vanishing cycles degenerating at the two Morse-Bott critical points are $a$ and $b-a$), there is a moment map 
\[
(z_0,\ldots,z_3) \ \mapsto \ (|z_0|^2 - |z_1|^2, |z_0|^2 - |z_1|^2 + |z_2|^2 - |z_3|^2)
\]
The $T^2$ action descends to the quotient with moment map 
\[
(z_0,\ldots,z_3) \ \mapsto \ (|z_0|^2 - |z_1|^2, (|z_0|^2 - |z_1|^2 + |z_2|^2 - |z_3|^2)/k)
\]
and with vanishing cycles $a, kb-a$ (cf. Example \ref{Ex:MBLefschetz}).  Let $E_k$ denote this quotient fibration; there is again a symplectomorphism
\[
\upphi_k\colon E_k \to E_k
\]
from parallel transport around the unit circle for the family of fibrations with critical fibres at $0$ and $\uplambda \in \bC^*$. 

\begin{Lemma} \label{Lem:lens_twist}
The full twist monodromy $\upphi_k\colon E_k \to E_k$ is Hamiltonian isotopic to the Lens space Dehn twist.
\end{Lemma}

\begin{proof} The $\bZ/k$-action discussed above exists in the total space of the family considered in Lemma \ref{Lem:local_MBL_twists}. By construction, there is therefore a commuting diagram
\[
\xymatrix{ T^*S^3 \ar[rr]^{\upphi_1} \ar[d] & & T^*S^3 \ar[d] \\ T^*L(k,1) \ar[rr]^{\upphi_k} & & T^*L(k,1) }
\]
with vertical maps the $\bZ/k$ quotients introduced previously. 
Since $\upphi_1$ is the spherical twist, $\upphi_k$ is the Lens space twist up to composition with a deck transformation; but both monodromies are compactly supported. 
\end{proof}

\begin{Lemma}\label{Lem:pure}
For $k\geq 1$, there is a natural  representation  $\uprho_k\colon \PBr_3 \to \uppi_0\,\Symp(W_k)$, whose image is generated by (spherical and Lens space) Dehn twists in the core components and their surgery.
\end{Lemma}
\begin{proof} As in Lemma \ref{Lem:mixed}, we obtain the representation $\uprho_k$ from parallel transport, now via the geometry of families of Morse-Bott-Lefschetz fibrations as in the set-up for Lemma \ref{Lem:local_MBL_twists} but for fibrations with three critical fibres. This immediately shows, when $k=1$, for the representation $\uprho_1$  the generating half-twists $a^2, b^2, ab^2a^{-1}$  correspond to Dehn twists in $Q_1, Q_2, K$, viewed as monodromies of global degenerations.  For higher $k$, an appeal to Lemma \ref{Lem:lens_twist} shows that two of the pure braid group generators which perform a full twist in a pair of generators are taken to spherical twists in the cores, whilst the `third' pure braid group generator acts as the Lens space twist.
\end{proof}

\begin{Remark} \label{Rmk:compactly_supported}
For $k\geq 1$, each element of the image of $\uprho_k$ is isotopic to a compactly supported symplectomorphism (this is not true when $k=0$). 
When $k=1$, the monodromy can also be understood via transport in families of affine flag 3-folds as in Lemma \ref{Lem:affine_flag}. From that viewpoint, it is clear that the whole representation $\uprho_1$ lifts to $\uppi_0\,\Symp_{ct}(W_1)$.
\end{Remark}

\begin{Remark}\label{Rmk:boundary_twist}
Lemma \ref{Lem:affine_flag} shows that the contact boundary of $W_1$ admits a free contact circle action, since it is the circle bundle over a smooth del Pezzo surface.  If $(Y,\upxi)$ is a contact manifold carrying a contact circle action, there is an associated compactly supported symplectomorphism of the symplectization (by suspending the action), which here gives rise to a symplectomorphism $\uptau_{\partial} \in \uppi_0\,\Symp_{ct}(W_1)$ which is a ``boundary twist".   The monodromy $\uprho_1$ takes $(aba)^2$, which generates the centre of $\PBr_3$, to the boundary twist $\uptau_{\partial}$. This follows by considering a Lefschetz pencil of $(1,1)$-divisors in $\bP^2\times\bP^2$  over $\bP^1$, with 3 critical fibres; the global monodromy rotates the normal bundle of the base locus. \end{Remark}

\subsection{A-side Twists}\label{Sec:twists}

The Dehn twist in an exact Lagrangian sphere $L \cong S^3 \subset X$ acts by a twist functor on the Fukaya category:
\begin{equation} \label{eqn:spherical_twist}
\uptau_L \simeq \mathrm{Cone}\, (\hom_{\scrF(X)}(L,\bullet) \otimes L \stackrel{ev}{\longrightarrow} \bullet)
\end{equation}
where \eqref{eqn:spherical_twist} denotes an isomorphism in the category $nu$-$fun(\scrF(X),\scrF(X))$ of non-unital $A_{\infty}$-functors of $\scrF(X)$, see \cite{Seidel:FCPLT}.  The action of a Lens space Dehn twist on the Fukaya category was analysed by Mak and Wu \cite{Mak-Wu}.  They prove that if $L \cong L(p,1) \subset X$ is an exact Lagrangian Lens space and one works over a field $\bK$ of characteristic $p$, then for any other exact Lagrangian $K\subset X$ there is a quasi-isomorphism
\begin{equation} \label{eqn:lens_twist}
\uptau_L(K) \simeq \mathrm{Cone}\, (\hom_{\scrF(X)}(\scrL, K) \otimes_{\bK[\uppi_1(L)]} \scrL \stackrel{ev}{\longrightarrow} K)
\end{equation}
where $\scrL$ denotes the universal local system over $L$ and $ev$ is now a suitable equivariant evaluation map, see \cite{Mak-Wu} for details. We briefly recall that, fixing an isomorphism $\uppi_1(L) = \bZ/p$, there is an isomorphism of $\bK[\uppi_1(L)]$-modules
\[
\hom_{\scrF}^*(\scrL,\scrL) = (C^*(\widehat{L}) \otimes \Hom_{\bK}(\bK[\bZ/p], \bK[\bZ/p]))^{\bZ/p}
\]
where $S^3 = \widehat{L} \to L$ is the universal cover, and we also note that \cite[Lemma 2.13]{Mak-Wu} gives an isomorphism of $\bK$-algebras
\begin{equation} \label{eqn:univ_local_system}
HF^*(\scrL,\scrL) = H^*(S^3) \otimes_{\bK} \bK[\uppi_1(L)].
\end{equation}
\begin{Remark}
It is presumably true, but not yet known, that \eqref{eqn:lens_twist} again arises from an exact triangle of functors.  
\end{Remark}

The $A_\infty$-endomorphism algebra of an object $E$ in a $3$-CY category gives rise to a noncommutative deformation functor, and hence a universal object $\scrE$.  We say that $E$ is fat-spherical if $\Hom^i(\scrE,E)$ is the cohomology of a sphere \cite{Toda}.   The dimension of this sphere can vary; in Remark~\ref{Rmk:pull_back_from_surface} the fat-spherical object is the 2-sphere, whilst below for flopping contractions only the 3-sphere will turn out to be possible.

\medskip
Below it will be a feature of our categories $\scrC_k$ that the universal objects $\scrE$ of our fat spherical objects will always turn out to satisfy
\begin{equation} \label{eqn:fat}
\Hom^i(\scrE,\scrE) = \begin{cases} \bK[x]/(x^n) & i=0,3 \\ 0 & \mathrm{else}, \end{cases}
\end{equation}
and indeed such an $\scrE$ will be mirror to the universal local system $\scrL$ of \eqref{eqn:univ_local_system} (the universal object $\scrE$ respectively $\scrL$ arises via the same iterated mapping cone on $E$ respectively $L$).  The usual association of autoequivalences to spherical objects carries over to the fat-spherical case: any fat-spherical object (or more precisely, its corresponding universal object) determines an autoequivalence, by e.g.\ \cite{Toda}. 

\medskip
If $\charac\,\bK=0$, then the Lens space $L(p,1)$ supports $p$ distinct rank one local systems associated to the unitary characters of $\uppi_1(L)$; each of these determines a spherical object  $S_i \in \scrF(X;\bK)$. These spherical objects are pairwise orthogonal, and so their associated twists commute.  \cite{Mak-Wu} implies that
\begin{itemize}
\item when $\charac(\bK)=p$, the Lens space twist acts on objects of the Fukaya category by the fat spherical twist associated to the object $\scrL$;
\item when $\charac(\bK) = 0$, the Lens space twist acts on objects of $\scrF(X)$  as the product of spherical twists $T_{S_1}\ldots T_{S_p}$.
\end{itemize}

\subsection{Generation}

If $\scrD \subset \scrF(X)$ is a full subcategory, we say $\scrD$ split-generates if every object of $\mathrm{D}^{\uppi}\,\scrF(X)$ is quasi-equivalent to an object of $\mathrm{D}^{\uppi}\,\scrD$.  Let $\scrQ_n$ denote the full subcategory of $\mathrm{D}^{\uppi}\,\scrF(W_n)$ split-generated by the two core spheres $Q_1$ and $Q_2$.  

\begin{Lemma}\label{Lem:no_generates}
Let $n =0$. The core spheres $\{Q_i\}$ do not split-generate $\scrF(W_n)$.
\end{Lemma}

\begin{proof}
Consider the Morse-Bott surgery $K = \{Q_1 \stackrel{\upkappa}{\longrightarrow}Q_2\}$, diffeomorphic to $S^1\times S^2$.  Let $\upxi \to K$ be a rank one local system with non-trivial monodromy; this defines a Lagrangian brane $\upxi\to K$ which is an object of $\scrF(W_0)$. The endomorphisms of this object are given by
\[
HF^*((\upxi \to K), (\upxi \to K)) \cong H^*(K,\End(\upxi)) \cong H^*(K) 
\] so $\upxi \to K$ is a non-zero object. From the description of $W_n$ in terms of a $(\bC^*)^2$-Morse-Bott-Lefschetz fibration in Example \ref{Ex:MBLefschetz}, one sees that one can construct the surgery $K$ so that there are clean intersections $K\cap Q_i \cong S^1$ for each $i$, where moreover the $S^1$ is a homology generator for $H_1(K;\bZ)$. The Floer cohomology
\[
HF^*(Q_i, (\upxi \to K)) \cong H^*(S^1, \upxi)
\]
which vanishes for non-trivial $\upxi$.   We have therefore produced a non-trivial object of $\scrF(W_0)$ which is orthogonal to both the $Q_i$.  
\end{proof}

When $n>1$, the same proof goes through over $\bC$: the cone $K$ is a Lens space which is split-generated by the $Q_i$ when equipped with the trivial local system, but not when equipped with a local system associated to a non-trivial character of $\bZ/n$.

\begin{Lemma}\label{Lem:yes_generates}
Let $n=1$. The core spheres $\{Q_i\}$ split-generate $\scrF(W_1)$.
\end{Lemma}

\begin{proof}
This follows from Remark \ref{Rmk:boundary_twist}  and \cite[Corollary 5.8]{Seidel:FCPLT}; a ``boundary twist", which acts on $\scrF(W_1)$ by a non-trivial shift, can be expressed in terms of Dehn twists in spheres generated by $\scrQ_1$ (compare to Remark \ref{Rmk:boundary_twist}).  For a Lefschetz pencil of hypersurfaces $\{D_t\} \subset Z$ on a projective variety cut out by sections of a line bundle $\scrL \to Z$ with $\scrL^{\otimes -d} = K_Z$, the shift associated to the global monodromy on the affine part $D_0 \backslash (D_0\cap D_{\infty})$ is $4-2d$, cf. \cite[Section 19b]{Seidel:FCPLT}. In our case, we have a pencil of $(1,1)$-divisors in $\bP^2\times \bP^2$ so $d=3$.
\end{proof}

We now give a version of Lemma \ref{Lem:yes_generates} which does not rely on having an affine model for $W_k$ when $k \neq \{0,1\}$. 
\medskip

Let $\scrW(\scrW_k;\bK)$ denote the wrapped Fukaya category of the Stein manifold $W_k$, defined over the field $\bK$. The wrapped category is invariant under not necessarily compactly supported Hamiltonian isotopies, so there is a representation
\begin{equation} \label{eqn:trivial_then_shift}
\uppi_0\,\Symp(W_k) \longrightarrow \Auteq(\scrW(W_k;\bK))/\langle [2]\rangle
\end{equation}
into the quotient of the autoequivalence group by twice the shift. This lifts to a representation of the central $\bZ$-extension $\uppi_0\,\Symp_{gr}(W_k)$ of isotopy classes of graded symplectomorphisms into $\Auteq(\scrW_k;\bK)$. The following result should be compared to Remark \ref{Rmk:boundary_twist}.

\begin{Lemma} \label{Lem:Morse_bott_boundary_twist}
Suppose $k=1$ and $\charac\,\bK =0$ or $k>2$ is prime and $\bK$ has characteristic $k$. 
The centre of the pure braid group $\PBr_3$ acts on $\scrW(W_k;\bK)$ by a non-trivial negative shift.
\end{Lemma}

\begin{proof}
Working with not necessarily compactly supported representatives coming from parallel transport, the central element of the braid group is represented by the `boundary twist' in $\PBr_3 \simeq \uppi_0\,\Diff(D^2; \{p,q,r\})$.  Thus $\uprho$ is geometrically a lift to the total space of the Morse-Bott-Lefschetz fibration $W_k$ of the Dehn twist near infinity in the base; undoing that Dehn twist by an isotopy which rotates infinity clockwise shows that $\uprho$ is trivial in $\uppi_0\,\Symp(W_k)$.  Then \eqref{eqn:trivial_then_shift} shows that it acts on the wrapped category by an even shift. The value of the shift can then be checked on any object. When $k=1$, we know the shift is $[-2]$ by Lemma \ref{Lem:yes_generates}.  Away from $k=1$, we argue that the shift is still $[-2]$ as follows. 
\medskip

Following the convention from \cite{Seidel:graded},  the Dehn twist in a Lagrangian sphere has a canonical lift to a graded symplectomorphism\footnote{Gradings in $(X,\omega)$ are determined by a choice of infinite cyclic cover of the Lagrangian Grassmannian bundle of $TX$ which is fibrewise universal, and the canonical graded lift is uniquely determined by acting trivially on the fibres of this cover near infinity in $X$.}; the same is true for a Lens space twist.  We will continue to write $\uptau_L$ for this canonical graded lift. In both cases, the canonical lift shifts the object itself by $[-2]$, cf. \cite[Lemma 5.7]{Seidel:graded}. We  have fixed gradings on the core spheres $Q_1$ and $Q_2$ so that 
\begin{equation} \label{eqn:symmetric_for_us}
HF^*(Q_1, Q_2) = H^*(S^1)[-1]  = HF^*(Q_2,Q_1)
\end{equation}
are both concentrated in degrees $1$ and $2$. Furthermore, we know from the discussion around \eqref{eq:Lhandle} that the Lagrange surgeries $K=K_1$ (fibred over an arc in the upper half-plane in the Morse-Bott-Lefschetz model) and $K' = K_2$ (fibred over an arc in the lower half-plane) are  canonically graded by their presentation as Bott surgeries, with phase functions which extend those on the $Q_i$. The analogue of \cite[Lemma 5.8]{Seidel:graded} in this context,  with exactly the same proof,  is then that there are isotopies
\begin{equation} \label{eqn:graded_isotopy}
\uptau_{Q_1}(K_2) = K_1 \qquad \uptau_{Q_2}(K_1 ) = K_2
\end{equation}
as \emph{graded} Lagrangians.  Note that our situation is more symmetric than the general case treated by Seidel in \emph{op. cit.} because of the grading symmetry in \eqref{eqn:symmetric_for_us}. 
 \medskip
  
With $s_1$ and $s_2$ being the usual generators of the braid group, we have
\[
\uptau_{Q_1} = s_1^2, \ \uptau_{Q_2} = s_2^2,  \ \uptau_{K_2} = s_2 s_1^2 s_2^{-1}
\] 
(the third being the Lens space twist in the $K_2$-surgery). The centre $\uprho$ of the pure braid group is given by 
\[
\uprho = \uptau_{Q_1} \circ \uptau_{K_2} \circ \uptau_{Q_2}
\]
where we have broken symmetry by ordering $Q_1$ and $Q_2$ which accounts for the appearance of a choice of surgery in the middle. We compute the shift by applying this to the object $A = \uptau_{Q_2}^{-1}(K_2)$; this is mapped to $\uptau_{Q_1}(K_2[-2])$, using that $\uptau_{K_2}(K_2) = K_2[-2]$ as commented above.  On the other hand, we already know this image must be some even shift $A[2k]$ of $A$ itself. To prove $k=-1$  it remains to show that $\uptau_{Q_1}(K_2[-2]) = \uptau_{Q_2}^{-1}(K_2)[-2]$, or equivalently that $\uptau_{Q_2}\uptau_{Q_1}(K_2) = K_2$, which follows from \eqref{eqn:graded_isotopy}.
\end{proof}

\begin{Corollary} \label{Cor:generates}
Let $k>2$ be prime and suppose $\charac\,\bK = k$. Then the core spheres $Q_i$ split-generate the compact category $\scrF(W_k;\bK)$. 
\end{Corollary}

\begin{proof}
Whenever a composition of (fat) spherical twists acts by a shift, the corresponding objects split-generate, compare to \cite{Seidel:FCPLT}.  This means the core spheres together with the fat-spherical object (whose universal object is) $\scrL$ underlying the Lens space twist split-generate.  In characteristic $k$, the Lens space $L=L(k,1)$ has a non-trivial class  $\upeta \in HF^1(L,L) = \bK$ and hence a graded self-extension $L \stackrel{\upeta}{\longrightarrow}L$.  When $k=2$, this represents the universal local system $\scrL$ associated to $L$. More generally, the universal local system $\scrL$ over $L$ has endomorphisms as in \eqref{eqn:fat}; it determines the underlying fat-spherical object $L$ of the associated autoequivalence, and the general theory of fat-spherical objects (see e.g.\ \cite{Kawamata, Mak-Wu}) implies that $\scrL$ is filtered by copies of $L$.  Therefore the core spheres together with $L$ split-generate.   But the Lens space itself is manifestly generated by the cores.
\end{proof}

\begin{Remark} \label{rmk:char_0_needs_local_systems} When $k>2$ and $\charac\,\bK=0$ the same argument  and the description of the Lens space twist at the end of Section \ref{Sec:twists} shows that $\scrF(W_k)$ is split-generated by the cores $Q_i$ and the collection of branes given by equipping the Lens space with a non-trivial rank one local system. 
\end{Remark}

\begin{Remark} \label{Rmk:as_if_deformed}
If $k>1$, then the Lens space Dehn twist $\uptau_L \in \uppi_0\,\Symp_{ct}(W_k)$ acts on $\scrF(W_k;\bC)$ by a product of spherical twists.  However, the individual twist equivalences $T_{S_i} \in \Auteq(\scrF(W_k;\bC))$ have no obvious geometric origin, i.e.\ they are not visibly in the image of $\uppi_0\,\Symp_{ct}(W_k)$. \end{Remark}

\begin{Remark} Proposition \ref{Prop:they_are_quivers} gives an equivalence over $\bC$ 
\[
\scrQ_1 \simeq \scrQ_k \qquad k\geq 2
\]
However, this is \emph{not} $\PBr_3$-equivariant for the parallel transport actions constructed above. Indeed, the generators $a^2,b^2,ab^2a^{-1}$ of $\PBr_3$ act by spherical twists on $\scrQ_1$, but -- as autoequivalences of the compact category $\Tw\,\scrF(W_k)$ -- two act by spherical twists and one by a product of $k$ spherical twists in orthogonal objects for $k>1$. When $k>1$, the representation 
\[
\PBr_3 \to \Auteq(\Tw\,\scrF(W_k))
\] 
is not especially natural from the viewpoint of the subcategory $\Tw\,\scrQ_k \subset \Tw\,\scrF(W_k)$, since that subcategory does not contain the objects given by the Lens space with non-trivial local systems. The equivalence $\scrQ_1 \simeq \scrQ_k$ induces a different $\PBr_3$ action on $\scrQ_k$, in which all three `standard' generators are spherical twists: the Lens space twist is replaced just by $T_{S_1}$ where $S_1$ is the spherical object corresponding to the trivial local system.\end{Remark}

\subsection{B-side twists}
We now revert back to the specific crepant resolution $Y_k\to\Spec R_k$ constructed in Lemma~\ref{specific Yn}, for any $\bK$, and any $k\geq 0$.

\medskip
For any $\bK$ and any $k\geq 0$, the threefold $Y_k$ has spherical objects $\scrS_i = \scrO_{\mathrm{C}_i}(-1)$ for $i=1,2$ corresponding to the two $(-1,-1)$-curves in Lemma~\ref{specific Yn}.  Both give rise to spherical twists.  As on the A-side, the third generator depends on whether $k=0$ or not.

\begin{itemize}
\item When $k=0$, and $\charac\,\bK=0$, we can consider the twist autoequivalence over the affine line $\bK[x]$ considered in \cite[Example 6.4]{DW-noncommutative_enhancements}, which we denote $\uptau_{\mathrm{mix}}$.

\item When $k>0$, regardless of the characteristic, since the mutation functors give derived equivalences between the NCCRs it can be seen directly that the mapping cone $\scrS_1 \to \scrS_2$ on the non-trivial morphism is fat-spherical.  We write $\uptau_{\fat}$ for its associated twist, which is spherical over $\bK[x]/x^k$.  Hence this third generator is a `genuine' spherical object if and only if $k=1$.

\end{itemize}

\medskip
Paralleling the monodromy actions by symplectomorphisms gives the following.

\begin{Lemma}\label{B-side what is hom}
With notation as above:
\begin{enumerate}
\item  The category $\scrC_0$ admits an action of the mixed braid group $\MBr_3$.
\item For $k>0$, the category $\scrC_k$ admits an action of the pure braid group $\PBr_3$.
\end{enumerate}
In both cases, one may take the spherical twists in the objects $\scrO_{\mathrm{C}_i}(-1)$ to correspond to the elements $a^2$ and $b^2$.  When $k=0$ the third generator is $\uptau_{\mathrm{mix}}$ which corresponds to $aba$, and when $k>0$ the third generator is  $\uptau_{\fat}$ which corresponds to $(aba)^2$.
\end{Lemma}
\begin{proof} 
Part (1) is a special case of results in \cite{Donovan-Segal}; see also \cite[Example 6.4]{DW-noncommutative_enhancements}.  The generators of the mixed braid group action are the spherical twists in the two $(-1,-1)$-curves in the resolution, and the family twist $\uptau_{\mathrm{mix}}$. Since the resolution curves move in a surface, the local geometry is modelled on the product of the resolution of a surface singularity with $\Spec\,\bC[x]$, and the family twist is essentially the pullback of the twist from one factor.  In terms of the usual braid group generators, these correspond to $\{a^2,b^2,aba\}$.

\medskip
When $k>0$ and $\charac\,\bK=0$ this is contained in \cite{DW-twists_and_braids}; the pure braid group generators can be taken either to be given by the spherical twists in the two resolution curves and the spherical twist in their cone, or the spherical twists in the resolution curves and the autoequivalence associated to their union.  These correspond to the braid group elements $\{a^2,b^2,ab^2a^{-1}\}$ and $\{a^2,b^2,(aba)^2\}$ respectively.  When $\charac\,\bK=p$ these results still hold, since they only rely on the mutation functors and the tilting order, which still exist.
\end{proof}

\begin{Proposition} \label{entwine}
The following statements hold.
\begin{enumerate}
\item\label{entwine 1} If $\bK=\mathbb{C}$, then the equivalence of categories $\scrQ_1 \simeq \scrC_1$ of Theorem \ref{thm:main} entwines the pure braid group actions, whilst the equivalence $\scrQ_0 \simeq \scrC_0$ entwines the mixed braid group actions on the level of objects.
\item\label{entwine 2} If $\charac \,\bK=p>2$ then the equivalence $\scrQ_p \simeq \scrC_{p}$ of Theorem \ref{thm:main} entwines the action of the pure braid group actions, on the level of objects.
\end{enumerate}
\end{Proposition}
\begin{proof} In Lemma \ref{Lem:pure} and \ref{Lem:mixed} we showed that the representations $\MBr_3\to\uppi_0\,\Symp(W_0)$ respectively $\PBr_3 \to \uppi_0\,\Symp(W_k)$ take the generators $a^2,b^2$ to the Dehn twists in the core spheres $\uptau_{Q_1},\uptau_{Q_2}$. The equivalence of Theorem \ref{thm:main} take the sphericals $Q_i$ to the simple modules corresponding to the two vertices of the cyclic quiver.  In turn, these go to the sheaves $\scrO_{\mathrm{C}_i}(-1)$, hence take the twists $\uptau_{Q_i}$ to the spherical twists in these sheaves.  

\medskip
For $W_1$, when $\bK=\mathbb{C}$ the third generator is again a spherical twist  on both sides (now in the cone), and the result follows.  When $\charac\,\bK=p$, the third generator on both sides is now fat spherical.  However, the results of  \cite{Mak-Wu} are sufficient only to infer that the Lens space twist $\uptau_L$ and the appropriate mapping cone agree object-wise.  Hence, we can only conclude that the actions entwine on the level of objects.

\medskip
For $W_0$, the generator $aba$ of $\MBr_3$  is the lift of the Dehn twist in $T^*S^2$ for $W_0$ and is the `family spherical twist' for $\scrC_0$.  In both cases, this is a two-dimensional twist which exchanges the two simples and shifts each of them by $[-1]$, which is sufficient to determine the action on objects. 
\end{proof}

\begin{Remark} One could probably upgrade the statement for $\scrC_0, \scrQ_0$ to a full entwining of actions by proving mirror symmetry for slightly larger categories which incorporates the non-compact object defining the appropriate twist, cf. Remark \ref{Rmk:pull_back_from_surface}. \end{Remark}

\subsection{Faithfulness}\label{Sec:faithful}
Faithfulness of the pure braid group action for flopping contractions was first established in \cite{Hirano-Wemyss}, for the case $\bK=\mathbb{C}$. When $\charac\,\bK=p$ the faithfulness still holds, since all techniques used only rely on properties of noncommutative resolutions, mutation functors, torsion pairs and the tilting order, which are characteristic independent.

\medskip
Recall that the group $\Symp(W_p)$ has a central $\bZ$-extension of `graded' symplectomorphisms, and there is a corresponding graded symplectic mapping class group $\uppi_0\,\Symp_{gr}(W_p)$ which acts by autoequivalences on the compact Fukaya category (the usual mapping class group acts by autoequivalences up to shift).

\begin{Corollary}\label{Cor:faithful}
If $p=1$ or $p>2$ is prime, then the natural representation $\uprho_p\colon  \PBr_3 \to \uppi_0\,\Symp_{gr}(W_p)$ is faithful.
 \end{Corollary}
\begin{proof}
Appealing to Corollary~\ref{comparison A and B}, $\scrQ_1 \simeq \scrC_1$ in characteristic zero, or $\scrQ_p \simeq \scrC_p$ in characteristic $p>2$.  Either way, by \cite{Hirano-Wemyss} and the above paragraph, the natural homomorphism
\[
\PBr_3\to\Auteq\scrC_k
\] 
from Lemma~\ref{B-side what is hom} is injective.  Even although in Proposition~\ref{entwine} we don't know functorially that the pure braid actions entwine under the equivalence $\scrQ_p \simeq \scrC_p$, it is still true that the natural homomorphism
\begin{equation}
\PBr_3\to\Auteq\scrQ_p\label{A-side injective}
\end{equation}
is injective.  This is since injectivity in \cite{Hirano-Wemyss} is determined by Deligne normal form and dimensions of Ext groups, and this information is determined entirely by the action of the functors on objects.  Now by Lemma~\ref{Lem:pure} the injective homomorphism \eqref{A-side injective} factors as
\[
\PBr_3\to\uppi_0\,\Symp_{gr}(W_k)\to\Auteq \scrQ_p,
\]
and so it follows that the first map must also be injective.
\end{proof}

Using Remark \ref{Rmk:compactly_supported}, one can improve this when $k=1$ to a faithful representation $\PBr_3 \to \uppi_0\Symp_{ct}(W_1)$.

\section{The wrapped Fukaya category and Koszul Duality}

In this section we lift the equivalence $\scrQ_k \simeq \scrC_k$, over the appropriate coefficient field, to yield a new description and comparison of the wrapped category of $W_k$ and the relative singularity category of $Y_k$.  This involves $A_{\infty}$-Koszul duality, which we now briefly recall, following \cite{LPWZ}, \cite{Kalck-Yang} and \cite{Ekholm-Lekili}.  
\medskip

For an $A_{\infty}$-category $\scrA$, recall that we have categories of proper and perfect modules
\[
\scrA^{per} \subset \scrA^{mod} \supset \scrA^{prop}.
\]
We say $\scrA$ is homologically smooth if the diagonal bimodule is perfect; in this case, $\scrA^{prop} \subset \scrA^{per}$.  
\medskip

\subsection{A-side Koszul duality}\label{A-side Koszul duality}
Consider a Stein manifold $X$ containing distinguished compact Lagrangians $L_i$ and dual non-compact Lagrangian discs $T_i$; we will always consider a plumbing with the $L_i$ being cores and the $T_i$ co-cores (cotangent fibres).  Note that, in this setting, the $\{T_i\}$ are known to generate the wrapped Fukaya category $\scrW(X)$ by \cite{CDRGG}. We have compact and wrapped endomorphism algebras $\scrQ = \oplus_{i,j} CF^*(L_i,L_j)$ and $\scrW = \oplus_{i,j} CW^*(T_i,T_j)$, which are $A_{\infty}$-algebras naturally augmented over the semisimple ring $\bfk = \oplus_i \bK$ generated by the identity idempotents of the objects.   It is known that the wrapped Fukaya category of any Stein manifold is homologically smooth, by deep results of Chantraine, Dmitroglou Rizell, Ghiggini and Golovko \cite{CDRGG}.
\medskip

We impose the following \emph{hypotheses}: $H^*(\scrQ)$ is concentrated in non-negative degrees, $H^0(\scrQ) = \bfk$, and $H^*(\scrW)$ is concentrated in non-positive degrees. The first two hold for the double bubble plumbings by Theorem \ref{thm:main}, whilst the last holds for the wrapped category of a double bubble by \cite{Abouzaid-Smith_plumbing}.
\medskip

Given that, Ekholm and Lekili prove that there is an equivalence
\begin{equation} \label{eqn:this_way_works}
\scrQ \simeq \RHom_{\scrW}(\bfk,\bfk)
\end{equation}
and there is a natural map to a certain completion
\begin{equation} \label{eqn:may_need_completion}
\scrW \to \widehat{\scrW} \simeq \RHom_{\scrQ}(\bfk,\bfk).
\end{equation}
When no completion is necessary, we say the compact and wrapped categories are Koszul dual; in general, the compact category is dual to $\widehat{\scrW}$. 
\medskip

Let $\scrA$ and $\scrB$ be $A_{\infty}$-categories and $F\colon \scrA \to \scrB$ a fully faithful functor. Write $\Gamma(F)$ for the graph bimodule of $F$, so
\[
\Gamma(F)(Y,X) \colonequals hom_{\scrB}(Y, F(X)), \qquad X \in \mathrm{Ob}\, \scrA, \ Y \in \mathrm{Ob}\, \scrB.
\]  
Let $\scrY$ denote the left Yoneda functor $\scrB^{op} \to (\scrB^{op})^{mod}$. 
We will make use of the following result of Lekili and Ueda.  

\begin{Proposition}\label{Prop:LU}
Suppose that: 
\begin{enumerate}
\item $\scrA$ is generated by finitely many spherical objects $\{S_i\}$;
\item For any $Y, Y'$ in $\scrB$ the complex $hom_{\scrB}(Y,Y')$ is bounded above (resp. below);
\item For any $X\in \scrA$ and $Y \in \scrB$ the complex $hom_{\scrB}(Y,F(X))$ of $\bK$-modules is perfect; 
\item There is a quasi-equivalence of functors $T_{F(S_1)}\circ\cdots\circ T_{F(S_n)} \simeq [m]$ for a non-trivial negative (resp. positive) integer $m$.
\end{enumerate}
Then the functor $\Gamma(F)\otimes_{\scrB} \scrY(-): \scrB^{op} \to (\scrA^{op})^{mod}$ is fully faithful.
\end{Proposition}

\begin{proof} This is exactly \cite[Proposition 6.5]{LU}. \end{proof}

\begin{Proposition}\label{A-side Koszul main}
Let $k=1$ and $\bK = \bC$ or $k>2$ be prime and $\charac\,\bK  = k$. Then $\scrQ_k$ is Koszul dual to $\scrW_k$.
\end{Proposition}

\begin{proof}
Note that, by reversing the sign of the symplectic form and exchanging left and right modules, there is an equivalence between the exact Fukaya category and its opposite, see \cite[Appendix A]{Sheridan:pants}. By Corollary \ref{Cor:generates} the category $\scrA=\scrQ_k$ is generated by the spherical objects given by the cores, which furthermore split-generate if we work in the right characteristic.  The dual collection of cotangent fibres $T_1\oplus T_2$ gives a generating subcategory $\scrB=\scrW_k \subset \scrW(W_k)$, which satisfies the condition that complexes $hom_{\scrB}$ are bounded above as discussed previously. The crucial hypothesis on a product of twists giving a shift is given by Lemma  \ref{Lem:Morse_bott_boundary_twist}.
Applying Proposition \ref{Prop:LU}, we have a fully faithful functor $\scrW_k \to (\scrQ_k)^{mod}$, which sends the generator $T_1\oplus T_2$ to $\bfk$. This implies the result, c.f.\ the proof of \cite[Theorem 6.11]{LU}.
\end{proof}

\begin{Remark}\label{A-side Kosz issue}
 A sufficient criterion for $\scrQ$ and $\scrW$ to be Koszul dual (no completion required) is that $H^*(\scrW)$  is concentrated in non-positive degrees and finite-dimensional in each degree; in fact, an argument due to Kalck-Yang \cite{Kalck-Yang} show that it is even sufficient to have $H^0(\scrW)$ be finite-dimensional.  The double-bubble plumbing $W_k$ can be presented as the result of Morse-Bott surgery on $T^*S^3$, attaching the cotangent bundle of a solid torus to the Legendrian conormal torus of an unknot $S^1 \subset S^3$.   This gives a description of the wrapped Floer complex of the cocores in terms of binormal chords between an unknot and a distinct point (corresponding to the handle that is attached to yield $T^*S^3$ from $T^*\bR^3$). Assuming an extension of the `surgery' relation between Lagrangian and Legendrian invariants from \cite{BEE,Ekholm-Lekili} from the case of handle attachment to Morse-Bott handle attachment, one can presumably use such a presentation to establish finite-dimensionality of $H^0(\scrW)$ directly. This would yield the Koszul duality statement without requiring the detour via Lemma \ref{Lem:Morse_bott_boundary_twist}.  In contrast, both $CW^0$ and $HW^0$ are infinite-dimensional for the plumbing $W_0$. 
\end{Remark}

\subsection{B-side Koszul duality}\label{B-side Koszul duality}
For any $k\geq 1$ consider again the flopping contraction $Y_k \to \Spec(R_k)$ over a field $\bK$.  In this subsection we will show that $\scrC_k$ is Koszul dual to the relative singularity $\Db(\coh Y_k)/\langle \scrO_{Y_k}\rangle$, following \cite{Kalck-Yang}. An advantage of the B-side is that this Koszul duality works over all $k\geq 1$ and all $\bK$. 

\medskip
Since we are working algebraically (instead of complete locally as in \cite{Kalck-Yang}), we must first resolve the slightly subtle issue of whether our categories are idempotent complete. This is taken care of in the following two results.

\begin{Lemma}\label{general idem comp}
Suppose that $R$ is a 3-sCY excellent normal domain with a unique singular point, which admits an NCCR $\End_R(M)$ where $M$ decomposes into  rank one $R$-modules. Then $\underline{\mathrm{CM}}\,R\to \underline{\mathrm{CM}}\,\widehat{R}$ is an equivalence, and in particular, $\underline{\mathrm{CM}}\,R$ is idempotent complete.
\end{Lemma}
\begin{proof}
Since the assumptions imply that $R$ is an isolated G-ring, it is known generally \cite[A.1]{KMV} that the completion functor $F\colon \underline{\mathrm{CM}}\,R\to \underline{\mathrm{CM}}\,\widehat{R}$ is an equivalence up to summands, so we just need to prove that $F$ is essentially surjective. Consider an arbitrary $X\in\underline{\mathrm{CM}}\,\widehat{R}$.  Since NCCRs pass to the completion $F{M}$ is a cluster tilting object \cite[5.5]{Iyama-Wemyss-MMAspaper}, so by \cite[4.12]{Iyama-Wemyss-MMAspaper} there exists a triangle
\[
Y_0\xrightarrow{h} Y_1\to X\to
\]
with $Y_i\in\mathrm{add}\, F{M}$. By assumption $M$ decomposes into rank one modules  $M\cong\bigoplus L_i$ say, so $FM\cong\bigoplus FL_i$.  Since each $FL_i$ has rank one, each is still indecomposable.  Hence by Krull--Schimdt both $Y_i$ are direct sums of the form $\bigoplus (FL_i)^{\oplus a_i}$, and thus are in the image of $F$.  Say $Y_i=FM_i$.  We already know that $F$ is fully faithful, so $h=Fg$ for some $g\colon M_0\to M_1$.  Set $Y=\mathrm{Cone}\,g$, then it follows $FY\cong X$, since both are the cone on $h=Fg$.
\end{proof}

\begin{cor}\label{rel sing idem compl}
For all $k\geq 1$, and any $\bK$, the singularity category $\underline{\mathrm{CM}}\,R_k$ and the relative singularity category $\Db(\coh Y_k)/\langle \scrO_{Y_k}\rangle$ are both idempotent complete.
\end{cor}
\begin{proof}
Since $R_k$ admits an NCCR by Proposition~\ref{Rk admits NCCR} which is clearly a direct sum of rank one $R_k$-modules, using Lemma~\ref{general idem comp} we see that $\underline{\mathrm{CM}}\,R_k$ is idempotent complete.  The fact that the relative singularity is idempotent complete follows immediately from this using Schlichting's negative K-Theory (see \cite[Theorem 3.2]{Burban-Kalck}).
\end{proof}

It then follows that the category $\scrC_k$ and the relative singularity category are Koszul dual.

\begin{Proposition}\label{B-side Koszul main}
For any $k\geq 1$ and any $\bK$, $\scrC_k\simeq\per(A)$ and $\Db(\coh Y_k)/\langle \scrO_{Y_k}\rangle\simeq\per(A^!)$, with $A^{!!}=A$ up to quasi-isomorphism.
\end{Proposition}
\begin{proof}
By \cite[Corollary 2.12(a)]{Kalck-Yang}, there is a DG-algebra $B$ such that $\per(B)$ is equivalent to the idempotent completion of the relative singularity category.  Hence, using Corollary~\ref{rel sing idem compl}, $\per(B)\simeq\Db(\coh Y_k)/\langle \scrO_{Y_k}\rangle$ itself.

\medskip
Consider the $A_\infty$-algebra $A=\mathbf{R}\End_\Lambda(\scrS)$.  We already know that $\scrC_k\simeq\per(A)$,  and further that $A^{!!}=A$ (see e.g.\ \cite[2.7]{Kalck-Yang}), so it suffices to show that $A^!=B$. It is easy to check that  $B^! = A$ up to quasi-isomorphism (see e.g.\ \cite[2.6]{Kalck-Yang}), thus $B^{!!} = A^!$.  Hence it remains to justify why $B = B^{!!}$, which is much harder.

\medskip
As in the A-side (Remark~\ref{A-side Kosz issue}), the difficultly is in proving that the non-positive DGA $B$ is finite dimensional in each degree; given this, the fact $B^{!!}=B$ is standard \cite[Theorem 2.8(b)]{Kalck-Yang}.  The key point here is that we need the auxiliary data in the form of the existence of the NCCR $\Lambda$, which allows us to say that all $\Lambda/\Lambda e\Lambda$-modules have finite projective dimension as a $\Lambda$-module.   We can then appeal directly to \cite[Corollary 2.13(i)]{Kalck-Yang} to deduce that $B$ is finite dimensional in each degree, with the two main points being that $H^0(B)$ is the contraction algebra, and thus is finite dimensional since $k\geq 1$, and $\mathrm{D}_{fd}(B)\subseteq \per(B)$ due to the existence of the NCCR.
\end{proof}

 \subsection{Matching the A- and B-sides}
The previous subsections \S\ref{A-side Koszul duality} and \S\ref{B-side Koszul duality} combine to give the following, which lifts Theorem~\ref{thm:main} to include non-compact objects.

\begin{Corollary}\label{Koszul main text}
For $k=1$ and $\bK=\bC$ or $k>2$ prime and $\bK$ of characteristic $k$, there is an equivalence of categories $\scrW(W_k;\bK) \simeq \Db(\coh Y_k)/\langle \scrO_{Y_k}\rangle$.
\end{Corollary}
\begin{proof}
Corollary~\ref{comparison A and B} is proved by establishing that $\scrQ_k\simeq\per(A)\simeq\scrC_k$.  Hence combining Propositions~\ref{A-side Koszul main} and \ref{B-side Koszul main} it follows that $\scrW(W_k;\bK) \simeq \per(A^!)\simeq\Db(\coh Y_k)/\langle \scrO_{Y_k}\rangle$.
\end{proof}

\begin{Corollary} \label{Cor:wrapped_is_Ginzburg} Suppose $k=1$ and $\bK=\bC$ or $k>2$ is prime and $\bK$ has characteristic $k$. The wrapped Floer cohomology algebra $HW^*(T_{q_0}^* \oplus T_{q_1}^*, T_{q_0}^* \oplus T_{q_1}^*; \bK)$ of the cocores to the core components is the Ginzburg $dg$-algebra of the potential $(ef)^{k+1}$ on the two-cycle quiver.
\end{Corollary}

We remind the reader of our standing abuse of notation in the finite characteristic case, where the Ginzburg algebra is not strictly well-defined given issues with cyclic permutation and differentiation.  But the coherent interpretation of the statement is  correct,  in particular, $HW^0$ is the corresponding Jacobian algebra, i.e.\ the quotient of the path algebra by the ideal killing paths of length $2k+1$.

\section{Classification of Spherical Objects}

In this section, we consider the 3-fold flopping contraction $f\colon Y_k\to\Spec R_k$, with $k\geq 1$, over any $\bK$, and we classify all fat-spherical objects in the category $\scrC_k$.  Our main result is that there are only three, up to shift and up to the action of the pure braid group.

\medskip
To ease notation, throughout this section we fix some $k\geq 1$ and henceforth drop all mention of it from the notation.  The flopping contraction $f\colon Y\to\Spec R$ has two $(-1,-1)$-curves above the origin, say $\Curve_1$ and $\Curve_2$, meeting at a point and as in \S\ref{sect:NCres}, there is an equivalence
\begin{equation}
\Psi\colon\Db(\coh Y)\xrightarrow{\sim}\Db(\mod\Lambda)\label{NCCR Db}
\end{equation}
where $\Lambda$ is the NCCR from Lemma~\ref{present NCCR}.  There is an idempotent $e\in\Lambda$, corresponding to the vertex labelled $0$, such that $\Lambda_{\con}=\Lambda/\langle e\rangle$ is the algebra
\begin{equation}
\begin{array}{c}
\begin{tikzpicture}[scale=0.8]
\node (A) at (0,0) {$\scriptstyle 1$};
\node (B) at (2,0) {$\scriptstyle 2$};
\draw[->, bend left] (A) to node[above]{$\scriptstyle a$} (B);
\draw[->, bend left] (B) to node[below]{$\scriptstyle b$} (A);
\end{tikzpicture}
\end{array}\label{general Acon}
\end{equation}
such that $(ab)^ka=0$ and $b(ab)^k=0$. Since $f$ is a flopping contraction, we have
\[
\scrC= \{ a\in\Db(\coh Y)\mid \mathbf{R}f_*a=0\},
\]
and across the fixed equivalence \eqref{NCCR Db}, $\scrC$ is equivalent to the subcategory of $\Db(\mod\Lambda)$ consisting of those complexes whose cohomology groups are all annihilated by $e$, or in other words, those complexes whose cohomology groups are $\Lambda_{\con}$-modules.  

\begin{Remark} \label{rmk:t-structure}
There are various t-structures on $\scrC$.  However, throughout this section, for $a\in\scrC$ we will always write $\mathrm{H}^*(-)$ for cohomology with respect to the standard t-structure in $\Db(\mod\Lambda)$.  Equivalently, we could always work in $\Db(\coh Y)$, and take cohomology with respect to the perverse t-structure (with perversity $p=-1$).  Further, for $a\in\mod\Lambda_{\con}$, we will often write $\Hom_{\scrC}(a,a)$ for $\Hom_{\scrC}(\Psi^{-1}a,\Psi^{-1}a)$, which is the same as $\Hom_\Lambda(a,a)$. 
\end{Remark}

A key feature of below is that whilst the normal bundle of a curve being flopped does not change, the normal bundle of other curves it intersects with can change.  Algebraically, this corresponds to the fact that the quiver mutation of \eqref{general Acon} changes the algebra (and introduces a loop) if and only if $k>1$.  Thus, below the complexity of our argument increases slightly in the $k>1$ case, although the main ideas remain the same.

\subsection{Mutation Functors\label{Subsec:mutation}}
We recall next the \emph{mutation functors} $\Phi_1$ and $\Phi_2$ corresponding to mutation between the NCCRs of $R$ at vertex $1$, and vertex $2$ respectively.  These form a representation of the Deligne groupoid, as follows:
\[
\begin{tikzpicture}[scale=1.3,bend angle=15, looseness=1,>=stealth]
\coordinate (A1) at (135:2cm);
\coordinate (A2) at (-45:2cm);
\coordinate (B1) at (153.435:2cm);
\coordinate (B2) at (-26.565:2cm);
\draw[red!30] (A1) -- (A2);
\draw[black!30] (-2,0)--(2,0);
\draw[black!30] (0,-1.8)--(0,1.8);
\node (C+) at (45:1.5cm) {$\scriptstyle\Db(\Lambda)$};
\node (C1) at (112.5:1.5cm) {$\scriptstyle\Db(\Lambda_1)$};
\node (C2) at (157.5:1.5cm){$\scriptstyle\Db(\Lambda_{12})$}; 
\node (C-) at (225:1.5cm) {$\scriptstyle\Db(\Lambda_{121})$}; 
\node (C4) at (-67.5:1.5cm) {$\scriptstyle\Db(\Lambda_{21})$}; 
\node (C5) at (-22.5:1.5cm) {$\scriptstyle\Db(\Lambda_{2})$}; 
\draw[->, bend right]  (C+) to (C1);
\draw[->, bend right]  (C1) to (C+);
\draw[->, bend right]  (C1) to (C2);
\draw[->, bend right]  (C2) to (C1);
\draw[->, bend right]  (C2) to (C-);
\draw[->, bend right]  (C-) to (C2);
\draw[<-, bend right]  (C+) to  (C5);
\draw[<-, bend right]  (C5) to  (C+);
\draw[<-, bend right]  (C5) to  (C4);
\draw[<-, bend right]  (C4) to (C5);
\draw[<-, bend right]  (C4) to  (C-);
\draw[<-, bend right]  (C-) to (C4);
\node at (78.75:0.9cm) {$\scriptstyle \Phi_1$};
\node at (78.75:1.6cm) {$\scriptstyle \Phi_1$};
\node at (135:1.075cm) {$\scriptstyle \Phi_2$};
\node at (135:1.7cm) {$\scriptstyle \Phi_2$};
\node at (198:0.9cm) {$\scriptstyle \Phi_1$};
\node at (198:1.6cm) {$\scriptstyle \Phi_1$};
\node at (258.75:0.9cm) {$\scriptstyle \Phi_2$};
\node at (258.75:1.6cm) {$\scriptstyle \Phi_2$};
\node at (315:1cm) {$\scriptstyle \Phi_1$};
\node at (315:1.75cm) {$\scriptstyle \Phi_1$};
\node at (8:0.95cm) {$\scriptstyle \Phi_2$};
\node at (8:1.6cm) {$\scriptstyle \Phi_2$};
\end{tikzpicture}
\]
where the algebras $\Lambda, \Lambda_{1}, \Lambda_{2}, \Lambda_{12}, \Lambda_{21}, \Lambda_{121}$ are the six NCCRs corresponding to the six crepant resolutions of $\Spec R$, and $\Lambda$ is the NCCR in Lemma~\ref{present NCCR}.  It is possible to present the other NCCRs as quivers with relations, but we refrain from doing this here.

\medskip
 As explained in \cite[Remark 4.7]{Hirano-Wemyss}, using only the tilting order, there is a functorial isomorphism $\Phi_1\Phi_2\Phi_1\cong\Phi_2\Phi_1\Phi_2$.  By slight abuse of notation, we write $\scrS_1, \scrS_2$ for the simple modules for the NCCRs, regardless of which NCCR is under consideration.  The labelling is consistent, so that mutation $\Phi_i$ always shifts $\scrS_i$.

\medskip
Algebraically, the mutation functors can always be described as $\mathbf{R}\Hom(T_i,-)$ where $T_i$ is a certain tilting module of projective dimension one.  From our perspective, the known facts that we require are contained in the following two lemmas.

\begin{lemma}\label{lovely facts 3folds}
Let $x$ be a finite length module viewed as a complex in homological degree zero. Then for $i\in\{1,2\}$, the following statements hold.
\begin{enumerate}
\item\label{lovely facts 3folds 1} $\mathrm{H}^j(\Phi_i(x))=0$ unless $j=0,1$.
\item\label{lovely facts 3folds 2} $\mathrm{H}^j(\Phi_i^{-1}(x))=0$ unless $j=-1,0$.
\item\label{lovely facts 3folds 3}  If $\Hom(x,\scrS_i)=0$ then $\mathrm{H}^j(\Phi_i(x))=0$ for all $j\neq0$.
\item\label{lovely facts 3folds 4} If $\Hom(\scrS_i,x)=0$ then $\mathrm{H}^j(\Phi_i^{-1}(x))=0$ for all $j\neq0$.
\end{enumerate}
\end{lemma}
\begin{proof}
(1) \& (2) are a direct consequence of the fact that $T_i$ is a tilting module, and so has projective dimension one.\\
(3) \& (4) are \cite[Corollary 5.10]{Wemyss-MMP}, which needs $x$ to be finite length.
\end{proof}

As notation, set $M_1=\mathrm{Cone}(\scrS_1[-1]\to \scrS_2)$, and $M_2=\mathrm{Cone}(\scrS_2[-1]\to \scrS_1)$.

\begin{lemma}\label{actions on some modules 3folds}
The following statements hold.
\begin{enumerate}
\item\label{actions on some modules 3folds 1} $\Phi_1(\scrS_1)=\scrS_1[-1]$, $\Phi_1(\scrS_2)=M_1$ and $\Phi_1(M_2)=\scrS_2$.   
\item\label{actions on some modules 3folds 2} $\Phi_2(\scrS_2)=\scrS_2[-1]$, $\Phi_2(\scrS_1)=M_2$ and $\Phi_2(M_1)=\scrS_1$.   
\item\label{actions on some modules 3folds 3}  The composition $\Phi_1\Phi_2\Phi_1$ sends $\scrS_1\mapsto\scrS_2[-1]$, $\scrS_2\mapsto\scrS_1[-1]$, and further if $a\in\mod\Lambda_{\con}$, then $a\mapsto x[-1]$ for some $x\in\mod\Lambda_{\con}$.
\end{enumerate}
\end{lemma}
\begin{proof}
(1) The fact that $\Phi_i(\scrS_i)=\scrS_i[-1]$ is \cite[4.15]{Wemyss-MMP}.   The fact that $\Phi_1(\scrS_2)=M_1$ follows using the same method: by torsion pairs $\Phi_1(\scrS_2)$ is a module in degree zero, and the dimension vector of this module can be determined using the exchange sequences, which are \cite[5.29]{Iyama-WemyssCrelle}.  The final statement $\Phi_1(M_2)=\scrS_2$ follows from the first two facts, and applying $\Phi_1$ to the exact sequence $0\to\scrS_1\to M_2\to\scrS_2\to 0$.  All of Part (2) follows by symmetry. 
For (3), the statement that $\scrS_1\mapsto \scrS_2[-1]$ follows from \cite[(3.I)]{DW-twists_and_braids}, which shows that $\Phi_1\Phi_2\Phi_1$ is isomorphic to a direct mutation functor, together with \cite[4.15]{Wemyss-MMP} which shows simples shift under mutation.  The fact that $\scrS_1$ gets sent to $\scrS_2[-1]$, as opposed to $\scrS_1[-1]$, is due to our global numbering conventions here, and is explained by the fact $(\upvartheta_1,\upvartheta_2)\mapsto(-\upvartheta_2,-\upvartheta_1)$ underneath \cite[(3.I)]{DW-twists_and_braids}. The other statement $\scrS_2\mapsto\scrS_1[-1]$ follows by symmetry.  The final statement follows, since every module is filtered by simples.
 \end{proof}

\begin{remark}\label{rem:PhiPhiisTwist}
For $i=1,2$ the composition $\Phi_i\circ\Phi_i$ is functorially isomorphic to the spherical twist $T_{\scrS_i}$ around the spherical object $\scrS_i$ \cite{DW-twists_and_braids}.  Thus Lemma~\textnormal{\ref{actions on some modules 3folds}} shows that $T_{\scrS_1}(M_2)=M_1$ and $T_{\scrS_2}(M_1)=M_2$, which is the B-side analogue of \eqref{eqn:graded_isotopy}.
\end{remark}

\subsection{Representation Theory of Contraction Algebras}\label{sect:rep of contraction algebras} Our proof of the classification of fat-spherical objects will track through the functors $\Phi_i$ in the groupoid above.  As such, at a technical level the categories change mildly, although of course all are equivalent.  

\medskip
The six crepant resolutions of $\Spec R$ give rise to six versions of the category $\scrC$, all of which are equivalent.  To avoid a proliferation of notation, we will mildly abuse notation, and refer to each simply as $\scrC$.  Further, each crepant resolution has a tilting bundle giving an equivalence to the corresponding NCCR, and again we will abuse notation and write $\Psi$ for this equivalence, regardless of which crepant resolution we are considering.

\medskip
Each of the categories $\scrC$, across the relevant $\Psi$, is equivalent to the subcategory of the derived category of the NCCR whose cohomology is a module for the corresponding contraction algebra. Since the mutation functors restrict to an equivalence on these subcategories, we summarise this information in the following diagram.  The right hand side shows the quiver of the corresponding contraction algebra, where the loops exist if and only if $k>1$. 
\[
\begin{array}{ccc}
\begin{array}{c}
\begin{tikzpicture}[scale=1.3,bend angle=15, looseness=1,>=stealth]
\coordinate (A1) at (135:2cm);
\coordinate (A2) at (-45:2cm);
\coordinate (B1) at (153.435:2cm);
\coordinate (B2) at (-26.565:2cm);
\draw[red!30] (A1) -- (A2);
\draw[black!30] (-2,0)--(2,0);
\draw[black!30] (0,-1.8)--(0,1.8);
\node (C+) at (45:1.5cm) {$\scriptstyle\scrC$};
\node (C1) at (112.5:1.5cm) {$\scriptstyle\scrC$};
\node (C2) at (157.5:1.5cm){$\scriptstyle\scrC$}; 
\node (C-) at (225:1.5cm) {$\scriptstyle\scrC$}; 
\node (C4) at (-67.5:1.5cm) {$\scriptstyle\scrC$}; 
\node (C5) at (-22.5:1.5cm) {$\scriptstyle\scrC$}; 
\draw[->, bend right]  (C+) to (C1);
\draw[->, bend right]  (C1) to (C+);
\draw[->, bend right]  (C1) to (C2);
\draw[->, bend right]  (C2) to (C1);
\draw[->, bend right]  (C2) to (C-);
\draw[->, bend right]  (C-) to (C2);
\draw[<-, bend right]  (C+) to  (C5);
\draw[<-, bend right]  (C5) to  (C+);
\draw[<-, bend right]  (C5) to  (C4);
\draw[<-, bend right]  (C4) to (C5);
\draw[<-, bend right]  (C4) to  (C-);
\draw[<-, bend right]  (C-) to (C4);
\node at (78.75:0.9cm) {$\scriptstyle \Phi_1$};
\node at (78.75:1.6cm) {$\scriptstyle \Phi_1$};
\node at (135:1.075cm) {$\scriptstyle \Phi_2$};
\node at (135:1.7cm) {$\scriptstyle \Phi_2$};
\node at (198:0.9cm) {$\scriptstyle \Phi_1$};
\node at (198:1.6cm) {$\scriptstyle \Phi_1$};
\node at (258.75:0.9cm) {$\scriptstyle \Phi_2$};
\node at (258.75:1.6cm) {$\scriptstyle \Phi_2$};
\node at (315:1cm) {$\scriptstyle \Phi_1$};
\node at (315:1.75cm) {$\scriptstyle \Phi_1$};
\node at (8:0.95cm) {$\scriptstyle \Phi_2$};
\node at (8:1.6cm) {$\scriptstyle \Phi_2$};
\end{tikzpicture}
\end{array}
&&
\begin{array}{c}
\begin{tikzpicture}[scale=1.3,bend angle=15, looseness=1,>=stealth]
\coordinate (A1) at (135:2cm);
\coordinate (A2) at (-45:2cm);
\coordinate (B1) at (153.435:2cm);
\coordinate (B2) at (-26.565:2cm);
\draw[red] (A1) -- (A2);
\draw[black] (-2,0)--(2,0);
\draw[black] (0,-1.8)--(0,1.8);
\node (C+) at (45:1.5cm) {$\begin{tikzpicture}[scale=0.4]
\node (A) at (0,0) [Bs] {};
\node (B) at (2,0) [Bs] {};
\draw[->, bend left] (A) to (B);
\draw[->, bend left] (B) to (A);
\end{tikzpicture}$};
\node (C1) at (112.5:1.5cm) {$\begin{tikzpicture}[scale=0.4]
\node (A) at (0,0) [Bs] {};
\node (B) at (2,0) [Bs] {};
\draw[->, bend left] (A) to (B);
\draw[->, bend left] (B) to (A);
\draw[->]  (B) edge [in=30,out=-40,loop,looseness=6,pos=0.5]  (B);
\end{tikzpicture}$};
\node (C2) at (157.5:1.5cm){$\begin{tikzpicture}[scale=0.4]
\node (A) at (0,0) [Bs] {};
\node (B) at (2,0) [Bs] {};
\draw[->, bend left] (A) to (B);
\draw[->, bend left] (B) to (A);
\draw[->]  (B) edge [in=30,out=-40,loop,looseness=6,pos=0.5]  (B);
\end{tikzpicture}$}; 
\node (C-) at (225:1.5cm) {$\begin{tikzpicture}[scale=0.4]
\node (A) at (0,0) [Bs] {};
\node (B) at (2,0) [Bs] {};
\draw[->, bend left] (A) to (B);
\draw[->, bend left] (B) to (A);
\end{tikzpicture}$}; 
\node (C4) at (-67.5:1.5cm) {$\begin{tikzpicture}[scale=0.4]
\node (A) at (0,0) [Bs] {};
\node (B) at (2,0) [Bs] {};
\draw[->, bend left] (A) to (B);
\draw[->, bend left] (B) to (A);
\draw[->]  (A) edge [in=150,out=220,loop,looseness=6,pos=0.5]  (A);
\end{tikzpicture}$}; 
\node (C5) at (-22.5:1.5cm) {$\begin{tikzpicture}[scale=0.4]
\node (A) at (0,0) [Bs] {};
\node (B) at (2,0) [Bs] {};
\draw[->, bend left] (A) to (B);
\draw[->, bend left] (B) to (A);
\draw[->]  (A) edge [in=150,out=220,loop,looseness=6,pos=0.5]  (A);
\end{tikzpicture}$}; 
\end{tikzpicture}
\end{array}
\end{array}
\]

Also considering the relations (which we have suppressed), by symmetry, up to isomorphism there are only two algebras in the right hand side above, namely
\[
\begin{array}{ccc}
\begin{array}{c}
\begin{tikzpicture}[scale=0.8]
\node (A) at (0,0) [B] {};
\node (B) at (2,0) [B] {};
\draw[->, bend left] (A) to node[above]{$\scriptstyle a$} (B);
\draw[->, bend left] (B) to node[below]{$\scriptstyle b$} (A);
\end{tikzpicture}
\end{array}
&\quad&
\begin{array}{c}
\begin{tikzpicture}[scale=0.8]
\node (A) at (0,0) [B] {};
\node (B) at (2,0) [B] {};
\draw[->, bend left] (A) to node[above]{$\scriptstyle a$} (B);
\draw[->, bend left] (B) to node[below]{$\scriptstyle b$} (A);
\draw[->]  (B) edge [in=30,out=-40,loop,looseness=14,pos=0.5] node[right] {$\scriptstyle y$} (B);
\end{tikzpicture}
\end{array}\\
(ab)^ka=0=b(ab)^k
&&
y^k=ba, \,\,ay=0=yb.
\end{array}
\]
We write $\Lambda_{\con}$ for the first, and $\Upgamma_{\con}$ for the second.  Note that both depend on $k$, and they are isomorphic if and only if $k=1$.  Further, $\dim_{\bK}\Lambda_{\con}=2+4k$ and $\dim_{\bK}\Upgamma_{\con}=k+5$.

\begin{Remark} The difference between $k=1$ and $k>1$ is also manifest symplectically. The Stein manifold $W_k$ has a Lagrangian skeleton $Q_0 \cup Q_1$, but also a Lagrangian skeleton given by the union of one $Q_i$ and the Morse-Bott surgery $K$; cf.\ Example \ref{Ex:MBLefschetz}, where $K$ fibres over a path in the base of the Morse-Bott-Lefschetz fibration of Figure \ref{Fig:MBLdegeneration} between the outermost two critical points. When $k\neq 1$ and this is a Lens space, the endomorphism algebra of the second core (and its wrapped counterpart) differs from that of the first, even though they are derived equivalent. \end{Remark}  

For our purposes, we only require the following two facts, which follow from an analysis of the representation theory.  These will turn out to be the key to understanding how to iteratively apply mutation functors later to `improve' fat spherical objects.

\begin{prop}\label{key rep theory fact}
Suppose that $a,b\in\mod\Lambda_{\con}$ satisfies $\Hom_{\scrC}(a,b)=0$. Then either $a$ or $b$ is filtered by a unique simple.  The same result holds for $\Upgamma_{\con}$.  
\end{prop}
\begin{proof}
We give the proof for $k=1$ here.  Indeed, when $k=1$ there are only six indecomposable $\Lambda_{\con}$-modules, namely the two projective $\Lambda_{\con}$-modules, say $\scrQ_1$ and $\scrQ_2$, and the modules $\scrS_1$, $\scrS_2$, $M_1$ and $M_2$.  Either by directly computing homomorphism spaces between these very small dimensional quiver representations, or by knitting on the Auslander--Reiten quiver (as proved in \S\ref{delayed proofs}), or by computer algebra (e.g.\ QPA \cite{QPA, GAP4}), the dimension of the $\Hom(X,Y)$ with $X$ a member of the rows, and $Y$ a member of the columns, is summarised in the following table
\begin{equation}
\begin{array}{ccccccc}
&\scrQ_1&\scrQ_2&M_{1}&M_{2}&\scrS_1&\scrS_2\\
\scrQ_1&2&1&1&1&1&0\\
\scrQ_2&1&2&1&1&0&1\\
M_{1}&1&1&1&1&1&0\\
M_{2}&1&1&1&1&0&1\\
\scrS_1&1&0&0&1&1&0\\
\scrS_2&0&1&1&0&0&1
\end{array}\label{table of Homs}
\end{equation}
Now, if $\scrQ_1$ is a summand of $a$, since by assumption $\Hom(a,b)=0$, the only possibility for $b$ is $\scrS_2^{\oplus t}$, regardless of what other summands $a$ has.   Ditto $M_1$.  Similarly, if  $\scrQ_2$ is a summand of $a$, then $b$ is $\scrS_1^{\oplus t}$.   Ditto $M_2$.  The only remaining options for $a$ are the remaining rows and combinations of them, so either $\scrS_1^{\oplus s}$, $\scrS_1^{\oplus t}$, or $\scrS_1^{\oplus s}\oplus \scrS_2^{\oplus t}$.  In the latter case, if a module is orthogonal to both simples, it is zero. Hence the last two cases are $a$ is either $\scrS_1^{\oplus s}$ or $\scrS_2^{\oplus t}$.

\medskip
The general $k>1$ proof is similar, but the number of indecomposable modules increases. Since these method of proof uses knitting on the AR quiver (which gives an independent check on the computer algebra claim above), and is quite different to the rest of this section, we delay the full proof until Section~\ref{delayed proofs}.
\end{proof}

The following is rather similar.
\begin{lemma}\label{only four bricks}
There are only four modules $X\in\mod\Lambda_{\con}$ for which $\End_\scrC(X)=\bK$.  They are: $\scrS_1$, $\scrS_2$, $M_1$ and $M_2$.  The same statement holds for $X\in\mod\Upgamma_{\con}$.
\end{lemma}
\begin{proof}
For the case $k=1$, this follows directly from observing the diagonal entries in the table \eqref{table of Homs}.  The $k>1$ case again is similar, and uses knitting, and consequently we again delay the proof until Section~\ref{delayed proofs}.
\end{proof}

\subsection{Induction on Length}
With the mild abuse of notation as outlined in Section~\ref{sect:rep of contraction algebras}, for $ a\in\scrC$, write $t$ for the maximum $j$ such that $\mathrm{H}^j(\Psi a)\neq 0$, and $b$ for the minimum $j$ such that $\mathrm{H}^j(\Psi a)\neq 0$.  The \emph{homological length} of $ a\in\scrC$ is defined to be $\ell( a)=t-b$.  In particular,  $\ell( a)=0$ if and only if $\Psi a$ is in a single homological degree, which holds if and only if $\Psi a$ is a shift of a $\Lambda_{\con}$-module (or $\Upgamma_{\con}$-module, depending which $\scrC$ we are considering).

\begin{Remark}
The homological length used here is very different to the length invariant used to induct in the surfaces case \cite{Ishii-Uehara}. Indeed, comparing to Remark \ref{rmk:t-structure}, in the two-dimensional case there are again two natural 2-CY categories: the derived category of modules over the preprojective algebra, and the derived category of the resolution (the analogues in our setting are respectively $\Db(\mod\Lambda)$ and $\Db(\coh Y)$).  \emph{Op.\ cit}.\ works with the  latter rather than the former, which  leads to a more complicated induction.
\end{Remark}

\begin{lemma}\label{Ht to Hb is zero 3-fold}
Suppose that $ a\in\scrC$ satisfies $\Hom_{\scrC}( a, a[i])=0$ for all $i<0$.   If $\ell( a)>0$, then $\Hom_{\scrC}(\mathrm{H}^t(\Psi a),\mathrm{H}^b(\Psi a))=0$.
\end{lemma}
\begin{proof}
This follows from the spectral sequence
\[
E_2^{p,q}=\bigoplus_{i}\Hom^p(\mathrm{H}^i(\Psi a),\mathrm{H}^{i+q}(\Psi a))\Rightarrow \Hom^{p+q}(\Psi a,\Psi a).
\]
The bottom left of the spectral sequence is $E_2^{0,-\ell( a)}=\Hom(\mathrm{H}^t(\Psi a),\mathrm{H}^b(\Psi a))$, and so this survives to $E_\infty$, to give part of  $\Hom_{\Db}(\Psi a,\Psi a[-\ell( a)])$.  Since by assumption $\Psi a$ has no negative Exts, and $\ell( a)>0$, it follows that $E_2^{0,-\ell( a)}=0$, proving the assertion.
\end{proof}

The following very simple spectral sequence argument is one of our key technical lemmas.  It  demonstrates why we don't study the twist functors directly in dimension three: in contrast, the spectral sequence for the twists does not degenerate immediately, since it has $E_2^{0,q}$, $E_2^{1,q}$, and  $E_2^{2,q}$.  Thus, our strategy is to use the mutation functors to simplify first (as for surfaces), then afterwards put everything in terms of the twists.

\begin{lemma}\label{ell drops 3folds}
If $ a\in\scrC$ with $\ell( a)>0$, write $\mathrm{H}^t=\mathrm{H}^t(\Psi a)$ for the top cohomology, and $\mathrm{H}^b=\mathrm{H}^b(\Psi a)$ for the bottom. Then for $i\in\{0,1\}$, the following statements hold.
\begin{enumerate}
\item\label{ell drops 3folds 1} If $\mathrm{H}^1\bigl(\Phi_i(\mathrm{H}^t)\bigr)=0$ and  $\mathrm{H}^{0}\bigl(\Phi_i(\mathrm{H}^b)\bigr)=0$, then $\ell(\Phi_i( a))<\ell( a)$. 
\item\label{ell drops 3folds 2} If $\mathrm{H}^0\bigl(\Phi_i^{-1}(\mathrm{H}^t)\bigr)=0$ and  $\mathrm{H}^{-1}\bigl(\Phi_i^{-1}(\mathrm{H}^b)\bigr)=0$, then $\ell(\Phi_i^{-1}( a))<\ell( a)$. 
\end{enumerate}
\end{lemma}
\begin{proof}
(1) Consider the spectral sequence
\[
E_2^{p,q}=\mathrm{H}^p\bigl(\Phi_i(\mathrm{H}^q(\Psi a))\bigr)\Rightarrow \mathrm{H}^{p+q}\bigl(\Phi_i(\Psi a)\bigr)
\]
By Lemma~\ref{lovely facts 3folds}\eqref{lovely facts 3folds 1}, only $E_2^{0,q}$ and $E_2^{1,q}$ can possibly be non-zero, so the spectral sequence degenerates immediately.  The bottom left entry is $\mathrm{H}^0(\Phi_i(\mathrm{H}^b))$ and the top right entry is $\mathrm{H}^1(\Phi_i(\mathrm{H}^t))$.  If these are both zero, it follows that in the resulting complex $\Phi_i(\Psi a)$, the bottom cohomology is in homological degree $b+1$, and the top cohomology is in homological degree $t$.  Hence the complex has become smaller, and $\ell(\Phi_i(\Psi a))<\ell( a)$.\\
(2) This is very similar, using
\[
E_2^{p,q}=\mathrm{H}^p\bigl(\Phi^{-1}_i(\mathrm{H}^q(\Psi a))\bigr)\Rightarrow \mathrm{H}^{p+q}\bigl(\Phi^{-1}_i(\Psi a)\bigr).
\]
Now by Lemma~\ref{lovely facts 3folds}\eqref{lovely facts 3folds 2} the only non-zero terms are $E_2^{-1,q}$ and $E_2^{0,q}$; the spectral sequence still degenerates immediately.
\end{proof}

The following is the main technical result of this section, and gives a constructive way of simplifying the complex $ a$ by applying mutation functors.
\begin{Theorem}\label{how to drop ell 3folds}
Suppose that $ a\in\scrC$  satisfies $\Hom_{\scrC}( a, a[i])=0$ for all $i<0$, and $\ell( a)>0$.  
\begin{enumerate}
\item Either $\mathrm{H}^t(\Psi a)$ or $\mathrm{H}^b(\Psi a)$ is filtered by a unique simple module $\scrS_i$, with $i\in\{0,1\}$.
\item When $\mathrm{H}^b(\Psi a)$ is filtered only by $\scrS_i$, applying $\Phi_i$ decreases the length $\ell$.
\item When $\mathrm{H}^t(\Psi a)$  is filtered only by $\scrS_i$, applying $\Phi_i^{-1}$ decreases the length $\ell$.
\end{enumerate}
\end{Theorem}
\begin{proof}
(1) By Lemma~\ref{Ht to Hb is zero 3-fold}, $\Hom_{\scrC}(\mathrm{H}^t,\mathrm{H}^b)=0$, so the result follows from the key representation theory fact in Proposition~\ref{key rep theory fact}.\\
(2) Since $\Phi_i(\scrS_i)=\scrS_i[-1]$ by Lemma~\ref{actions on some modules 3folds}, by induction on the length of the filtration of $\mathrm{H}^b$, it follows that $\mathrm{H}^0(\Phi_i(\mathrm{H}^b))=0$.   Further, $\Hom(\mathrm{H}^t,\mathrm{H}^b)=0$ implies $\Hom(\mathrm{H}^t,\scrS_i)=0$, so by Lemma~\ref{lovely facts 3folds}\eqref{lovely facts 3folds 3} $\mathrm{H}^{1}(\Phi_i(\mathrm{H}^t))=0$.  Combining, using Lemma~\ref{ell drops 3folds}\eqref{ell drops 3folds 1} we see that $\ell(\Phi_i( a))<\ell( a)$.\\
(3) This is very similar, using $\Phi_i^{-1}(\scrS_i)=\scrS_i[1]$, and Lemmas~\ref{lovely facts 3folds}\eqref{lovely facts 3folds 4} and \ref{ell drops 3folds}\eqref{ell drops 3folds 2}.
\end{proof}

For any element $\upbeta$ of the Deligne groupoid of the hyperplane arrangement mentioned at the start of Section \ref{Subsec:mutation}, cf. the diagram at the start of that section, there is a composition of mutation functors giving an equivalence $\Db(\Lambda) \to \Db(\Lambda_{\upbeta})$ from the NCCR at the base vertex to that labelling the chamber associated to $\upbeta$. Recall also that in the notational convention from Section \ref{sect:rep of contraction algebras}, we identify $\scrC$ with a subcategory of any $\Db(\Lambda)$ and hence the mutation functors act on elements of $\scrC$.

\begin{Corollary}\label{cor 1 3folds}
Suppose that $ a\in\scrC$ satisfies $\Hom_{\scrC}( a, a[i])=0$ for all $i<0$.  Then there exists a composition of mutation functors $F\colon \Db(\Lambda)\to\Db(\Lambda_{\upbeta})$ such that $Fa\cong a'$, where $a'$ is a module for the contraction algebra of $\Lambda_{\upbeta}$.
\end{Corollary}
\begin{proof}
By inducting along Theorem~\ref{how to drop ell 3folds}, it is easy to see that there exists a composition on mutation functors $G$ such that $\ell(G a)=0$.  Hence $G (\Psi a)\cong a'[j]$ where $a'$ is a module in degree zero, and $j\in\mathbb{Z}$.  Since $\Psi a\in\scrC$, and the mutation functors preserve $\scrC$, necessarily $a'$ is a module for the relevant contraction algebra.

\medskip
Now by Lemma~\ref{actions on some modules 3folds}\eqref{actions on some modules 3folds 3}  $\Phi_1\Phi_2\Phi_1\cong\Phi_2\Phi_1\Phi_2$ sends every contraction algebra module to a shift of some contraction algebra module.   Hence, applying $\Phi_1\Phi_2\Phi_1$ or its inverse repeatedly if necessary, we can move $G(\Psi a)$ to homological degree zero.
\end{proof}

The following is the main result of this subsection.
\begin{Corollary}\label{cor 2 3folds}
Suppose that $ a\in\scrC$ satisfies $\Hom_{\scrC}( a, a[i])=0$ for all $i<0$, and further $\dim_k\Hom_{\scrC}( a, a)=1$.  Then the following statements hold.
\begin{enumerate}
\item There exists a composition of mutation functors $F\colon\Db(\Lambda)\to\Db(\Lambda_{\upbeta})$ such that $F a$ is either $\scrS_1$ or $\scrS_2$.
\item The complex $ a$ is fat-spherical, and up to the action of the pure braid group, is either $\scrS_1,  \scrS_2, M_1$, or their shifts by $[1]$.
\end{enumerate}
\end{Corollary}
\begin{proof}
(1) By Corollary~\ref{cor 1 3folds}, we can find a composition of mutation functors $F\colon\Db(\Lambda)\to\Db(\Lambda_{\upbeta})$ such that $F a\cong a'$, where $a'$ is a $(\Lambda_{\upbeta})_{\con}$-module in homological degree zero.   But 
\[
\Hom_{\scrC}(a',a')\cong\Hom_{\scrC}(a,a)=\bK,
\]
so Lemma~\ref{only four bricks} forces  $a'$ to be either $M_1$, $M_2$, $\scrS_1$ or $\scrS_2$.  Since $\Phi_2(M_1)\cong\scrS_1$ and $\Phi_1(M_2)\cong\scrS_2$ by Lemma~\ref{actions on some modules 3folds}, in all cases overall we can find a composition of mutation functors taking $a$ to either $\scrS_1$ or $\scrS_2$.\\
(2) By (1) there exists $F$ such that $Fa$ is either $\scrS_1$ and $\scrS_2$. Depending on the NCCR $\Lambda_{\upbeta}$, the $\scrS_i$ can be either spherical or fat-spherical (when $k=1$ they are always spherical).  Since  autoequivalences preserve both, $ a$ itself must also be spherical, or fat-spherical.  For the last statement, consider the shortest positive chain of mutation functors $\Phi_\upalpha\colon\Db(\Lambda)\to\Db(\Lambda_{\upbeta})$.  Since $Fa$ is a simple, by torsion pairs \cite[(5.B)]{Hirano-Wemyss} $\Phi_{\upalpha}^{-1}Fa$ is either a module, or a shift $[1]$ of a module.  Since $\End_{\scrC}(\Phi_{\upalpha}^{-1}Fa)$ is one-dimensional,  by Lemma~\ref{only four bricks} it is either $M_1$, $M_2$, $\scrS_1$ or $\scrS_2$, or their shift $[1]$.  Given $\Phi_1\Phi_1(M_2)=M_1$ by Remark~\ref{rem:PhiPhiisTwist}, the result follows.
\end{proof}

\subsection{Symplectic applications}

We collect a number of symplectic applications of Corollary \ref{cor 2 3folds}, which in particular establish several claims from the Introduction.   Recall that our basic equivalence $\scrQ_p \simeq \scrC_p$ concerns $\bZ$-graded categories, and an exact Lagrangian $L\subset W_p$ (more generally, an exact Lagrangian equipped with a local system) defines an object of the $\bZ$-graded Fukaya category precisely when it has vanishing Maslov class. The basic translation of the Corollary is as follows, where by a trivial local system we mean no local system.

\begin{Corollary}\label{cor: homotopy main}
Let $p=1$ or $p>2$ be prime. Let $L\subset W_p$ be a closed connected exact Lagrangian submanifold with vanishing Maslov class, equipped with a  rank one $\bK$-local system. Then $L$ is quasi-isomorphic to a pure braid image of one of the two core components or the Morse-Bott surgery of the core components. 
\end{Corollary}
\begin{proof} By exactness, for any rank one local system $\xi \to L$, the endomorphisms of the corresponding Lagrangian brane reproduce those of the trivial local system, which are in turn given by ordinary cohomology, $HF^*(\xi,\xi) \cong HF^*(L,L) \cong H^*(L;\bK)$. The hypotheses  imply that $HF^{<0}(L,L) = 0$ and $HF^0(L,L) \cong \bK$. Under the mirror equivalence $\scrF(W_p;\bK) = D^{\pi}\scrQ_p \simeq \scrC_p$, valid over a field of the appropriate characteristic, $L$ (or more generally $\xi \to L$) defines an object $\scrE_L \in \scrC_p$ satisfying the hypotheses of Corollary \ref{cor 2 3folds}.  The second statement of the Corollary shows that $\scrE_L$ is a pure braid image of one of the simples $\scrS_1, \scrS_2$ or the module $M_1$ obtained as the mapping cone. The equivalence $\scrQ_p \simeq \scrC_p$ identifies the cores $Q_i$ with the simples $\scrS_i$, so the result follows. \end{proof} 

It follows that every two graded exact Lagrangians in $W_p$, with $p=1$ or $p>2$ prime, have non-empty intersection, which is already not obvious. \medskip

\begin{Remark} \label{rmk:need_finite_char} We established in Proposition \ref{Prop:they_are_quivers} that, in characteristic zero, there is an equivalence $\scrQ_p \simeq \scrQ_1$ for any $p>0$. Nonetheless, one cannot obtain restrictions on Lagrangian submanifolds of $W_p$ from this, because for $p>1$ the subcategory $\scrQ_p \subset \scrF(W_p;\bK)$ is not split-generating if $\charac(\bK) = 0$, and not invariant under the action of the pure braid group, cf. the discussion after Lemma \ref{Lem:no_generates} and Remark \ref{rmk:char_0_needs_local_systems}.  The symplectic applications for Lagrangians in $W_p$ with $p\neq 1$ rely essentially on the equivalence of categories in finite characteristic.
\end{Remark}

Note that in the case $p>2$, the endomorphisms of the core components and of their surgery themselves differ: the surgery is fat spherical but not spherical.  

\begin{Corollary}
If $p>2$ is prime, then every Lagrangian sphere in $W_p$ admits an orientation such that, after applying a twist by the action of the pure braid group, it becomes quasi-isomorphic to one of the two oriented core components.  If $p=1$, the same is true, except that it may also be quasi-isomorphic to a surgery of the two cores.
\end{Corollary}

Identify $H_3(W_p;\bZ) \cong \bZ\oplus\bZ$ with generators the core components $Q_i$, with either choice of orientation. The following is the homological analogue for $W_p$ of the `nearby Lagrangian conjecture', which implies that an exact Lagrangian in the cotangent bundle is homologous to the zero-section.  (Note, however, that we require a hypothesis on the Maslov class.). 

\begin{Corollary} \label{Cor:homology_class} Let $p=1$ or $p>2$ be prime. If $L\subset W_p$ is a closed  exact Lagrangian submanifold with vanishing Maslov class, then $\pm [L] \in \{(1,0), (0,1), (1,\pm 1)\} \subset \bZ\oplus\bZ$.
\end{Corollary}
\begin{proof} 
Take an exact $L\subset W_p$ with $p>2$; the case $p=1$ is analogous. The pure braid group $\PBr_3$ acts on $W_p$ by symplectomorphisms, and the induced action on integral homology is trivial. Fix a field $\bK$ of characteristic $p$. According to Corollary  \ref{cor 2 3folds}, we have $\phi \in \PBr_3$ for which  $\phi(L) $ is quasi-isomorphic to a core component, which without loss of generality we take to be $S_1$, in $\scrF(W_p;\bK)$. From quasi-isomorphism over $\bK$ we can draw two conclusions:

\begin{enumerate}
\item  $\phi(L)$, and hence $L$, has homology class equal to $[S_1] \in H_3(W_p;\bK)$ (by considering its class in the Grothendieck group of $\scrF(W_p;\bK)$ and taking the open-closed map). Picking the correct orientation on $L$, we can assume its homology class in $H_3(W_p;\bZ) $ differs from $(1,0) \in \bZ \oplus \bZ$ by something $p$-divisible, hence has the form $(1+d. p, e.p)$ for some $d,e \in \bZ$.

\item $HF(L, T_x^*;\bK)= \bK$ and $HF(L,T_y^*;\bK) = 0$, where $T_x^*$ is the cotangent fibre to the first core component and $T_y^*$ is the cotangent fibre to the second core component. 
\end{enumerate}

Working with Floer cohomology over $\bZ$, Euler characteristic with a fixed object (here, the cotangent fibre)
\[
\chi(HF(T_x^*,-; \bZ)): \{\mathrm{Exact \, closed \, Lagrangians}\} \to \bZ
\]
defines a map from the set of exact closed Lagrangians in $W_p$ to $\bZ$. (There is no difference between ordinary and wrapped Floer cohomology here since one of the Lagrangians is assumed compact.) The Euler characteristic of $HF$ is, up to sign, given by the algebraic intersection number, so we conclude that the function sends $L$ to $\pm(1+d.p)$. 
\medskip

If $d>0$  we learn that $\rk_{\bZ}(HF(L,T_x^*)) \geq 1+d.p$.  If $d<0$ then $\rk_{\bZ}(HF(L,T_x^*)) \geq |d|p-1 > 1$ (since by hypothesis we are considering the case $p>2$). However, universal coefficients then says in either case that the rank of $HF(L,T_x^*;\bK)$ is $>1$, which violates the second point above; so we conclude that $d=0$. Similarly, considering $\chi(HF(T_y^*,—;\bZ))$ we conclude $e=0$. Thus, the $\bZ$-homology class of the Lagrangian $L$ is $(1,0)$. 
\end{proof}

\begin{Corollary}
If $p>2$ is prime, then any exact graded closed Lagrangian $L \subset W_p$  satisfies $H^*(L;\bK) = H^*(S^3;\bK)$ if $\pm[L] \in \{(1,0), (0,1)\} \subset H_3(W_p;\bZ)$, and  $H^*(L;\bK) = H^*(L(p,1);\bK)$ if $\pm [L] \in \{ (1,\pm1) \} \subset H_3(W_p;\bZ)$.
 \end{Corollary}
 
 As explained in Remark \ref{rmk:need_finite_char}, both the previous results rely essentially on our mirror symmetry statements in \emph{finite} characteristic. 
 
 \begin{Corollary} \label{Cor:no_S1xS2} 
 If $p=1$ or $p>2$ is prime, there is no exact Lagrangian $S^1\times S^2 \subset W_p$ with vanishing Maslov class. \end{Corollary}
\begin{proof} 
Although the graded rings $H^*(L(p,1);\bZ/p)$ and $H^*(S^1\times S^2;\bZ/p)$ agree, the former has a non-trivial $\mu^3$-operation (Massey product), whilst the latter is formal. Since the $A_{\infty}$-structure in the endomorphism algebra of an exact Lagrangian reproduces the classical structure on cohomology,  the result then follows as before. \end{proof} 

 Note that $W_0$ \emph{does} contain an exact graded Lagrangian $S^1\times S^2$, namely the surgery of the two core spheres.

\subsection{Proof of \ref{key rep theory fact} and \ref{only four bricks}}\label{delayed proofs}
In this subsection, which is logically independent of the remainder of the paper, we give the full proofs of the two representation theory facts presented in Section~\ref{sect:rep of contraction algebras}. Namely, we establish that if two modules for the contraction algebras $\Lambda_{\con}$ or $\Upgamma_{\con}$ are orthogonal, then one is filtered by a single simple, and furthermore we prove that both $\Lambda_{\con}$ and $\Upgamma_{\con}$ admit precisely four bricks.

\medskip
Our method to approach both problems uses Auslander--Reiten (AR) theory.  

\begin{lemma}\label{AR quivers with dim vectors}
The AR quivers of the contraction algebras $\Lambda_{\con}$ and $\Upgamma_{\con}$ are illustrated in Figure~\ref{AR figure}, where the numbers denote the dimension vectors of the corresponding representation.
\begin{figure}[h]
\[
\begin{array}{ccc}
\includegraphics[angle=0,scale = 1]{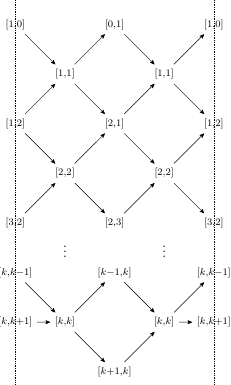}\quad&
\includegraphics[angle=0,scale = 1]{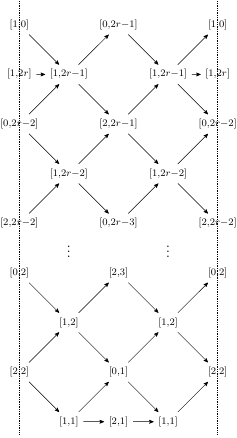}\quad&
\includegraphics[angle=0,scale = 1]{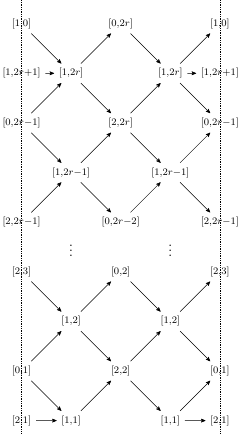}\\
\begin{array}{c}
\Lambda_{\con}\\
{\scriptstyle k\geq 1}
\end{array}
&
\begin{array}{c}
\Upgamma_{\con}\\ {\scriptstyle k=2r-1, \,\, r\geq 2}
\end{array}&
\begin{array}{c}
\Upgamma_{\con}\\
{\scriptstyle k=2r, \,\, r\geq 1}
\end{array}
\end{array}
\]
\caption{AR quivers of $\Lambda_{\con}$ and $\Upgamma_{\con}$.}
\label{AR figure}
\end{figure}
In each case, the left and right dotted arrows are identified, and so AR quiver is really drawn on a cylinder.
\end{lemma}
\begin{proof}
Both $\Lambda_{\con}$ and $\Upgamma_{\con}$ are Brauer tree algebras of finite representation type.  Further,  $\Lambda_{\con}$ is a symmetric Nakayama algebra.  The above AR quivers are well-known, and can be extracted from e.g.\ \cite{Gabriel-Reidtmann}. 
\end{proof}

\begin{proof}[Proof of \ref{key rep theory fact}]
The proof for $\Lambda_{\con}$ and for $\Upgamma_{\con}$ is slightly different.  For $\Lambda_{\con}$, we claim that if $X$ and $Y$ are indecomposable $\Lambda_{\con}$-modules which are not simple, then $\Hom(X,Y)\neq 0$.  This is just clear from the shape of the AR quiver, and the fact that the simples are in the top row.  For example, when $k=3$, the blue region in the following knitting calculation:
\[
\begin{array}{c}
\begin{tikzpicture}[scale=0.6]
\node (At) at (0,2) [B] {};
\node (A4) at (0,3) [B] {};
\node (A2) at (0,5)   {$\scriptstyle 1$};
\node (A1) at (0,7)  [B] {};

\node (Bt) at (1,2) [B] {};
\node (B2) at (1,4) {$\scriptstyle 1$};
\node (B1) at (1,6) {$\scriptstyle 1$};

\node (Ct) at (2,1) [B] {};
\node (C4) at (2,3) {$\scriptstyle 1$};
\node (C2) at (2,5) {$\scriptstyle 1$};
\node (C1) at (2,7) {$\scriptstyle 1$};

\node (Dt) at (3,2) {$\scriptstyle 1$};
\node (D2) at (3,4) {$\scriptstyle 1$};
\node (D1) at (3,6) {$\scriptstyle 1$};

\node (Et) at (4,2) {$\scriptstyle 1$};
\node (E4) at (4,3) {$\scriptstyle 1$};
\node (E2) at (4,5) {$\scriptstyle 1$};
\node (E1) at (4,7) {$\scriptstyle 0$};

\node (Ft) at (5,2)  {$\scriptstyle 1$};
\node (F2) at (5,4) {$\scriptstyle 1$};
\node (F1) at (5,6) {$\scriptstyle 0$};

\node (Gt) at (6,1) {$\scriptstyle 1$};
\node (G4) at (6,3) {$\scriptstyle 1$};
\node (G2) at (6,5) {$\scriptstyle 0$};
\node (G1) at (6,7) {$\scriptstyle 0$};

\node (Ht) at (7,2) {$\scriptstyle 1$};
\node (H2) at (7,4) {$\scriptstyle 0$};
\node (H1) at (7,6) {$\scriptstyle 0$};

\node (It) at (8,2) {$\scriptstyle 1$};
\node (I4) at (8,3) {$\scriptstyle 0$};
\node (I2) at (8,5) {$\scriptstyle 0$};
\node (I1) at (8,7) {$\scriptstyle 0$};

\node (Jt) at (9,2) {$\scriptstyle 0$};
\node (J2) at (9,4) {$\scriptstyle 0$};
\node (J1) at (9,6) {$\scriptstyle 0$};

\draw[->] (A1)--(B1);
\draw[->] (A2)--(B1);
\draw[->] (A2)--(B2);
\draw[->] (A4)--(B2);
\draw[->] (A4)--(Bt);
\draw[->] (At)--(Bt);

\draw[->] (B1)--(C1);
\draw[->] (B1)--(C2);
\draw[->] (B2)--(C2);
\draw[->] (B2)--(C4);
\draw[->] (Bt)--(C4);
\draw[->] (Bt)--(Ct);

\draw[->] (C1)--(D1);
\draw[->] (C2)--(D1);
\draw[->] (C2)--(D2);
\draw[->] (C4)--(D2);
\draw[->] (C4) --(Dt);
\draw[->] (Ct) --(Dt);

\draw[->] (D1)--(E1);
\draw[->] (D1)--(E2);
\draw[->] (D2)--(E2);
\draw[->] (D2)--(E4);
\draw[->] (Dt)--(E4);
\draw[->] (Dt)--(Et);

\draw[->] (E1)--(F1);
\draw[->] (E2)--(F1);
\draw[->] (E2)--(F2);
\draw[->] (E4)--(F2);
\draw[->] (E4)--(Ft);
\draw[->] (Et)--(Ft);

\draw[->] (F1)--(G1);
\draw[->] (F1)--(G2);
\draw[->] (F2)--(G2);
\draw[->] (F2)--(G4);
\draw[->] (Ft)--(G4);
\draw[->] (Ft)--(Gt);

\draw[->] (G1)--(H1);
\draw[->] (G2)--(H1);
\draw[->] (G2)--(H2);
\draw[->] (G4)--(H2);
\draw[->] (G4) --(Ht);
\draw[->] (Gt) --(Ht);

\draw[->] (H1)--(I1);
\draw[->] (H1)--(I2);
\draw[->] (H2)--(I2);
\draw[->] (H2)--(I4);
\draw[->] (Ht)--(I4);
\draw[->] (Ht)--(It);

\draw[->] (I1)--(J1);
\draw[->] (I2)--(J1);
\draw[->] (I2)--(J2);
\draw[->] (I4)--(J2);
\draw[->] (I4)--(Jt);
\draw[->] (It)--(Jt);
\draw[densely dotted] (0,1) -- (0,7.5);
\draw[densely dotted] (4,1) -- (4,7.5);
\draw[densely dotted] (8,1) -- (8,7.5);

\draw[line width=1pt,rounded corners=8pt,blue!80!black] (-0.5,5)--(4,0.5)--(6,0.5) -- (6.5,1.25)-- (3.25,4.5) --(3.25,6.5) --(1, 6.5) --cycle;
\draw[blue!80!black] (0,5) circle (7pt);
\end{tikzpicture}
\end{array}
\]
shows that there is a non-zero homomorphism from the circled representation to every other indecomposable non-simple module.  By the highly regular shape of the AR quiver, the calculation for every other non-simple representation is very similar. 

\medskip
Now suppose that $a, b\in\mod\Lambda_{\con}$ satisfies $\Hom(a,b)= 0$.  By above, if $a$ contains a non-simple module, then $b$ cannot.  But then $b=\scrS_1^{\oplus s}\oplus \scrS_2^{\oplus t}$. Since no non-zero module can be orthogonal to both simples, either $b=\scrS_1^{\oplus s}$ or $b=\scrS_2^{\oplus t}$.  The only other case is when $a$ only contains simple modules; for a similar reason either $a=\scrS_1^{\oplus s}$ or $a= \scrS_2^{\oplus t}$. 

\medskip
For $\Upgamma_{\con}$, the proof is slightly different. The trick here is that, due to the asymmetry in the quiver, $\scrS_2$ injects into almost all modules, and almost all modules surject onto $\scrS_2$.  Indeed, the left hand side of the following digram circles all indecomposable modules $X$ for which $\Hom(\scrS_2,X)=0$, and the right hand side circles all indecomposable modules $Y$ for which $\Hom(Y,\scrS_2)=0$. We illustrate the case when $k$ is odd; the $k$-even case is similar.    Both of these calculations can be achieved via knitting (see below for a very similar calculation). 

\[
\begin{array}{ccc}
\begin{array}{c}
\begin{tikzpicture}[scale=0.6]
\node (At) at (0,0) [B] {};
\node (A4) at (0,2) [B] {};
\node (A3) at (0,3) [B] {};
\node (A2) at (0,5) [B] {};
\node (A1b) at (0,6) [B] {};
\node (A1) at (0,7) [B] {};

\node (Bt2) at (1,-1) [B] {};
\node (Bt) at (1,1) [B] {};
\node (B2) at (1,4) [B] {};
\node (B1) at (1,6) [B] {};
\draw (Bt2) circle (10pt);

\node (Ct1) at (2,-1) [B] {};
\draw (Ct1) circle (10pt);
\node (Ct) at (2,0) [B] {};
\node (C4) at (2,2) [B] {};
\node (C3) at (2,3) [B] {};
\node (C2) at (2,5) [B] {};
\node (C1) at (2,7) [B] {};

\node (Dt1) at (3,-1) [B] {};
\node (Dt) at (3,1) [B] {};
\node (D2) at (3,4) [B] {};
\node (D1) at (3,6) [B] {};

\node (Et) at (4,0) [B] {};
\node (E4) at (4,2) [B] {};
\node (E3) at (4,3) [B] {};
\node (E2) at (4,5) [B] {};
\node (E1b) at (4,6) [B] {};
\node (E1) at (4,7) [B] {};
\draw (E1) circle (10pt);

\node at (1,2.5) {$\vdots$};
\node at (3,2.5) {$\vdots$};

\draw[->] (A1)--(B1);
\draw[->] (A1b)--(B1);
\draw[->] (A2)--(B1);
\draw[->] (A2)--(B2);
\draw[->] (A3)--(B2);
\draw[->] (A4)--(Bt);
\draw[->] (At)--(Bt);
\draw[->] (At)--(Bt2);

\draw[->] (B1)--(C1);
\draw[->] (B1)--(C2);
\draw[->] (B2)--(C2);
\draw[->] (B2)--(C3);
\draw[->] (Bt)--(C4);
\draw[->] (Bt)--(Ct);
\draw[->] (Bt2)--(Ct);
\draw[->] (Bt2)--(Ct1);

\draw[->] (C1)--(D1);
\draw[->] (C2)--(D1);
\draw[->] (C2)--(D2);
\draw[->] (C3)--(D2);
\draw[->] (C4) --(Dt);
\draw[->] (Ct) --(Dt);
\draw[->] (Ct) --(Dt1);
\draw[->] (Ct1) --(Dt1);

\draw[->] (D1)--(E1);
\draw[->] (D1)--(E1b);
\draw[->] (D1)--(E2);
\draw[->] (D2)--(E2);
\draw[->] (D2)--(E3);
\draw[->] (Dt)--(E4);
\draw[->] (Dt)--(Et);
\draw[->] (Dt1)--(Et);

\draw[densely dotted] (0,-1.25) -- (0,7.5);
\draw[densely dotted] (4,-1.25) -- (4,7.5);
\end{tikzpicture}
\end{array}

&\quad&
\begin{array}{c}
\begin{tikzpicture}[scale=0.6]
\node (At) at (0,0) [B] {};
\node (A4) at (0,2) [B] {};
\node (A3) at (0,3) [B] {};
\node (A2) at (0,5) [B] {};
\node (A1b) at (0,6) [B] {};
\node (A1) at (0,7) [B] {};

\node (Bt2) at (1,-1) [B] {};
\node (Bt) at (1,1) [B] {};
\node (B2) at (1,4) [B] {};
\node (B1) at (1,6) [B] {};

\node (Ct1) at (2,-1) [B] {};
\draw (Ct1) circle (10pt);
\node (Ct) at (2,0) [B] {};
\node (C4) at (2,2) [B] {};
\node (C3) at (2,3) [B] {};
\node (C2) at (2,5) [B] {};
\node (C1) at (2,7) [B] {};

\node (Dt1) at (3,-1) [B] {};
\draw (Dt1) circle (10pt);
\node (Dt) at (3,1) [B] {};
\node (D2) at (3,4) [B] {};
\node (D1) at (3,6) [B] {};

\node (Et) at (4,0) [B] {};
\node (E4) at (4,2) [B] {};
\node (E3) at (4,3) [B] {};
\node (E2) at (4,5) [B] {};
\node (E1b) at (4,6) [B] {};
\node (E1) at (4,7) [B] {};
\draw (E1) circle (10pt);

\node at (1,2.5) {$\vdots$};
\node at (3,2.5) {$\vdots$};

\draw[->] (A1)--(B1);
\draw[->] (A1b)--(B1);
\draw[->] (A2)--(B1);
\draw[->] (A2)--(B2);
\draw[->] (A3)--(B2);
\draw[->] (A4)--(Bt);
\draw[->] (At)--(Bt);
\draw[->] (At)--(Bt2);

\draw[->] (B1)--(C1);
\draw[->] (B1)--(C2);
\draw[->] (B2)--(C2);
\draw[->] (B2)--(C3);
\draw[->] (Bt)--(C4);
\draw[->] (Bt)--(Ct);
\draw[->] (Bt2)--(Ct);
\draw[->] (Bt2)--(Ct1);

\draw[->] (C1)--(D1);
\draw[->] (C2)--(D1);
\draw[->] (C2)--(D2);
\draw[->] (C3)--(D2);
\draw[->] (C4) --(Dt);
\draw[->] (Ct) --(Dt);
\draw[->] (Ct) --(Dt1);
\draw[->] (Ct1) --(Dt1);

\draw[->] (D1)--(E1);
\draw[->] (D1)--(E1b);
\draw[->] (D1)--(E2);
\draw[->] (D2)--(E2);
\draw[->] (D2)--(E3);
\draw[->] (Dt)--(E4);
\draw[->] (Dt)--(Et);
\draw[->] (Dt1)--(Et);

\draw[densely dotted] (0,-1.25) -- (0,7.5);
\draw[densely dotted] (4,-1.25) -- (4,7.5);
\end{tikzpicture}
\end{array}
\\
\Hom(\scrS_2,X)=0&&\Hom(Y,\scrS_2)=0
\end{array}
\]
Write $N$ for the indecomposable module in the middle of the bottom row, which is circled in both sides above.  Now, we are given $a,b$ such that $\Hom(a,b)=0$.  If $a$ has a summand from the complement of the right hand circles, \emph{and} $b$ has a summand from the  complement of the left hand circles, then $a\twoheadrightarrow \scrS_2\hookrightarrow b$, which is a non-zero homomorphism.  Hence we can assume that either $a=\scrS_1^{\oplus s}\oplus M_1^{\oplus t}\oplus N^{\oplus u}$ or $b=\scrS_1^{\oplus s}\oplus M_2^{\oplus t}\oplus N^{\oplus u}$.

\medskip
Now, observe that, for any $x\in\mod\Upgamma_{\con}$,
\begin{align*}
\Hom(M_1,x)=0&\iff \Hom(N,x)=0\tag{by positions in AR quiver}\\
&\iff \Hom(x,N)=0\tag{$\Upgamma_{\con}$ is 0-CY, and $N$ is projective}  \\
&\iff \Hom(x,M_2)=0,\tag{by positions in the AR quiver}  
\end{align*}
and furthermore the following knitting calculation (illustrated for $k=5$) shows that the first condition, and hence all conditions, is equivalent to $x$ being filtered by $\scrS_2$.
\[
\begin{array}{ccc}
\begin{array}{c}
\begin{tikzpicture}[yscale=0.6,xscale=0.5]
\node (xt1) at (-1,-1) {$\scriptstyle 1$};
\node (xt) at (-1,1) [B] {};
\node (x3) at (-1,3) [B] {};
\node (x2) at (-1,5) [B] {};
\node (x1) at (-1,7) [B] {};

\node (At) at (0,0) {$\scriptstyle 1$};
\node (A4) at (0,2) [B] {};
\node (A3) at (0,4) [B] {};
\node (A2) at (0,6) [B] {};
\node (A1b) at (0,7) [B] {};
\node (A1) at (0,8) [B] {};

\node (Bt2) at (1,-1)  {$\scriptstyle 0$};
\node (Bt) at (1,1) {$\scriptstyle 1$};
\node (B3) at (1,3) [B] {};
\node (B2) at (1,5) [B] {};
\node (B1) at (1,7) [B] {};

\node (Ct1) at (2,-1)  {$\scriptstyle 0$};
\node (Ct) at (2,0) {$\scriptstyle 0$};
\node (C4) at (2,2) {$\scriptstyle 1$};
\node (C3) at (2,4) [B] {};
\node (C2) at (2,6) [B] {};
\node (C1) at (2,8) [B] {};

\node (Dt1) at (3,-1) {$\scriptstyle 0$};
\node (Dt) at (3,1) {$\scriptstyle 0$};
\node (D3) at (3,3)  {$\scriptstyle 1$};
\node (D2) at (3,5) [B] {};
\node (D1) at (3,7) [B] {};

\node (Et) at (4,0) {$\scriptstyle 0$};
\node (E4) at (4,2) {$\scriptstyle 0$};
\node (E3) at (4,4) {$\scriptstyle 1$};
\node (E2) at (4,6) [B] {};
\node (E1b) at (4,7) [B] {};
\node (E1) at (4,8) [B] {};

\node (BBt2) at (5,-1) {$\scriptstyle 0$};
\node (BBt) at (5,1) {$\scriptstyle 0$};
\node (BB3) at (5,3)  {$\scriptstyle 0$};
\node (BB2) at (5,5) {$\scriptstyle 1$};
\node (BB1) at (5,7) [B] {};

\node (CCt1) at (6,-1) {$\scriptstyle 0$};
\node (CCt) at (6,0) {$\scriptstyle 0$};
\node (CC4) at (6,2) {$\scriptstyle 0$};
\node (CC3) at (6,4) {$\scriptstyle 0$};
\node (CC2) at (6,6) {$\scriptstyle 1$};
\node (CC1) at (6,8) [B] {};

\node (DDt1) at (7,-1) {$\scriptstyle 0$};
\node (DDt) at (7,1) {$\scriptstyle 0$};
\node (DD3) at (7,3) {$\scriptstyle 0$};
\node (DD2) at (7,5) {$\scriptstyle 0$};
\node (DD1) at (7,7) {$\scriptstyle 1$};

\node (EEt) at (8,0)  {$\scriptstyle 0$};
\node (EE4) at (8,2) {$\scriptstyle 0$};
\node (EE3) at (8,4) {$\scriptstyle 0$};
\node (EE2) at (8,6) {$\scriptstyle 0$};
\node (EE1b) at (8,7) {$\scriptstyle 1$};
\node (EE1) at (8,8) {$\scriptstyle 1$};

\node (bt2) at (9,-1) {$\scriptstyle 0$};
\node (bt) at (9,1)  {$\scriptstyle 0$};
\node (b3) at (9,3) {$\scriptstyle 0$};
\node (b2) at (9,5) {$\scriptstyle 0$};
\node (b1) at (9,7) {$\scriptstyle 1$};

\node (ct1) at (10,-1) {$\scriptstyle 0$};
\node (ct) at (10,0) {$\scriptstyle 0$};
\node (c4) at (10,2) {$\scriptstyle 0$};
\node (c3) at (10,4) {$\scriptstyle 0$};
\node (c2) at (10,6) {$\scriptstyle 1$};
\node (c1) at (10,8) {$\scriptstyle 0$};

\node (dt1) at (11,-1) {$\scriptstyle 0$};
\node (dt) at (11,1) {$\scriptstyle 0$};
\node (d3) at (11,3) {$\scriptstyle 0$};
\node (d2) at (11,5) {$\scriptstyle 1$};
\node (d1) at (11,7) {$\scriptstyle 0$};

\node (et) at (12,0) {$\scriptstyle 0$};
\node (e4) at (12,2) {$\scriptstyle 0$};
\node (e3) at (12,4) {$\scriptstyle 1$};
\node (e2) at (12,6) {$\scriptstyle 0$};
\node (e1b) at (12,7) {$\scriptstyle 0$};
\node (e1) at (12,8) {$\scriptstyle 0$};

\node (bbt2) at (13,-1) {$\scriptstyle 0$};
\node (bbt) at (13,1) {$\scriptstyle 0$};
\node (bb3) at (13,3) {$\scriptstyle 1$};
\node (bb2) at (13,5) {$\scriptstyle 0$};
\node (bb1) at (13,7) {$\scriptstyle 0$};

\node (cct1) at (14,-1) {$\scriptstyle 0$};
\node (cct) at (14,0) {$\scriptstyle 0$};
\node (cc4) at (14,2) {$\scriptstyle 1$};
\node (cc3) at (14,4) {$\scriptstyle 0$};
\node (cc2) at (14,6) {$\scriptstyle 0$};
\node (cc1) at (14,8) {$\scriptstyle 0$};

\node (ddt1) at (15,-1) {$\scriptstyle 0$};
\node (ddt) at (15,1) {$\scriptstyle 1$};
\node (dd3) at (15,3) {$\scriptstyle 0$};
\node (dd2) at (15,5) {$\scriptstyle 0$};
\node (dd1) at (15,7) {$\scriptstyle 0$};

\node (eet) at (16,0) {$\scriptstyle 1$};
\node (ee4) at (16,2) {$\scriptstyle 0$};
\node (ee3) at (16,4) {$\scriptstyle 0$};
\node (ee2) at (16,6) {$\scriptstyle 0$};
\node (ee1b) at (16,7) {$\scriptstyle 0$};
\node (ee1) at (16,8) {$\scriptstyle 0$};

\node (bbbt2) at (17,-1) {$\scriptstyle 1$};
\node (bbbt) at (17,1) {$\scriptstyle 0$};
\node (bbb3) at (17,3) {$\scriptstyle 0$};
\node (bbb2) at (17,5) {$\scriptstyle 0$};
\node (bbb1) at (17,7) {$\scriptstyle 0$};

\node (ccct1) at (18,-1) {$\scriptstyle 1$};
\node (ccct) at (18,0) {$\scriptstyle 0$};
\node (ccc4) at (18,2) {$\scriptstyle 0$};
\node (ccc3) at (18,4) {$\scriptstyle 0$};
\node (ccc2) at (18,6) {$\scriptstyle 0$};
\node (ccc1) at (18,8) {$\scriptstyle 0$};

\node (dddt1) at (19,-1) {$\scriptstyle 0$};
\node (dddt) at (19,1) {$\scriptstyle 0$};
\node (ddd3) at (19,3) {$\scriptstyle 0$};
\node (ddd2) at (19,5) {$\scriptstyle 0$};
\node (ddd1) at (19,7) {$\scriptstyle 0$};

\node (eeet) at (20,0) [B] {};
\node (eee4) at (20,2) [B] {};
\node (eee3) at (20,4) [B] {};
\node (eee2) at (20,6) [B] {};
\node (eee1b) at (20,7) [B] {};
\node (eee1) at (20,8) [B] {};

\draw[->] (x1)--(A1);
\draw[->] (x1)--(A1b);
\draw[->] (x1)--(A2);
\draw[->] (x2)--(A2);
\draw[->] (x2)--(A3);
\draw[->] (x3)--(A3);
\draw[->] (x3)--(A4);
\draw[->] (xt)--(A4);
\draw[->] (xt)--(At);
\draw[->] (xt1)--(At);

\draw[->] (A1)--(B1);
\draw[->] (A1b)--(B1);
\draw[->] (A2)--(B1);
\draw[->] (A2)--(B2);
\draw[->] (A3)--(B2);
\draw[->] (A3)--(B3);
\draw[->] (A4)--(Bt);
\draw[->] (A4)--(B3);
\draw[->] (At)--(Bt);
\draw[->] (At)--(Bt2);

\draw[->] (B1)--(C1);
\draw[->] (B1)--(C2);
\draw[->] (B2)--(C2);
\draw[->] (B2)--(C3);
\draw[->] (B3)--(C3);
\draw[->] (B3)--(C4);
\draw[->] (Bt)--(C4);
\draw[->] (Bt)--(Ct);
\draw[->] (Bt2)--(Ct);
\draw[->] (Bt2)--(Ct1);

\draw[->] (C1)--(D1);
\draw[->] (C2)--(D1);
\draw[->] (C2)--(D2);
\draw[->] (C3)--(D2);
\draw[->] (C3)--(D3);
\draw[->] (C4) --(Dt);
\draw[->] (C4)--(D3);
\draw[->] (Ct) --(Dt);
\draw[->] (Ct) --(Dt1);
\draw[->] (Ct1) --(Dt1);

\draw[->] (D1)--(E1);
\draw[->] (D1)--(E1b);
\draw[->] (D1)--(E2);
\draw[->] (D2)--(E2);
\draw[->] (D2)--(E3);
\draw[->] (D3)--(E3);
\draw[->] (D3)--(E4);
\draw[->] (Dt)--(E4);
\draw[->] (Dt)--(Et);
\draw[->] (Dt1)--(Et);

\draw[->] (E1)--(BB1);
\draw[->] (E1b)--(BB1);
\draw[->] (E2)--(BB1);
\draw[->] (E2)--(BB2);
\draw[->] (E3)--(BB2);
\draw[->] (E3)--(BB3);
\draw[->] (E4)--(BBt);
\draw[->] (E4)--(BB3);
\draw[->] (Et)--(BBt);
\draw[->] (Et)--(BBt2);

\draw[->] (BB1)--(CC1);
\draw[->] (BB1)--(CC2);
\draw[->] (BB2)--(CC2);
\draw[->] (BB2)--(CC3);
\draw[->] (BB3)--(CC3);
\draw[->] (BB3)--(CC4);
\draw[->] (BBt)--(CC4);
\draw[->] (BBt)--(CCt);
\draw[->] (BBt2)--(CCt);
\draw[->] (BBt2)--(CCt1);

\draw[->] (CC1)--(DD1);
\draw[->] (CC2)--(DD1);
\draw[->] (CC2)--(DD2);
\draw[->] (CC3)--(DD2);
\draw[->] (CC3)--(DD3);
\draw[->] (CC4) --(DDt);
\draw[->] (CC4)--(DD3);
\draw[->] (CCt) --(DDt);
\draw[->] (CCt) --(DDt1);
\draw[->] (CCt1) --(DDt1);

\draw[->] (DD1)--(EE1);
\draw[->] (DD1)--(EE1b);
\draw[->] (DD1)--(EE2);
\draw[->] (DD2)--(EE2);
\draw[->] (DD2)--(EE3);
\draw[->] (DD3)--(EE3);
\draw[->] (DD3)--(EE4);
\draw[->] (DDt)--(EE4);
\draw[->] (DDt)--(EEt);
\draw[->] (DDt1)--(EEt);

\draw[->] (EE1)--(b1);
\draw[->] (EE1b)--(b1);
\draw[->] (EE2)--(b1);
\draw[->] (EE2)--(b2);
\draw[->] (EE3)--(b2);
\draw[->] (EE3)--(b3);
\draw[->] (EE4)--(bt);
\draw[->] (EE4)--(b3);
\draw[->] (EEt)--(bt);
\draw[->] (EEt)--(bt2);

\draw[->] (b1)--(c1);
\draw[->] (b1)--(c2);
\draw[->] (b2)--(c2);
\draw[->] (b2)--(c3);
\draw[->] (b3)--(c3);
\draw[->] (b3)--(c4);
\draw[->] (bt)--(c4);
\draw[->] (bt)--(ct);
\draw[->] (bt2)--(ct);
\draw[->] (bt2)--(ct1);

\draw[->] (c1)--(d1);
\draw[->] (c2)--(d1);
\draw[->] (c2)--(d2);
\draw[->] (c3)--(d2);
\draw[->] (c3)--(d3);
\draw[->] (c4) --(dt);
\draw[->] (c4)--(d3);
\draw[->] (ct) --(dt);
\draw[->] (ct) --(dt1);
\draw[->] (ct1) --(dt1);

\draw[->] (d1)--(e1);
\draw[->] (d1)--(e1b);
\draw[->] (d1)--(e2);
\draw[->] (d2)--(e2);
\draw[->] (d2)--(e3);
\draw[->] (d3)--(e3);
\draw[->] (d3)--(e4);
\draw[->] (dt)--(e4);
\draw[->] (dt)--(et);
\draw[->] (dt1)--(et);

\draw[->] (e1)--(bb1);
\draw[->] (e1b)--(bb1);
\draw[->] (e2)--(bb1);
\draw[->] (e2)--(bb2);
\draw[->] (e3)--(bb2);
\draw[->] (e3)--(bb3);
\draw[->] (e4)--(bbt);
\draw[->] (e4)--(bb3);
\draw[->] (et)--(bbt);
\draw[->] (et)--(bbt2);

\draw[->] (bb1)--(cc1);
\draw[->] (bb1)--(cc2);
\draw[->] (bb2)--(cc2);
\draw[->] (bb2)--(cc3);
\draw[->] (bb3)--(cc3);
\draw[->] (bb3)--(cc4);
\draw[->] (bbt)--(cc4);
\draw[->] (bbt)--(cct);
\draw[->] (bbt2)--(cct);
\draw[->] (bbt2)--(cct1);

\draw[->] (cc1)--(dd1);
\draw[->] (cc2)--(dd1);
\draw[->] (cc2)--(dd2);
\draw[->] (cc3)--(dd2);
\draw[->] (cc3)--(dd3);
\draw[->] (cc4) --(ddt);
\draw[->] (cc4)--(dd3);
\draw[->] (cct) --(ddt);
\draw[->] (cct) --(ddt1);
\draw[->] (cct1) --(ddt1);

\draw[->] (dd1)--(ee1);
\draw[->] (dd1)--(ee1b);
\draw[->] (dd1)--(ee2);
\draw[->] (dd2)--(ee2);
\draw[->] (dd2)--(ee3);
\draw[->] (dd3)--(ee3);
\draw[->] (dd3)--(ee4);
\draw[->] (ddt)--(ee4);
\draw[->] (ddt)--(eet);
\draw[->] (ddt1)--(eet);

\draw[->] (ee1)--(bbb1);
\draw[->] (ee1b)--(bbb1);
\draw[->] (ee2)--(bbb1);
\draw[->] (ee2)--(bbb2);
\draw[->] (ee3)--(bbb2);
\draw[->] (ee3)--(bbb3);
\draw[->] (ee4)--(bbbt);
\draw[->] (ee4)--(bbb3);
\draw[->] (eet)--(bbbt);
\draw[->] (eet)--(bbbt2);

\draw[->] (bbb1)--(ccc1);
\draw[->] (bbb1)--(ccc2);
\draw[->] (bbb2)--(ccc2);
\draw[->] (bbb2)--(ccc3);
\draw[->] (bbb3)--(ccc3);
\draw[->] (bbb3)--(ccc4);
\draw[->] (bbbt)--(ccc4);
\draw[->] (bbbt)--(ccct);
\draw[->] (bbbt2)--(ccct);
\draw[->] (bbbt2)--(ccct1);

\draw[->] (ccc1)--(ddd1);
\draw[->] (ccc2)--(ddd1);
\draw[->] (ccc2)--(ddd2);
\draw[->] (ccc3)--(ddd2);
\draw[->] (ccc3)--(ddd3);
\draw[->] (ccc4) --(dddt);
\draw[->] (ccc4)--(ddd3);
\draw[->] (ccct) --(dddt);
\draw[->] (ccct) --(dddt1);
\draw[->] (ccct1) --(dddt1);

\draw[->] (ddd1)--(eee1);
\draw[->] (ddd1)--(eee1b);
\draw[->] (ddd1)--(eee2);
\draw[->] (ddd2)--(eee2);
\draw[->] (ddd2)--(eee3);
\draw[->] (ddd3)--(eee3);
\draw[->] (ddd3)--(eee4);
\draw[->] (dddt)--(eee4);
\draw[->] (dddt)--(eeet);
\draw[->] (dddt1)--(eeet);

\draw[densely dotted] (0,-1.25) -- (0,7.75);
\draw[densely dotted] (4,-1.25) -- (4,7.75);
\draw[densely dotted] (8,-1.25) -- (8,7.75);
\draw[densely dotted] (12,-1.25) -- (12,7.75);
\draw[densely dotted] (16,-1.25) -- (16,7.75);
\draw[densely dotted] (20,-1.25) -- (20,7.75);
\end{tikzpicture}
\end{array}
&
=
&
\begin{array}{c}
\begin{tikzpicture}[yscale=0.6,xscale=0.6]

\node (At) at (0,0) {$\scriptstyle 2$};
\node (A4) at (0,2)  {$\scriptstyle 0$};
\node (A3) at (0,4) {$\scriptstyle 2$};
\node (A2) at (0,6) {$\scriptstyle 0$};
\node (A1b) at (0,7) {$\scriptstyle 1$};
\node (A1) at (0,8) {$\scriptstyle 1$};

\node (Bt2) at (1,-1)  {$\scriptstyle 1$};
\node (Bt) at (1,1) {$\scriptstyle 1$};
\node (B3) at (1,3) {$\scriptstyle 1$};
\node (B2) at (1,5) {$\scriptstyle 1$};
\node (B1) at (1,7) {$\scriptstyle 1$};

\node (Ct1) at (2,-1)  {$\scriptstyle 1$};
\node (Ct) at (2,0) {$\scriptstyle 0$};
\draw (Ct) circle (10pt);
\node (C4) at (2,2) {$\scriptstyle 2$};
\node (C3) at (2,4) {$\scriptstyle 0$};
\draw (C3) circle (10pt);
\node (C2) at (2,6) {$\scriptstyle 2$};
\node (C1) at (2,8) {$\scriptstyle 0$};
\draw (C1) circle (10pt);

\node (Dt1) at (3,-1) {$\scriptstyle 1$};
\node (Dt) at (3,1) {$\scriptstyle 1$};
\node (D3) at (3,3)  {$\scriptstyle 1$};
\node (D2) at (3,5) {$\scriptstyle 1$};
\node (D1) at (3,7) {$\scriptstyle 1$};

\node (Et) at (4,0) {$\scriptstyle 2$};
\node (E4) at (4,2) {$\scriptstyle 0$};
\draw (E4) circle (10pt);
\node (E3) at (4,4) {$\scriptstyle 2$};
\node (E2) at (4,6) {$\scriptstyle 0$};
\draw (E2) circle (10pt);
\node (E1b) at (4,7) {$\scriptstyle 1$};
\node (E1) at (4,8) {$\scriptstyle 1$};

\draw[->] (A1)--(B1);
\draw[->] (A1b)--(B1);
\draw[->] (A2)--(B1);
\draw[->] (A2)--(B2);
\draw[->] (A3)--(B2);
\draw[->] (A3)--(B3);
\draw[->] (A4)--(Bt);
\draw[->] (A4)--(B3);
\draw[->] (At)--(Bt);
\draw[->] (At)--(Bt2);

\draw[->] (B1)--(C1);
\draw[->] (B1)--(C2);
\draw[->] (B2)--(C2);
\draw[->] (B2)--(C3);
\draw[->] (B3)--(C3);
\draw[->] (B3)--(C4);
\draw[->] (Bt)--(C4);
\draw[->] (Bt)--(Ct);
\draw[->] (Bt2)--(Ct);
\draw[->] (Bt2)--(Ct1);

\draw[->] (C1)--(D1);
\draw[->] (C2)--(D1);
\draw[->] (C2)--(D2);
\draw[->] (C3)--(D2);
\draw[->] (C3)--(D3);
\draw[->] (C4) --(Dt);
\draw[->] (C4)--(D3);
\draw[->] (Ct) --(Dt);
\draw[->] (Ct) --(Dt1);
\draw[->] (Ct1) --(Dt1);

\draw[->] (D1)--(E1);
\draw[->] (D1)--(E1b);
\draw[->] (D1)--(E2);
\draw[->] (D2)--(E2);
\draw[->] (D2)--(E3);
\draw[->] (D3)--(E3);
\draw[->] (D3)--(E4);
\draw[->] (Dt)--(E4);
\draw[->] (Dt)--(Et);
\draw[->] (Dt1)--(Et);

\draw[densely dotted] (0,-1.25) -- (0,7.75);
\draw[densely dotted] (4,-1.25) -- (4,7.75);
\end{tikzpicture}
\end{array}
\end{array}
\]
In the right hand side of the above, we have circled the modules $x$ for which $\Hom(M_1,x)=0$.  By inspection of the dimension vectors in the AR quiver in Figure~\ref{AR figure}, the circled modules are precisely those modules filtered by $\scrS_2$.

\medskip
In the case that $a=\scrS_1^{\oplus s}\oplus M_1^{\oplus t}\oplus N^{\oplus u}$, it follows from the proceeding paragraph that if either $u\geq 1$ or $t\geq 1$, then $b$ is filtered by $\scrS_2$.  Hence we can assume that $u=t=0$, in which case $a=\scrS_1^{\oplus s}$, and so $a$ is filtered by $\scrS_1$.

\medskip
In the case that $b=\scrS_1^{\oplus s}\oplus M_2^{\oplus t}\oplus N^{\oplus u}$, in a very similar way, it follows from the proceeding paragraph that if either $u\geq 1$ or $t\geq 1$, then $a$ is filtered by $\scrS_2$.  Hence we can assume that $u=t=0$, in which case $b=\scrS_1^{\oplus s}$, and so $b$ is filtered by $\scrS_1$.
\end{proof}

\begin{proof}[Proof of \ref{only four bricks}]
Running the knitting algorithm as above repeatedly, it is easy (but long) to calculate that the dimension of the self-Hom of each module is as follows:
\[
\begin{array}{ccc}
\begin{array}{c}
\begin{tikzpicture}[scale=0.7]
\node (At) at (0,1)  {$\scriptstyle k+1$};
\node (A4) at (0,2)  {$\scriptstyle k$};
\node (A3) at (0,3)  {$\scriptstyle 3$};
\node (A2) at (0,5)   {$\scriptstyle 2$};
\node (A1) at (0,7)   {$\scriptstyle 1$};

\node (Bt) at (1,1)  {$\scriptstyle k$};
\node (B2) at (1,4) {$\scriptstyle 2$};
\node (B1) at (1,6)  {$\scriptstyle 1$};
\draw (B1) circle (10pt);

\node (Ct) at (2,0) {$\scriptstyle k+1$};
\node (C4) at (2,2) {$\scriptstyle k$};
\node (C3) at (2,3) {$\scriptstyle 3$};
\node (C2) at (2,5) {$\scriptstyle 2$};
\node (C1) at (2,7) {$\scriptstyle 1$};
\draw (C1) circle (10pt);

\node (Dt) at (3,1) {$\scriptstyle k$};
\node (D2) at (3,4) {$\scriptstyle 2$};
\node (D1) at (3,6) { $\scriptstyle 1$};
\draw (D1) circle (10pt);

\node (Et) at (4,1)  {$\scriptstyle k+1$};
\node (E4) at (4,2) {$\scriptstyle k$};
\node (E3) at (4,3) {$\scriptstyle 3$};
\node (E2) at (4,5) {$\scriptstyle 2$};
\node (E1) at (4,7) {$\scriptstyle 1$};
\draw (E1) circle (10pt);
 
\node at (1,2.5) {$\vdots$};
\node at (3,2.5) {$\vdots$};

\draw[->] (A1)--(B1);
\draw[->] (A2)--(B1);
\draw[->] (A2)--(B2);
\draw[->] (A3)--(B2);
\draw[->] (A4)--(Bt);
\draw[->] (At)--(Bt);

\draw[->] (B1)--(C1);
\draw[->] (B1)--(C2);
\draw[->] (B2)--(C2);
\draw[->] (B2)--(C3);
\draw[->] (Bt)--(C4);
\draw[->] (Bt)--(Ct);

\draw[->] (C1)--(D1);
\draw[->] (C2)--(D1);
\draw[->] (C2)--(D2);
\draw[->] (C3)--(D2);
\draw[->] (C4) --(Dt);
\draw[->] (Ct) --(Dt);

\draw[->] (D1)--(E1);
\draw[->] (D1)--(E2);
\draw[->] (D2)--(E2);
\draw[->] (D2)--(E3);
\draw[->] (Dt)--(E4);
\draw[->] (Dt)--(Et);

\draw[densely dotted] (0,-0.25) -- (0,7.5);
\draw[densely dotted] (4,-0.25) -- (4,7.5);
\end{tikzpicture}
\end{array}
&
\begin{array}{c}
\begin{tikzpicture}[scale=0.7]
\node (At) at (0,0) {$\scriptstyle 3$};
\node (A4) at (0,2) {$\scriptstyle $};
\node (A3) at (0,3) {$\scriptstyle $};
\node (A2) at (0,5) {$\scriptstyle k-1$};
\node (A1b) at (0,6) {$\scriptstyle k+1$};
\node (A1) at (0,7) {$\scriptstyle 1$};

\node (Bt2) at (1,-1) {$\scriptstyle 1$};
\draw (Bt2) circle (10pt);
\node (Bt) at (1,1) {$\scriptstyle 2$};
\node (B2) at (1,4) {$\scriptstyle k-1$};
\node (B1) at (1,6) {$\scriptstyle k$};

\node (Ct1) at (2,-1) {$\scriptstyle 2$};
\node (Ct) at (2,0) {$\scriptstyle 1$};
\draw (Ct) circle (10pt);
\node (C4) at (2,2) {$\scriptstyle $};
\node (C3) at (2,3) {$\scriptstyle $};
\node (C2) at (2,5) {$\scriptstyle k+1$};
\node (C1) at (2,7) {$\scriptstyle k$};

\node (Dt1) at (3,-1) {$\scriptstyle 1$};
\draw (Dt1) circle (10pt);
\node (Dt) at (3,1) {$\scriptstyle 2$};
\node (D2) at (3,4) {$\scriptstyle k-1$};
\node (D1) at (3,6) {$\scriptstyle k$};

\node (Et) at (4,0) {$\scriptstyle 3$};
\node (E4) at (4,2) {$\scriptstyle $};
\node (E3) at (4,3) {$\scriptstyle $};
\node (E2) at (4,5) {$\scriptstyle k-1$};
\node (E1b) at (4,6) {$\scriptstyle k+1$};
\node (E1) at (4,7) {$\scriptstyle 1$};
\draw (E1) circle (10pt);

\node at (1,2.5) {$\vdots$};
\node at (3,2.5) {$\vdots$};

\draw[->] (A1)--(B1);
\draw[->] (A1b)--(B1);
\draw[->] (A2)--(B1);
\draw[->] (A2)--(B2);
\draw[->] (A3)--(B2);
\draw[->] (A4)--(Bt);
\draw[->] (At)--(Bt);
\draw[->] (At)--(Bt2);

\draw[->] (B1)--(C1);
\draw[->] (B1)--(C2);
\draw[->] (B2)--(C2);
\draw[->] (B2)--(C3);
\draw[->] (Bt)--(C4);
\draw[->] (Bt)--(Ct);
\draw[->] (Bt2)--(Ct);
\draw[->] (Bt2)--(Ct1);

\draw[->] (C1)--(D1);
\draw[->] (C2)--(D1);
\draw[->] (C2)--(D2);
\draw[->] (C3)--(D2);
\draw[->] (C4) --(Dt);
\draw[->] (Ct) --(Dt);
\draw[->] (Ct) --(Dt1);
\draw[->] (Ct1) --(Dt1);

\draw[->] (D1)--(E1);
\draw[->] (D1)--(E1b);
\draw[->] (D1)--(E2);
\draw[->] (D2)--(E2);
\draw[->] (D2)--(E3);
\draw[->] (Dt)--(E4);
\draw[->] (Dt)--(Et);
\draw[->] (Dt1)--(Et);

\draw[densely dotted] (0,-1.25) -- (0,7.25);
\draw[densely dotted] (4,-1.25) -- (4,7.25);
\end{tikzpicture}
\end{array}
&
\begin{array}{c}
\begin{tikzpicture}[scale=0.65]
\node (At1) at (0,-1)  {$\scriptstyle 2$};
\node (At) at (0,0)  {$\scriptstyle 1$};
\node (A4) at (0,2)  {$\scriptstyle $};
\node (A3) at (0,3) {$\scriptstyle $};
\node (A2) at (0,5)  {$\scriptstyle k-1$};
\node (A1b) at (0,6)  {$\scriptstyle k+1$};
\node (A1) at (0,7)  {$\scriptstyle 1$};

\node (Bt2) at (1,-1) {$\scriptstyle 1$};
\draw (Bt2) circle (10pt);
\node (Bt) at (1,1)  {$\scriptstyle 2$};
\node (B2) at (1,4) {$\scriptstyle k-1$};
\node (B1) at (1,6)  {$\scriptstyle k$};

\node (Ct) at (2,0) {$\scriptstyle 3$};
\node (C4) at (2,2) {$\scriptstyle $};
\node (C3) at (2,3) {$\scriptstyle $};
\node (C2) at (2,5) {$\scriptstyle k+1$};
\node (C1) at (2,7) {$\scriptstyle k$};

\node (Dt1) at (3,-1) {$\scriptstyle 1$};
\draw (Dt1) circle (10pt);
\node (Dt) at (3,1) {$\scriptstyle 2$};
\node (D2) at (3,4) {$\scriptstyle k-1$};
\node (D1) at (3,6) {$\scriptstyle k$};

\node (Et1) at (4,-1) {$\scriptstyle 2$};
\node (Et) at (4,0) {$\scriptstyle 1$};
\draw (Et) circle (10pt);
\node (E4) at (4,2) {$\scriptstyle $};
\node (E3) at (4,3) {$\scriptstyle $};
\node (E2) at (4,5) {$\scriptstyle k-1$};
\node (E1b) at (4,6) {$\scriptstyle k+1$};
\node (E1) at (4,7) {$\scriptstyle 1$};
\draw (E1) circle (10pt);

\node at (1,2.5) {$\vdots$};
\node at (3,2.5) {$\vdots$};

\draw[->] (A1)--(B1);
\draw[->] (A1b)--(B1);
\draw[->] (A2)--(B1);
\draw[->] (A2)--(B2);
\draw[->] (A3)--(B2);
\draw[->] (A4)--(Bt);
\draw[->] (At)--(Bt);
\draw[->] (At)--(Bt2);
\draw[->] (At1)--(Bt2);

\draw[->] (B1)--(C1);
\draw[->] (B1)--(C2);
\draw[->] (B2)--(C2);
\draw[->] (B2)--(C3);
\draw[->] (Bt)--(C4);
\draw[->] (Bt)--(Ct);
\draw[->] (Bt2)--(Ct);

\draw[->] (C1)--(D1);
\draw[->] (C2)--(D1);
\draw[->] (C2)--(D2);
\draw[->] (C3)--(D2);
\draw[->] (C4) --(Dt);
\draw[->] (Ct) --(Dt);
\draw[->] (Ct) --(Dt1);

\draw[->] (D1)--(E1);
\draw[->] (D1)--(E1b);
\draw[->] (D1)--(E2);
\draw[->] (D2)--(E2);
\draw[->] (D2)--(E3);
\draw[->] (Dt)--(E4);
\draw[->] (Dt)--(Et);
\draw[->] (Dt1)--(Et);
\draw[->] (Dt1)--(Et1);

\draw[densely dotted] (0,-1.25) -- (0,7.25);
\draw[densely dotted] (4,-1.25) -- (4,7.25);
\end{tikzpicture}
\end{array}
\end{array}
\]
In all cases, there are precisely four modules with one-dimensional endomorphism ring, which we have circled.  By inspection of the AR quivers in Lemma~\ref{AR quivers with dim vectors}, the circled modules are either simple, or have dimension vector $(1,1)$.  By uniqueness of the non-split extensions between $\scrS_1$ and $\scrS_2$, the result follows.
\end{proof}

\section{A realisation question}\label{realisation section}

One can hope that Theorem~\ref{thm:main} sits in a much wider framework; we briefly sketch this here, and highlight  some obstructions and subtleties.  

\medskip
As demonstrated by Theorem~\ref{thm:main}, it may simply not be feasible to expect mirror symmetry statements for \emph{all} flops without changing the characteristic of the field.  In addition to this issue, 3-fold flops have moduli, and on the A-side it is not so clear how continuous parameters can enter into the plumbing of spheres, even if we iterate.  As such, following a philosophy that the mirror to the exact form should be some point in complex structure moduli space which is picked out arithmetically, we propose that the existence of the flop equation \eqref{flop equation} defined over $\mathbb{Z}$, not just $\mathbb{C}$, is one of the key features of this paper.

\medskip
Much more generally, we pose the following `realisation question', which is open on both the algebraic and symplectic sides.  In both parts we abuse notation, as above Corollary~\ref{A side char p main}, when describing $A_\infty$-structures over fields possibly of prime characteristic.

\begin{Question}\label{Realisation problem}
For the quivers $Q$ in Figure~\ref{fig.quivers}, and for those potentials $f$ defined over $\mathbb{Z}$ for which the Jacobi algebra $J_f$ is finite dimensional (over some field $\bK$), does there exist:
\vspace{-\parskip}
\begin{enumerate}
\item\label{Realisation problem 1} A $3$-fold flop $g\colon Y\to \Spec R$ over $\bK$, where $Y$ is smooth and defined over $\Z$, such that its contraction algebra is isomorphic to $J_f$, and the $A_{\infty}$-structure on the Ext-algebra of the curves is also described by $f$?
\item\label{Realisation problem 2} An exact symplectic manifold $W_f$ with $n=|Q_0|$ distinguished Lagrangians $L_1,\hdots, L_n$, over which each $L_i$ is  fat-spherical and the  $A_\infty$-algebra $\bigoplus_{i,j} HF^*(L_i,L_j)$ is described by $f$?
\end{enumerate}
\vspace{-\parskip}
\end{Question}
\begin{figure}[h]
\[
\begin{array}{lcl}
A_n\quad n\geq 1
&
\begin{array}{c}
\begin{tikzpicture}[scale=1.25,bend angle=15, looseness=1]
\node (b) at (0,0) [pvertex] {$1$};
\node at (0,0.6) [pvertex] {};
\node (c) at (1,0) [pvertex] {$\phantom{2}$};
\node at (1,0) {$\scriptstyle \hdots$};
\node (n) at (2,0) [pvertex] {$n$};
\draw[->,bend left] (c) to (b);
\draw[->,bend left] (b) to (c);
\draw[->,bend left] (n) to (c);
\draw[->,bend left] (c) to (n);
\draw[<-,densely dotted]  (b) edge [in=-120,out=-65,loop,looseness=7]  (b);
\draw[<-,densely dotted]  (c) edge [in=-120,out=-65,loop,looseness=7]  (c);
\draw[<-,densely dotted]  (n) edge [in=-120,out=-65,loop,looseness=7]  (n);
\end{tikzpicture}
\end{array}
&
\\
\bar{D}_n\quad n\geq 1
&
\begin{array}{c}
\begin{tikzpicture}[scale=1.25,bend angle=15, looseness=1]
\node (b) at (0,0) [pvertex] {$1$};
\node at (0,0.6) [pvertex] {};
\node (c) at (1,0) [pvertex] {$\phantom{2}$};
\node at (1,0) {$\scriptstyle \hdots$};
\node (n) at (2,0) [pvertex] {$n$};
\draw[->,bend left] (c) to (b);
\draw[->,bend left] (b) to (c);
\draw[->,bend left] (n) to (c);
\draw[->,bend left] (c) to (n);
\draw[<-]  (b) edge [in=-120,out=-65,loop,looseness=7]  (b);
\draw[<-]  (b) edge [in=120,out=65,loop,looseness=7]  (b);
\draw[<-,densely dotted]  (c) edge [in=-120,out=-65,loop,looseness=7]  (c);
\draw[<-,densely dotted]  (n) edge [in=-120,out=-65,loop,looseness=7]  (n);
\end{tikzpicture}
\end{array}
&
\\
\bar{E}_n\quad n=3,4,5
&
\begin{array}{c}
\begin{tikzpicture}[scale=1.25,bend angle=15, looseness=1]
\node at (0,0.6) [pvertex] {};
\node (1) at (-1,0) [pvertex] {$1$};
\node (2) at (0,0) [pvertex] {$2$};
\node (3) at (1,0) [pvertex] {$3$};
\node (dots) at (2,0) [pvertex] {$\phantom{2}$};
\node at (2,0) {$\scriptstyle \hdots$};
\node (n) at (3,0) [pvertex] {$n$};
\draw[->,bend left] (2) to (1);
\draw[->,bend left] (1) to (2);
\draw[->,bend left] (3) to (2);
\draw[->,bend left] (2) to (3);
\draw[->,bend left] (dots) to (3);
\draw[->,bend left] (3) to (dots);
\draw[->,bend left] (n) to (dots);
\draw[->,bend left] (dots) to (n);
\draw[white]  (1) edge [in=-120,out=-65,loop,looseness=7]  (1);
\draw[<-]  (2) edge [in=-120,out=-65,loop,looseness=7]  (2);
\draw[<-]  (2) edge [in=120,out=65,loop,looseness=7]  (2);
\draw[<-,densely dotted]  (3) edge [in=-120,out=-65,loop,looseness=7]  (3);
\draw[<-,densely dotted]  (dots) edge [in=-120,out=-65,loop,looseness=7]  (dots);
\draw[<-,densely dotted]  (n) edge [in=-120,out=-65,loop,looseness=7]  (n);
\end{tikzpicture}
\end{array}
&
\\
D_n\quad n\geq 4
&
\begin{array}{c}
\begin{tikzpicture}[scale=1.25,bend angle=15, looseness=1]
\node (1) at (-1,0) [pvertex] {$2$};
\node (2) at (0,0) [pvertex] {$3$};
\node (3) at (0,1)  [pvertex] {$1$};
\node (4) at (1,0)  [pvertex] {$4$};
\node (dots) at (2,0) [pvertex] {$\phantom{2}$};
\node at (2,0) {$\scriptstyle \hdots$};
\node (n) at (3,0) [pvertex] {$n$};
\draw[->,bend left] (2) to (1);
\draw[->,bend left] (1) to (2);
\draw[->,bend left] (4) to (2);
\draw[->,bend left] (2) to (4);
\draw[->,bend left] (2) to (3);
\draw[->,bend left] (3) to (2);
\draw[->,bend left] (dots) to (4);
\draw[->,bend left] (4) to (dots);
\draw[->,bend left] (n) to (dots);
\draw[->,bend left] (dots) to (n);
\draw[<-,densely dotted]  (1) edge [in=-120,out=-65,loop,looseness=7]  (1);
\draw[<-,densely dotted]  (2) edge [in=-120,out=-65,loop,looseness=7]  (2);
\draw[<-,densely dotted]  (3) edge [in=120,out=65,loop,looseness=7]  (3);
\draw[<-,densely dotted]  (4) edge [in=-120,out=-65,loop,looseness=7]  (4);
\draw[<-,densely dotted]  (dots) edge [in=-120,out=-65,loop,looseness=7]  (dots);
\draw[<-,densely dotted]  (n) edge [in=-120,out=-65,loop,looseness=7]  (n);
\end{tikzpicture}
\end{array}
&
\\
\widetilde{D}_n\quad n\geq 3 
&
\begin{array}{c}
\begin{tikzpicture}[scale=1.25,bend angle=15, looseness=1]
\node (1) at (-1,0) [pvertex] {$2$};
\node (2) at (0,0) [pvertex] {$3$};
\node (3) at (0,1)  [pvertex] {$1$};
\node (dots) at (1,0) [pvertex] {$\phantom{2}$};
\node at (1,0) {$\scriptstyle \hdots$};
\node (n) at (2,0) [pvertex] {$n$};
\draw[->,bend left,looseness=0.5] (1) to (3);
\draw[->,bend left,looseness=0.5] (3) to (1);
\draw[->,bend left] (2) to (1);
\draw[->,bend left] (1) to (2);
\draw[->,bend left] (dots) to (2);
\draw[->,bend left] (2) to (dots);
\draw[->,bend left] (2) to (3);
\draw[->,bend left] (3) to (2);
\draw[->,bend left] (n) to (dots);
\draw[->,bend left] (dots) to (n);
\draw[<-,densely dotted]  (1) edge [in=-120,out=-65,loop,looseness=7]  (1);
\draw[<-,densely dotted]  (2) edge [in=-120,out=-65,loop,looseness=7]  (2);
\draw[<-,densely dotted]  (dots) edge [in=-120,out=-65,loop,looseness=7]  (dots);
\draw[<-,densely dotted]  (n) edge [in=-120,out=-65,loop,looseness=7]  (n);
\end{tikzpicture}
\end{array}
&
\scriptsize
\begin{array}{p{4cm}}
But if $2$ has a loop, then
either $n\leq 4$, or\\ $n=5$ and $3$ has no loop, or\\
 $n=5$ and $4$ has no loop, or\\ $n=6$ and $3$ has no loop.
\end{array}\\
E_n\quad n=6,7,8
&
\begin{array}{c}
\begin{tikzpicture}[scale=1.1,bend angle=15, looseness=1]
\node (0) at (-2,0) [pvertex] {$2$};
\node (1) at (-1,0) [pvertex] {$3$};
\node (2) at (0,0) [pvertex] {$4$};
\node (3) at (0,1)  [pvertex] {$1$};
\node (4) at (1,0)  [pvertex] {$5$};
\node (5) at (2,0)  [pvertex] {$6$};

\node (dots) at (3,0) [pvertex] {$\phantom{2}$};
\node at (3,0) {$\scriptstyle \hdots$};
\node (n) at (4,0) [pvertex] {$n$};
\draw[->,bend left] (1) to (0);
\draw[->,bend left] (0) to (1);
\draw[->,bend left] (2) to (1);
\draw[->,bend left] (1) to (2);
\draw[->,bend left] (4) to (2);
\draw[->,bend left] (2) to (4);
\draw[->,bend left] (2) to (3);
\draw[->,bend left] (3) to (2);
\draw[->,bend left] (5) to (4);
\draw[->,bend left] (4) to (5);
\draw[->,bend left] (dots) to (5);
\draw[->,bend left] (5) to (dots);
\draw[->,bend left] (n) to (dots);
\draw[->,bend left] (dots) to (n);
\draw[<-,densely dotted]  (0) edge [in=-120,out=-65,loop,looseness=7]  (0);
\draw[<-,densely dotted]  (1) edge [in=-120,out=-65,loop,looseness=7]  (1);
\draw[<-,densely dotted]  (2) edge [in=-120,out=-65,loop,looseness=7]  (2);
\draw[<-,densely dotted]  (3) edge [in=120,out=65,loop,looseness=7]  (3);
\draw[<-,densely dotted]  (4) edge [in=-120,out=-65,loop,looseness=7]  (4);
\draw[<-,densely dotted]  (5) edge [in=-120,out=-65,loop,looseness=7]  (5);
\draw[<-,densely dotted]  (n) edge [in=-120,out=-65,loop,looseness=7]  (n);
\end{tikzpicture}
\end{array}
& 
\\
\widetilde{E}_n\quad n=5,6,7 
&
\begin{array}{c}
\begin{tikzpicture}[scale=1.25,bend angle=15, looseness=1]
\node (0) at (-2,0) [pvertex] {$2$};
\node (1) at (-1,0) [pvertex] {$3$};
\node (2) at (0,0) [pvertex] {$4$};
\node (3) at (-0.5,1)  [pvertex] {$1$};
\node (4) at (1,0) [pvertex] {$5$};
\node (dots) at (2,0) [pvertex] {$\phantom{2}$};
\node at (2,0) {$\scriptstyle \hdots$};
\node (n) at (3,0) [pvertex] {$n$};
\draw[->,bend left,looseness=0.5] (1) to (3);
\draw[->,bend left,looseness=0.5] (3) to (1);
\draw[->,bend left] (1) to (0);
\draw[->,bend left] (0) to (1);
\draw[->,bend left] (2) to (1);
\draw[->,bend left] (1) to (2);
\draw[->,bend left] (4) to (2);
\draw[->,bend left] (2) to (4);
\draw[->,bend left] (dots) to (4);
\draw[->,bend left] (4) to (dots);
\draw[->,bend left,looseness=0.5] (2) to (3);
\draw[->,bend left,looseness=0.5] (3) to (2);
\draw[->,bend left] (n) to (dots);
\draw[->,bend left] (dots) to (n);
\draw[<-,densely dotted]  (0) edge [in=-120,out=-65,loop,looseness=7]  (0);
\draw[<-,densely dotted]  (2) edge [in=-120,out=-65,loop,looseness=7]  (2);
\draw[<-,densely dotted]  (4) edge [in=-120,out=-65,loop,looseness=7]  (4);
\draw[<-,densely dotted]  (dots) edge [in=-120,out=-65,loop,looseness=7]  (dots);
\draw[<-,densely dotted]  (n) edge [in=-120,out=-65,loop,looseness=7]  (n);
\end{tikzpicture}
\end{array}
&
\scriptsize
\begin{array}{p{4cm}}
Moreover, if $n=6$ and $4$ has a loop, then $5$ does not have a loop, and if $n=7$ then $4$ does not have a loop.
\end{array}
\end{array}
\]
\caption{Possible quivers of contraction algebras of smooth $3$-fold flops, as a consequence of \cite[5.5]{Morrison}.  There is choice, as a single loop may or may not occur wherever a dotted arrow is present.}
\label{fig.quivers}
\end{figure}

An obvious challenge is to find some symplectic-geometric mechanism which distinguishes this particular collection of quivers with potential. The papers \cite{BW, Davison, Hao} give strong evidence towards Question~\ref{Realisation problem}\eqref{Realisation problem 1} on the B-side, particularly \cite[1.11]{BW}. 

\medskip
For any given potential $f$, whenever both parts of \ref{Realisation problem} exist, over the appropriate field $\bK$ the top horizontal row in the following is an equivalence of categories, and it should extend via Koszul duality to an equivalence in the dotted bottom horizontal row, modulo Remark~\ref{A-side Kosz issue}.
\[
\begin{tikzpicture}
\node (A1) at (0,0) {$\scrC\colonequals\{ a\mid \mathbf{R}g_*a=0$\}};
\node (A2) at (5,0) {$\scrL\colonequals\langle L_1,\hdots,L_n\rangle$};
\node (B1) at (0,-1.5) {$\Db(\coh Y)/\langle\scrO_Y\rangle$};
\node (B2) at (5,-1.5) {$\scrW(W_f;\bK) $};
\draw[<->] (A1)--node[above]{$\scriptstyle \sim$}(A2);
\draw[<->,densely dotted] (B1)--node[above]{$\scriptstyle \sim$}(B2);
\draw[right hook->] (A1) -- (B1);
\draw[right hook->] (A2) -- (B2);
\end{tikzpicture}
\]
Furthermore, on the B-side, there is a known real simplicial hyperplane arrangement $\scrH$ which induces a faithful group action
\[
\fundgp(\mathbb{C}^n\backslash \cH_\mathbb{C})\hookrightarrow\Auteq \scrL.
\]
When $\scrH$ is an ADE root system,  $\fundgp(\mathbb{C}^n\backslash \cH_\mathbb{C})$ is isomorphic to the pure braid group, but in general $\scrH$ need not be Coxeter, and further higher length braid relations exist.  

\medskip
Question~\ref{Realisation problem} is very subtle, even for $A_2$ and $A_3$.  Most notably, on the B-side the dual graph changes under flop, and so the families in Figure~\ref{fig.quivers} can be related via mutation. More plainly, adding loops  to the oriented two-cycle results in \emph{much} more complicated geometry than considered in this paper.

\bibliographystyle{plain}
\bibliography{bib}

\end{document}